\documentclass[10pt]{article}
\usepackage{pb-diagram}
\usepackage{amsmath}
\usepackage{amsthm}
\usepackage{mathtools}
\usepackage{amsfonts}
\usepackage{amssymb}
\usepackage{graphicx}
\usepackage[usenames]{color}
\usepackage{xcolor}
\usepackage{hyperref}
\usepackage{soul}
\usepackage{imakeidx}

\newtheorem{theorem}{Theorem}
\newtheorem{definition}{Definition}

\newtheorem{lemma}{Lemma}
\newtheorem{cor}{Corollary}
\newtheorem{remark}{Remark}
\newtheorem{conjecture}{Conjecture}
\newtheorem{prop}{Proposition}
\newtheorem{problem}{Problem}
\theoremstyle{remark}
\newtheorem{example}{\bf Example}
\newcommand{\be}{\begin{enumerate}}
\newcommand{\ee}{\end{enumerate}}
\newcommand{\beq}{\begin{equation}}
\newcommand{\eeq}{\end{equation}}
\newcommand{\OR}{\mathcal{O}}

\newcommand{\UTR}{\textrm{UT}_n(R)}
\newcommand{\TO}{\textrm{T}_n(\OR)}

\newcommand{\TR}{\textrm{T}_n(R)}
\newcommand{\TRf}{\textrm{T}_n(R,f)}
\newcommand{\Th}{\textrm{Th}}
\newcommand{\SLZ}{\textrm{SL}_n(\Z)}

\newcommand{\SLR}{\textrm{SL}_n(R)}

\newcommand{\SLO}{\textrm{SL}_n(\mathcal{O})}
\newcommand{\GLO}{\textrm{GL}_n(\mathcal{O})}
\newcommand{\GLR}{\textrm{GL}_n(R)}

\newcommand{\RM}{R^{\times}}
\def\N{{\mathbb{N}}}
\def\Z{{\mathbb{Z}}}
\def\R{{\mathbb{R}}}

\def\Q{{\mathbb{Q}}}

\def\MN{{\mathbb{N}}}
\def\MZ{{\mathbb{Z}}}
\def\MQ{{\mathbb{Q}}}

\def\MA{{\mathbb{A}}}
\def\MB{{\mathbb{B}}}
\def\MC{{\mathbb{C}}}

\newcommand{\OM}{\mathcal{O}^{\times}}

\def\A{{\mathcal{A}}}

\def\S{{\mathcal{S}}}

\def\F{{\mathcal{F}}}

\def\barA{\hat{A}}
\def\barB{\hat{B}}
\def\bara{\hat{a}}

\def\barx{\hat{x}}
\def\bary{\hat{y}}
\newcommand{\one}{{\bf 1}}
{}
{}
{}
\newcommand{\atreyer}[1]{{\bf \color{orange} [AT: #1]}\color{black}}

\title{Rich groups, weak second order logic, and applications}
\author{Olga Kharlampovich\footnote{Hunter College, CUNY; Supported by the grant 422503 from the Simons Foundation.} , Alexei Myasnikov, Mahmood Sohrabi \footnote{Stevens Institute of Technology.}} 

\date{}

\makeindex
\begin{document}

\maketitle

\begin{abstract} 
In this paper we initiate a study  of first-order rich groups, i.e., groups where the first-order logic has the same power as the weak second order logic. Surprisingly, there are quite a lot of finitely generated rich groups, they  are somewhere in between hyperbolic and nilpotent groups (these ones are not rich).  We provide some methods to prove that groups (and other structures) are rich and describe some of their properties. As corollaries we look at  Malcev's problems in various groups.

\end{abstract}

\tableofcontents

\section{Introduction} 
 The study of efficiency of the first-order logic goes back to Hilbert and Godel, and later, where algebra concerns, to Tarski and Malcev. For a group (or a ring,  or any structure) $A$  the first-order theory $Th(A)$ is the set of all sentences in group (ring, etc.) language that are true in $A$, so $Th(A)$ is all the possible information about $A$ expressible in the first-order logic of group theory. Two groups (rings, etc.) $A$ and $B$ are called elementarily equivalent ($A \equiv B$) if $Th(A) = Th(B)$, i.e., they are indistinguishable in the first-order logic. Tarski and Malcev pushed forward a problem of describing groups and  rings (in some natural classes)  that are elementarily equivalent. If $A$ and $B$ are isomorphic  then, obviously, $A \equiv B$. The most interesting question here is if $A \equiv B$ then how close to being isomorphic $A$ and $B$ could be, i.e., how good the description by $Th(A)$ of $A$ is? It was soon understood (L\"owenheim-Skolem theorem) that for any infinite structure $A$ there is a structure $B$ such that $A \equiv B$ and they have different cardinalities, so in the question above one may want to assume that $A$ and $B$ have the same cardinality. This leads to the notion of categoricity, and stability, and some other important developments in modern model theory.  But categoricity is a rather rare phenomenon. The following seems to be the general picture. Given a finitely generated group (or a ring or a structure) $A$ there is in general little hope to characterize all arbitrary countable groups $B$ with $A \equiv B$. But, rather often  there is a decent chance to describe such $B$ if it is assumed to be finitely generated. Algebraically, it makes sense to consider finitely generated objects, rather then countable ones. In the opposite direction (towards uncountable objects)  one can start with a group or 
 a ring $A$ and ``complete'' it by adding to $A$ solutions to various classes of equations until it is algebraically complete (with respect to some theory).
 The typical picture here is arithmetic, i.e., the ring of integers $\mathbb{Z}$: for every finitely generated ring $B$ if $\mathbb{Z} \equiv B$ then $\mathbb{Z} \cong B$; there are countable non-isomorphic such $B$ (``non-standard models'' of arithmetic); towards uncountable models one can go from $\mathbb{Z}$ to  $\mathbb{Q}$ and then to completions $\MC$, or $\mathbb{R}$, or $\mathbb{Q}_p$,  which are tame from the view-point of the first-order logic.  Since Tarski and Malcev there has been many interesting results about elementary equivalence of finitely generated groups and rings. A question dominating research in this area has been if and when elementary equivalence between finitely generated  groups (rings) implies isomorphism. Recently, Avni et. al. \cite{ALM} coined the term {\em first-order rigidity}: a finitely generated group (ring) $A$ is first-order rigid if any other finitely generated group (ring) elementarily equivalent to $A$ is isomorphic to $A$.   
 
 In another direction Malcev proposed to test the  expressive power of the first-order logic by studying definable subgroups of a given group. A subgroup $K$ of a group $A$ is definable in $A$ if there is a first-order formula $\phi(x)$ without parameters such that $K$ consists precisely of those elements in $A$ that make $\phi(x)$ true. The main general question here is  how rich are definable subsets in a group or algebra. 
  In particular, in 1965 Malcev asked which  subgroups of free groups $F$ are definable in the first-order logic.  This problem was solved in  \cite{KMdef} and  \cite{PPST}. It turned out that among proper subgroups of  $F$ only cyclic ones are definable. The same holds in  arbitrary non-elementary torsion-free hyperbolic groups. It follows that the first-order theory is not quite adequate in describing subgroups in these groups. On the other hand there are groups which have interesting definable subgroups (nilpotent, metabelian, etc.).    However, till now there were no known infinite groups where all subgroups were {\em uniformly definable}, i.e., groups $G$ where for each natural $n$ there is a formula $\Phi_n(x_1, \ldots,x_n,y)$ which holds in $G$ on elements $a_1, \ldots, a_n, b$ if and only if $b$ belongs to the subgroup generated by  $a_1, \ldots,a_n$. This is a very strong definability property which is not amenable to  any known algebraic techniques.

  We intend to address the first-order rigidity and Malcev's question on definability of subgroups in a wide class of groups and rings. 
This approach allows one to show that the language of the first-order logic  in  various  classical groups and algebras has the same expressive power as the language of the weak second-order logic. Such groups include 
 $GL_n(\mathbb Z), SL_n(\mathbb Z), T_n(\mathbb Z),  n\geq 3$ \cite{MS2019}, various finitely generated metabelian groups (for example, the free nonabelian ones), many polycyclic groups, free associative algebras, free group algebras over infinite fields,  and many others. 
To prove this we show that the group (ring) $G$ in question is bi-interpretable with $HF(G)$~--- the superstructure of the hereditary finite sets over $G$. In this case the weak second-order logic over $G$ is interpretable in $G$. We term  such a group (or any such structure) $G$ {\it rich}. In fact, to show that $G$ is rich it suffices to bi-interpret it with any rich structure. Notice, that the arithmetic $\mathbb{Z}$ is rich, so to prove that $G$ is rich it suffices to show that it is bi-interpretable with the ring $\mathbb{Z}$.

   Many objects associated to a rich group (or ring) $G$, such as finitely generated subgroups (subrings, ideals), the geometry of its Cayley graph, as well as many others are uniformly definable in it. We observe, that rich groups are first-order rigid. Furthermore, finitely generated rich structures are completely characterized by a single  axiom. Interestingly, as we mentioned above free and torsion-free hyperbolic groups are not rich. Actually they are very far from being rich, however, their group algebras over infinite fields are rich~\cite{KMga}. This shows  how much more expressive is the first order ring language  of  a group algebra of a  free group compared with  the first order language of the group.

  There are many examples of first-order rigid groups and algebras. For instance,   Avni, Lubotzky and  Meiri ~\cite{ALM} showed that non-uniform higher dimensional lattices are first-order rigid, as well as finitely generated profinite groups~\cite{Lub}. Lasserre showed that under some natural conditions  polycyclic groups~\cite{Lasserre} are also first-order rigid, etc.  Some of these groups are rich and some not. It turns out that studying richness or lack there of is also very useful in studying arbitrary groups (rings) elementarily equivalent to a given one. For example, when a finitely generated  group $G$ is bi-interpretable with $\Z$, in many cases arbitrary groups that are elementarily equivalent to $G$ seem to have a very particular structure, they are kind of  ``completions'' or ``closures'' of $G$ with respect to a ring $R$ elementarily equivalent to $\Z$. When dealing with classical groups or algebras such notions of completion or closure coincide with the classical ones, where completions have the same ``algebraic scheme'', but the points are over the ring $R$ as above. The typical example is the group $G = \SLZ$, it is rich (See (1) in Theorem~\ref{sln-biinter:thm} below), so it is first-order rigid, moreover, any other group $H$ with $G \equiv H$ is isomorphic to $\SLR$ with $R \equiv \mathbb{Z}$.  On the other hand the ``extent'' to which a group $G$ is not rich also often seems to affect the structure of arbitrary groups elementarily equivalent to $G$. Again it seems such groups are ``deformations'' of ``exact completions'' or ``exact closures'' of $G$ over a ring $R$ as above. It only seems proper that these deformations can usually be captured by cohomological data. For example, the group $\TO$ is not rich, if the ring $\OR$ of integers of a number field has an infinite group of units (though $\TO$ and $\OR$ are mutually interpretable in each other). In this case the lack of richness is modulo the infinite center, therefore any group $H$ with $H \equiv \TO$ is an abelian deformation of a group $\TR$ where $R \equiv \OR$ (See (3) in Theorem~\ref{sln-biinter:thm} below). It is very interesting to study groups  which are rich but there is no any obvious ``algebraic scheme'' lying around, for example,  finitely generated  free metabelian groups. 
   In this case the completions still exist, but they are not described by any algebraic schemes, the schemes here are more general.

\section{Interpretability and bi-iterpretagbility}

\subsection{Interpretability}

The model-theoretic technique of interpretation or definability is crucial in our considerations. Because of that we remind here some precise definitions and several known facts that may not be very familiar to algebraists. 

A \emph{language} (or \emph{a signature}) $L$ is a triple $(Fun, Pr, C)$, where $Fun = \{f, \ldots \}$ is a  set of functional (or operational) symbols $f$ coming together with their arities  $n_f \in \mathbb{N}$,  $Pr$ is  a set of relation (or predicate) symbols $Pr= \{P, \ldots \}$ coming together with their arities  $n_P \in \mathbb{N}$,  and a set of constant symbols $C = \{c, \ldots\}$. Sometimes we write $f(x_1, \ldots,x_n)$ or $P(x_1, \ldots,x_n)$ to show that $n_f = n$ or $n_P = n$.  Usually we denote variables by small letters $x,y,z, a,b, u,v, \ldots$, while the same symbols with bars $\bar x, \bar y,  \ldots$ denote tuples of the corresponding variables, say   $\bar x = (x_1, \ldots,x_n)$.  In this paper we always assume, if not said otherwise, that the languages we consider are finite. The following languages appear frequently throughout the text:  the language of semigroups $\{\cdot\}$, where $\cdot$ is the binary multiplication symbol; the language of monoids $\{\cdot, 1\}$, where 1 is the constant symbol for the identity element; the language of groups $\{ \cdot, ^{-1}, 1\}$,  where $^{-1}$ is the symbol of inversion; and the language of rings $\{+, \cdot, 0\}$ with the standard symbols for addition, multiplication, and the additive identity $0$.  Sometimes we add the constant $1$ to form the language of  unitary rings (a priori, our rings  are not unitary).

An interpretation of a constant symbol $c$ in a set $A$ is an element $c^A\in A$. For a functional symbol $f$ an interpretation in $A$ is a function $f^A\colon A^{n_f}\to A$, and for a predicate $P$ it is a set $P^A\subseteq A^{n_P}$. 

A structure in the language $L$ (an $L$-structure) with the base set $A$, sometimes denoted by $\MA = \langle A; L\rangle$ or simply by 
$\MA = \langle A; f, \ldots, P, \ldots,c, \ldots \rangle$, is the set $A$ together with interpretations $f^A,\ldots,P^A,\ldots,c^A,\ldots$ Sometimes we will denote this structure also by $\MA = \langle A; f^A, \ldots, P^A, \ldots,c^A, \ldots \rangle$. For a given structure $\MA$ by $L(\MA)$ we denote the language of $\MA$.   When the language $L$ is clear from the context, we follow the standard algebraic practice and denote the structure $\MA = \langle A; L\rangle$ simply by $A$. For example, we refer to a field  $\mathbb{F}  = \langle F; +,\cdot,0,1 \rangle$ simply  as $F$, or to a group $\mathbb{G} = \langle G; \cdot,^{-1},1\rangle$ as $G$, etc. Sometimes we refer to a first-order formula in a language $L$ as to an $L$-formula and denote by  $\F_L$ the set of all $L$-formulas. 

The graph of a constant $c$ in $A$ is the set $\{c^A\}$, the graph of a function $f$ in $A$ is $\{(b_1,\ldots,b_{n_f},b) \in A^{n_f+1}\mid  f^A(b_1,\ldots,b_{n_f})=b, b_1,\ldots,b_{n_f}\in A\}$, and the graph of a predicate $P$ in $A$ is $P^A$ or $\{(b_1,\ldots,b_{n_P})\in A^{n_P}\mid \MA\models P^A(b_1,\ldots ,b_{n_P}) \}$.

Let $\MB = \langle B ; L(\MB)\rangle$ be a structure. A subset $A \subseteq B^n$ is called {\em definable}\index{definable subset} ({\em $0$-definable}, or {\em absolutely definable}, or {\em definable without parameters}) in $\MB$ if there is a formula $\phi(x_1, \ldots,x_n)$ (without parameters) in $L(\MB)$ such that  $A = \{(b_1,\ldots,b_n) \in B^n \mid \MB \models \phi(b_1, \ldots,b_n)\}$. In this case we denote $A$ by $\phi(B^n)$ or $\phi(\MB)$ and  say that  \emph{$\phi$ defines $A$} in $\MB$. If $\psi(x_1,\ldots,x_n,y_1,\ldots,y_k)$ is a formula in $L(\MB)$ and $\bar{p}=(p_1,\ldots,p_k)$ is a tuple of elements from $B$, then the set $\{(b_1,\ldots,b_n) \in B^n \mid \MB \models \psi(b_1, \ldots,b_n, p_1,\ldots,p_k)\}$ is called \emph{definable} in $\MB$ \emph{with parameters} $\bar{p}$  and denoted by $\psi(B^n,\bar p)$ or $\psi(\MB,\bar p)$. 

Let $c$, $f$ and $P$ be a constant, an operator and a predicate from some language $L$ (may be $L\ne L(\MB)$) which have  interpretations $c^A$, $f^A$, $P^A$ on the set $A\subseteq B^n$. The interpretations $c^A$, $f^A$ and $P^A$ are \emph{definable} in $\MB$ if their graphs are definable in $\MB$. Sometimes in this case we say that the constant $c$, the operator $f$ and the predicate $P$ are definable on $A$ in $\MB$.

In the same vein  an algebraic structure $\MA = \langle A ;f, \ldots, P, \ldots, c, \ldots\rangle$ is \emph{definable} (or $0$-\emph{definable}, or {\em absolutely definable}) in $\MB$ if there is a definable subset $A^\ast  \subseteq  B^n$ and interpretations $f^A, P^A, c^A$ on $A^\ast $ of the symbols $f, P, c$ all definable in $\MB$, such that the structure $\MA^\ast  = \langle A^\ast ; f^A, \ldots, P^A, \ldots,c^A, \ldots \rangle$ is isomorphic to $\MA$. (Notice, that constants $c,\ldots $ belong to the language of  $\MA$, they are not parameters.) 
For example, if $Z(G)$ is the center of a group $G$  then it is definable as a group  in $G$. 

One can do a bit more in terms of definability. In the notation above if $\sim$ is a definable  equivalence relation on a definable subset $A^\ast  \subseteq B^n$ then we say that the quotient set $A^\ast /\sim$ is {\em interpretable}\index{interpretable} in $\MB$. Furthermore, interpretations $f^A, P^A,  c^A$ of symbols $f, P, c\in L$ on the quotient set $A^\ast /\sim$ are {\em interpretable} in $\MB$ if the full preimages of their graphs in $A^\ast $ are definable in $\MB$. 

\begin{definition} \label{de:interpretable} An algebraic  structure $\MA = \langle A ;f, \ldots, P, \ldots, c, \ldots\rangle$  is absolutely interpretable (or $0$-interpretable)  in a structure $\MB = \langle B ; L(\MB)\rangle$ if there is a subset $A^\ast  \subseteq B^n$  definable in $\MB$, there is an equivalence relation $\sim$ on $A^\ast $ definable in $\MB$, and there are interpretations $f^A$, $P^A$, $c^A$ of the symbols $f, P, c$ on the quotient set $A^\ast /{\sim}$ all interpretable in $\MB$, such that the structure $\MA^\ast  = \langle A^\ast /{\sim}; f^A, \ldots, P^A, \ldots,c^A, \ldots \rangle$ is isomorphic to $\MA$.
 \end{definition}
 
 For example, if $N$ is a normal definable subgroup of a group $G$, then the equivalence relation $x \sim y$ on $G$ given by $xN = yN$ is definable in $G$, so the quotient set $G/N$ of  all right cosets of $N$ is interpretable in $G$. It is easy to see that the multiplication induced 
from $G$ on $G/N$ is also interpretable in $G$. This show that the quotient group $G/N$ is interpretable in $G$. 

Now we introduce some useful notation. An interpretation  of $\MA$ in $\MB$ is described  by the following set of formulas in the language $L(\MB)$
$$
\Gamma  = \{U_\Gamma(\bar x), E_\Gamma(\bar x_1, \bar x_2), Q_\Gamma(\bar x_1, \ldots,\bar x_{t_Q}) \mid Q  \in L(\MA)\}
$$
(here $\bar x$ and $\bar x_i$ are $n$-tuples of variables)  which  interpret $\MA$ in $\MB$ (as in the Definition \ref{de:interpretable} above). Namely,  $U_\Gamma$  defines in $\MB$ a subset  $A_\Gamma  = U_\Gamma(B^n)  \subseteq B^n$, $E_\Gamma$  defines an  equivalence relation $\sim_\Gamma$ on $A_\Gamma$, and the formulas $Q_\Gamma$ define preimages of graphs for constants, functions, predicates $Q\in L(\MA)$ on the quotient set $A_\Gamma/\sim_\Gamma$ in such a way that the structure $\Gamma(\MB) = \langle A_\Gamma/\sim_\Gamma; L(\MA) \rangle $ is isomorphic to $\MA$.  Note, that   we interpret a constant $c \in L(\MA)$ in the structure $\Gamma(\MB)$ by the $\sim_\Gamma$-equivalence  class of some tuple $\bar b_c \in A_\Gamma$ defined in $\MB$ by the formula $c_\Gamma(\bar x)$. We refer to $\Gamma$ as an {\em interpretation code}\index{interpretation code} or just {\em code}. The number $n$ is called the {\em dimension} of $\Gamma$, denoted $n = \dim \Gamma$. And we write $\MA \simeq \Gamma(\MB)$ if the code $\Gamma$ interprets $\MA$ in $\MB$ as described above. By $\mu_\Gamma$ we denote a surjective map $A_\Gamma \to A$ (here $\MA = \langle A;L(\MA)\rangle$) that gives rise to an isomorphism   $\bar \mu_\Gamma\colon \Gamma(\MB) \to \MA$. We refer to  this map $\mu_\Gamma$ as the \emph{coordinate map}\index{coordinate map} of the interpretation $\Gamma$. Sometimes we call the relation $\sim_\Gamma$  the \emph{kernel}  of the coordinate map $\mu_\Gamma$ and denote it by $\ker(\mu_\Gamma)$. Finally, notation $\mu\colon \MB \rightsquigarrow \MA$  means that  $\MA$ is interpretable in $\MB$ with the coordinate map $\mu$.
We use this notation throughout the paper.

More generally, the formulas that interpret $\MA$ in $\MB$ may contain elements from $\MB$ that are not in the language $L(\MB)$, i.e., some parameters, say $p_1, \ldots,p_k \in B$.  In this case we assume that all the formulas from the code $\Gamma$ have a tuple  of extra variables $\bar y = (y_1, \ldots,y_k)$  for parameters in $\MB$: 
\begin{equation} \label{eq:code}
\Gamma =  \{U_\Gamma(\bar x,\bar y), E_\Gamma(\bar x_1, \bar x_2,\bar y), Q_\Gamma(\bar x_1, \ldots,\bar x_{t_Q},\bar y) \mid Q \in L(\MA)\}
\end{equation}
so that after the assignment $y_1 \to p_1, \ldots,y_k \to p_k$ the code $\Gamma$ interprets $\MA$ in $\MB$.  In this event we write $\MA \simeq \Gamma(\MB,\bar p)$ (here $\bar p = (p_1, \ldots,p_k)$), and say that $\MA$ is interpretable in $\MB$  by the code $\Gamma$ with parameters $\bar p$. We refer to $k$ as the {\em parameter dimension} of $\Gamma$ and denote it by $\dim_{par}\Gamma$. In the case when $\bar p = \emptyset$ one gets again the absolute interpretability. The coordinate map $U_\Gamma(\MB,\bar p)\to\MA$ we will also denote by $\mu_{\Gamma,\bar p}$, if needed, and will write $\mu_{\Gamma,\bar p}\colon \MB \rightsquigarrow \MA$ to emphasize that $\MA$ is interpretable in $\MB$ with the code $\Gamma$, parameters $\bar p$ and coordinate map $\mu_{\Gamma,\bar p}$. Sometimes we will also write $\MA\simeq_{\mu_{\Gamma,\bar p}} \Gamma(\MB, \bar p)$ (or $\MA\simeq_{\mu_{\Gamma}} \Gamma(\MB, \bar p)$, or  $\MA\simeq_\mu \Gamma(\MB, \bar p)$) in this case.

Often together with a map $\mu\colon U_\Gamma \to A$ we will consider $\mu^m\colon U^m_\Gamma\to A^m$ the Cartesian power of $\mu$. When it is clear from the context we will omit $m$ in the notation $\mu^m$.

We will say that a subset $D \subseteq A_\Gamma/\sim_\Gamma$ is definable in $\MB$ if its full preimage in $ A_\Gamma$ is definable in $\MB$. More generally, a subset $D \subseteq (A_\Gamma/\sim_\Gamma)^m$ is definable in $\MB$ if its full preimage in $A_\Gamma^m$ under the natural projection $A_\Gamma^m \to (A_\Gamma/\sim_\Gamma)^m$ is definable in $\MB$.

We say that a structure $\MA$ is interpreted in a given structure $\MB$ \emph{uniformly}\index{uniform interpretation} with respect to a subset $D \subseteq B^k$ if  there is a code $\Gamma$ such that $\MA \simeq \Gamma(\MB,\bar p)$ for every tuple of parameters $\bar p \in D$. If $\MA$ is interpreted in  $\MB$ uniformly with respect to a $0$-definable subset $D \subseteq B^k$ then we say that $\MA$ is \emph{regularly interpretable}\index{regular interpretation} in $\MB$ and write in this case $\MA \simeq \Gamma(\MB,\phi)$, provided  $D$ is defined by $\phi$ in $\MB$. 
Note that the absolute interpretability is a particular case of the regular interpretability where the set  $D$ is empty. 

We  need one more uniformity condition. Absolute interpretation of $\MA$ in a class of structures $\mathcal{C}$ in a language $L_C$ is    {\em uniform}  if there is an interpretation code $ \Gamma$ without parameters such that $\MA \simeq \Gamma(\MB)$  for every  structure $\MB$ from $\mathcal{C}$. Similarly, one defines uniform regular interpretation of $\MA$ in $\mathcal{C}$ if there is a code $\Gamma$ and a formula $\phi$ in the language $L_C$ such that $\MA \simeq \Gamma(\MB,\phi)$ for every $\MB \in \mathcal C$.  

Note that we may consider the code $\Gamma$ from~\eqref{eq:code} as a special set of $L(\MB)$-formulas without requiring that $\MA$ is interpretable in $\MB$. Let us define the {\em translation} or {\em reduction} $\psi \to \psi_\Gamma$ for formulas in the language $L(\MA)$ into formulas in the language $L(\MB)$ depending on the code $\Gamma$. Take a formula $\psi(x_1,\ldots,x_m)$ in $L(\MA)$ with free variables $x_1,\ldots, x_m$. For our purposes it is enough to define the translation on formulas $\psi$ that is obtained from unnested atomic formulas by logical connectives and quantifires and written in prenex normal form as usual. Here by an unnested atomic formula we understand formulas either of the type $P(x_1, \ldots,x_{n_P})$ or $f(x_1, \ldots,x_{n_f}) = x_0$, $x_i = x_j$, $x=c$, where $P$ is a predicate symbol from $L(\MA)$,  $f$ is a functional symbol from $L(\MA)$, $c$ is a constant from $L(\MA)$ and $x_i$ are variables. Indeed, in general every formula can be effectively rewritten into an equivalent formula which has the form described above. We begin by defining a map $\psi \to \psi'$ as follows. Firstly, we replace every variable $x_i$ by an $n$-tuple of variables $\bar x_i$, where $n = \dim\Gamma$. Then we define the map $\psi \to \psi'$ on unnested atomic formulas as 
 $$
 P(x_1, \ldots,x_{n_P}) \;\; \to \;\; P_\Gamma(\bar x_1, \ldots, \bar x_{n_P},\bar y), 
 $$ 
 
 $$
 f(x_1, \ldots,x_{n_f}) = x_0\;\; \to \;\; f_\Gamma (\bar x_1, \ldots, \bar x_{n_f},\bar x_0,\bar y) 
 $$ 
 and 
  $$ 
  x_i = x_j \;\;\to \;\; E_\Gamma(\bar x_i,\bar x_j,\bar y),
  $$
  $$
   x_i = c \;\; \to \;\; c_\Gamma(\bar x_i, \bar y)
  $$
 (here $P_\Gamma, f_\Gamma, c_\Gamma, E_\Gamma$ are formulas from the code $\Gamma$). Furthermore, we put 
 $$(\psi_1 \vee \psi_2)' = \psi'_1 \vee \psi'_2,  \ \ (\psi_1 \wedge \psi_2)' = \psi'_1 \wedge \psi'_2, 
 $$
 $$ 
 (\psi_1 \to \psi_2)' = \psi'_1\to \psi'_2,  \ \ (\neg \psi_1)' = \neg \psi'_1.
 $$
 Now for quantifiers we define (below $\forall \bar x $ denotes $\forall x_1 \ldots \forall x_n$ and $\exists \bar x$ denotes $\exists x_1 \ldots \exists x_n$,  provided $\bar x = (x_1, \ldots,x_n)$)
 $$\exists x \: \psi_0  \;\;\to\;\; \exists \bar x\: (U_\Gamma(\bar x, \bar y) \wedge \psi'_0),  \ \ \forall x \: \psi_0 \;\;\to\;\; \forall \bar x \:(U_\Gamma(\bar x,\bar y) \to \psi'_0).
 $$
 And finally, we define the $\Gamma$-translation by
 $$
 \psi(x_1,\ldots,x_m) \;\; \to \;\; \psi_\Gamma(\bar x_1,\ldots, \bar x_m, \bar y)= \psi'(\bar x_1,\ldots, \bar x_m, \bar y) \wedge \bigwedge\limits_{i=1}^m U_\Gamma(\bar x_i, \bar y).
 $$

We use the following denotation below. Assume that $\MA$ is interpretable in $\MB$  with parameters $\bar p$ and code $\Gamma$, and $\mu_\Gamma\colon U_\Gamma(B^n,\bar p) \to A$ is the coordinate map of the interpretation $\Gamma(\MB,\bar p)$, where $n=\dim \Gamma$. For any map $\psi \to \psi^+$ translating formulas $\psi(x_1,\ldots,x_m)$ in $L(\MA)$ into formulas $\psi^+(\bar x_1,\ldots,\bar x_m,\bar y)$ in $L(\MB)$ (it could be $\Gamma$-translation or something else) we will write 
$$
 \MA\models \psi\Longleftrightarrow_{\mu_\Gamma} \MB\models \psi^+
$$
to express the following statement: for any elements $a_1,\ldots,a_m\in A$ and  $\bar b_i\in \mu_\Gamma^{-1}(a_i)$ if $\MA\models \psi(a_1,\ldots,a_m)$ then $\MB\models \psi^+(\bar b_1,\ldots, \bar b_m,\bar p)$,
and inversely for any $\bar b_1,\ldots,\bar b_m\in B^n$ if $\MB\models \psi^+(\bar b_1,\ldots, \bar b_m,\bar p)$ then there are $a_1,\ldots,a_m\in A$, such that $\bar b_i\in \mu_\Gamma^{-1}(a_i)$ and $\MA\models \psi(a_1,\ldots,a_m)$.

The following is a principal result on interpretability (in various forms it appears in the literature, see, for example~\cite{Hodges} or~\cite{Marker}). 

\begin{lemma}  \label{le:interpr_corol} 
Let $\MA$ be interpretable in $\MB$ with parameters $\bar p$, so $\MA \simeq  \Gamma(\MB,\bar p)$ for some code $\Gamma$, $n=\dim \Gamma$, and $\mu_\Gamma$ be the corresponding coordinate map. Then in the notations above one has:
\begin{enumerate}
\item[1)] For any formula $\psi(x_1,\ldots,x_m)$ of $L(\MA)$ the translation $\psi_\Gamma$ is effectively constructed, and 
$$
 \MA\models \psi\Longleftrightarrow_{\mu_\Gamma} \MB\models \psi_\Gamma.
$$
In particular, if $\MA$ is $0$-interpretable in $\MB$, then the formula $\psi_\Gamma$ does not contain parameters from $\MB$.
\item[2)] If $S\subseteq A^m$ is a definable with parameters subset then $\mu_\Gamma^{-1}(S)\subseteq B^{m\cdot n}$ is definable with parameters. Namely, if $S\subseteq A^m$ defines by a formula \ $\psi(x_1, \ldots, x_m, a_1, \ldots, a_q)$, where $a_i\in A$, then $\mu_\Gamma^{-1}(S)$ defines by a formula $\psi_\Gamma(\bar x_1, \ldots, \bar x_m, \bar b_1, \ldots, \bar b_q, \bar p)$ with $\bar b_i\in \mu_\Gamma^{-1}(a_i)$. In particular, if $\MA$ is $0$-interpretable in $\MB$ and $S\subseteq A^m$ is definable (without parameters), then $\mu_\Gamma^{-1}(S)$ is definable (without parameters).
\item[3)]
    $$
 \MA\models \psi\Longleftrightarrow_{\mu_\Gamma} \MB\models \psi^+,
$$
    where $\psi(x_1,\ldots,x_m)$ is any $\forall\exists$-formula in $L(\MA)$ of the type
    \begin{multline*}
     \forall \: t_1\ldots \forall \: t_s  \; (\bigwedge\limits_{i=1}^{r} \varphi^i(t_1,\ldots,t_s,x_1,\ldots,x_m) \:\longrightarrow\\ \longrightarrow \: \exists z_1 \ldots \exists \: z_q \; \bigwedge\limits_{j=1}^{d} \psi^j(z_1,\ldots,z_q,t_1,\ldots,t_s,x_1,\ldots,x_m) \,)
    \end{multline*}
    and $\psi^+(\bar x_1,\ldots,\bar x_m,\bar y)$ is an $L(\MB)$-formula defined by
    \begin{multline*}
    (\forall \: \bar t_1 \ldots \forall\: \bar t_s \; (\bigwedge\limits_{i=1}^{r} \varphi^i_\Gamma(\bar t_1,\ldots,\bar t_s,\bar x_1,\ldots,\bar x_m, \bar y) \:\longrightarrow\\ \longrightarrow \: \exists \bar z_1 \ldots\exists\: \bar z_q\; \bigwedge\limits_{j=1}^{d}{\psi^j_\Gamma(\bar z_1,\ldots,\bar z_q,\bar t_1,\ldots,\bar t_s,\bar x_1,\ldots,\bar x_m, \bar y)} \,)\,)\;\wedge\\
    \;\wedge\bigwedge\limits_{i=1}^m U_\Gamma(\bar x_i, \bar y). 
    \end{multline*}
\end{enumerate}
\end{lemma}

\begin{proof}
The statement~1) follows by induction on construction of $\psi$ and the definition of the interpretation, while 2), 3) follow from~1).
\end{proof}
 
 The result above shows that every set definable in $\Gamma(\MB,\bar p)$ is definable with parameters in $\MB$. However, the converse is not  true in general, i.e., a subset of $\Gamma(\MB,\bar p)$ definable  in $\MB$ may not be definable in $\Gamma(\MB,\bar p)$. Following~\cite{Marker} we say that interpretation $\MA \simeq \Gamma(\MB, \bar p)$ is \emph{pure} if every subset of $\Gamma(\MB,\bar p)$ definable in $\MB$ is definable in $\Gamma(\MB,\bar p)$. We will say more on pure interpretations in Section~\ref{se:bi-int}.
 
Recall that a {\em sentence}\index{sentence} in $L$ is a formula without free variables. 

\begin{lemma}  \label{le:interpr_corol_2} 
Let $\MA$ be regularly  interpretable in $\MB$, so $\MA \simeq  \Gamma(\MB,\phi)$ for some code $ \Gamma$ and formula $\phi(\bar y)$. Then for every sentence  $\psi$ of $L(\MA)$ one can effectively construct  a sentence  $\psi_{\Gamma,\phi}$ of $L(\MB)$ such that 
  $$
  \MA\models \psi \iff \MB\models \psi_{\Gamma,\phi}.
  $$
\end{lemma}

\begin{proof}
Let $\psi$ be a sentence in the language $L(\MA)$. To construct $\psi_{\Gamma,\phi}$  take the $\Gamma$-translation $\psi_\Gamma(\bar y)$ and put
$$
\psi_{\Gamma,\phi}= \forall \bar y (\phi(\bar y) \rightarrow  \psi_\Gamma(\bar y)).
$$
Also we may take $\psi_{\Gamma,\phi}= \exists \bar y (\phi(\bar y) \rightarrow  \psi_\Gamma(\bar y))$.
\end{proof}

Again, we refer to the map $\psi \to \psi_{\Gamma,\phi}$ as the \emph{translation}  or \emph{reduction} with respect to the code $\Gamma$ and formula $\phi$.

{\bf Admissibility conditions for an interpretation $\Gamma$:} Let $\MA \simeq \Gamma(\MB,\bar p)$, where $\Gamma$ is the code~\eqref{eq:code}. Then there is a set $\mathcal{AC}_\Gamma$ of formulas in the language $L(\MB)$ in variables $\bar y=(y_1, \ldots,y_k)$ (where $k$ is the length $|\bar p|$ of the tuple $\bar p$, $k=\dim_{par}\Gamma$), such that for any $L(\MB)$-structure $\MB^\prime$ and any tuple $\bar q$ of elements of $\MB^\prime$ with  $|\bar q| = k$ which satisfies in $\MB^\prime$ all of the formulas in $\mathcal{AC}_\Gamma$ the code $\Gamma$ with parameters $\bar q$ interprets in $\MB^\prime$ an $L(\MA)$-structure $\Gamma(\MB^\prime,\bar q)$. The set $\mathcal{AC}_\Gamma(\bar y)$ is called the \emph{admissibility conditions for $\Gamma$}. The  sentences in $\mathcal{AC}_\Gamma$ say that:
\begin{enumerate} 
\item $U_\Gamma(\MB^\prime,\bar q)\ne\emptyset$:
$$
\exists \, \bar x \; U_\Gamma(\bar x, \bar y);
$$
\item $E_{\Gamma}(\bar x_1, \bar x_2,\bar q)$ defines an equivalence relation $\sim$ on $U_\Gamma(\MB^\prime,\bar q)$:
$$
    \forall \, \bar x\; (U_\Gamma(\bar x, \bar y) \,\longrightarrow\, E_\Gamma (\bar x, \bar x, \bar y)),
    $$
    $$
    \forall\, \bar x_1 \forall\, \bar x_2 \; (U_\Gamma(\bar x_1,\bar y) \wedge U_\Gamma(\bar x_2,\bar y) \wedge E_\Gamma(\bar x_1,\bar x_2, \bar y)\, \longrightarrow \, E_\Gamma (\bar x_2, \bar x_1, \bar y)),
    $$
\begin{multline*}
 \forall \, \bar x_1 \forall \,\bar x_2 \forall \, \bar x_3 \; (U_\Gamma(\bar x_1,\bar y) \wedge U_\Gamma (\bar x_2,\bar y) \wedge U_\Gamma (\bar x_3,\bar y)\wedge \\
    \wedge\, E_\Gamma(\bar x_1,\bar x_2,\bar y)\wedge E_\Gamma (\bar x_2, \bar x_3,\bar y) \, \longrightarrow \, E_\Gamma(\bar x_1,\bar x_3,\bar y));   
\end{multline*}
\item Every constant $c\in L(\MA)$ has a correct interpretation on $U_\Gamma(\MB^\prime,\bar q)/\sim$:
$$
\exists \, \bar x\; (U_\Gamma (\bar x,\bar y)\wedge c_\Gamma (\bar x, \bar y)),
$$
\begin{multline*}
        \forall \, \bar x_1 \forall \, \bar x_2 \; (U_\Gamma(\bar x_1,\bar y)\wedge U_\Gamma (\bar x_2, \bar y) \,\wedge \,c_\Gamma (\bar x_1,\bar y) \,\wedge \,c_\Gamma (\bar x_2, \bar y)\,\longrightarrow\\
        \longrightarrow\,E_\Gamma(\bar x_1, \bar x_2, \bar y));
\end{multline*}
\item Every function $f\in L(\MA)$ has a correct interpretation on $U_\Gamma(\MB^\prime,\bar q)/\sim$:
$$
    \forall \, \bar x_1 \ldots \forall \,\bar x_{n_f}\; (\bigwedge\limits_{i=1}^{n_f}U_\Gamma (\bar x_i,\bar y) \,\longrightarrow\, \exists \, \bar z \;(U_\Gamma(\bar z,\bar y) \wedge f_\Gamma (\bar x_1, \ldots, \bar x_{n_f},\bar z, \bar y))),
$$
\begin{multline*}
 \forall \, \bar x_1 \ldots \forall \,\bar x_{n_f}\forall \, \bar x_1^\ast \ldots \forall \,\bar x_{n_f}^\ast \forall \, \bar z\: \forall \, \bar z^\ast\; ( U_\Gamma (\bar z,\bar y) \,\wedge \,U_\Gamma (\bar z^\ast,\bar y)
  \,  \wedge \\
    \wedge \,\bigwedge\limits_{i=1}^{n_f}(U_\Gamma (\bar x_i,\bar y)\,\wedge \,U_\Gamma (\bar x_i^\ast,\bar y)\,\wedge\, E_\Gamma (\bar x_i,\bar x_i^\ast,\bar y)) \,\wedge\, \\
    \wedge\, f_\Gamma (\bar x_1, \ldots, \bar x_{n_f},\bar z, \bar y)
    \,\wedge 
    \, f_\Gamma (\bar x_1^\ast, \ldots, \bar x_{n_f}^\ast,\bar z^\ast, \bar y)\,\longrightarrow\, E_\Gamma(\bar z, \bar z^\ast,\bar y));   
\end{multline*}
    
    \item Every predicate $P\in L(\MA)$ has a correct interpretation on $U_\Gamma(\MB^\prime,\bar q)/\sim$:
\begin{multline*}
    \forall \, \bar x_1 \ldots \forall \,\bar x_{n_P}\forall \, \bar x_1^\ast \ldots \forall\, \bar x_{n_P}^\ast \\ (\bigwedge\limits_{i=1}^{n_P}(U_\Gamma (\bar x_i,\bar y)\,\wedge\, U_\Gamma (\bar x_i^\ast,\bar y)\,\wedge\, E_\Gamma (\bar x_i,\bar x_i^\ast,\bar y)) \,\longrightarrow\, \\
    \,\longrightarrow\, (P_\Gamma (\bar x_1, \ldots, \bar x_{n_P}, \bar y) \longleftrightarrow P_\Gamma (\bar x_1^\ast, \ldots, \bar x_{n_P}^\ast, \bar y))).
\end{multline*}
\end{enumerate}
It is clear that if the language $L(\MA)$ contains at least one constant $c$ then the item~1 above can be omitted. In this way $\MB^\prime\models\mathcal{AC}_\Gamma(\bar q)$ if and only if $L(\MA)$-structure $\Gamma(\MB^\prime,\bar q)$ is well-defined and it is interpreted in $\MB$. (However, we cannot assert that $\Gamma(\MB^\prime,\bar q)\simeq \MA$!)

Note since the language $L(\MA)$ is finite then the set $\mathcal{AC}_\Gamma$ is also finite. In this case  one can take a single formula $\Theta_\Gamma(\bar y)=\bigwedge \{\theta(\bar y) \mid \theta(\bar y) \in \mathcal{AC}_\Gamma\}$, such that for any $L(\MB)$-structure $\MB^\prime$ one has $\MB^\prime\models\Theta_\Gamma(\bar q)$ if and only if  $\Gamma(\MB^\prime,\bar q)$ is well-defined.

Now, let $\MA \simeq \Gamma(\MB,\phi)$ be a regular interpretation of $\MA$ in $\MB$, i.e., $\MA \simeq \Gamma(\MB, \bar p)$ for any tuple $\bar p \in \phi(\MB)$. Consider the set $\mathcal{AC}_\Gamma = \{\theta_i(\bar y) \mid  i \in I\}$ of admissible conditions. 
Then for every $i \in I$ the sentence  
$$
\theta_{i,\phi} = \forall \,\bar y\:( \phi(\bar y) \,\to\, \theta_i(\bar y))
$$
holds in $\MB$. Let $\mathcal{AC}_{\Gamma, \phi} = \{\theta_{i,\phi} \mid i \in I\} \cup\{\exists \bar y \,\phi(\bar y)\}$.
Then for every $L(\MB)$-structure $\MB^\prime$ if $\MB^\prime \models \mathcal{AC}_{\Gamma, \phi}$ then the set $\phi(\MB^\prime)$ is not empty, and for every tuple $q \in \phi(\MB^\prime)$ the $L(\MA)$-structure $\Gamma(\MB^\prime,\bar q)$ is well-defined. Again, since  the language $L(\MA)$ by our assumptions is finite  then the set of sentences $\mathcal{AC}_{\Gamma, \phi}$ is also finite. In this case the set of admissible conditions can be described  by a single sentence $\Theta_{\Gamma,\phi}$~--- the conjunction of the sentences in $\mathcal{AC}_{\Gamma, \phi}$. We will use this in our study of quasi-finitely  axiomatizable structures.

{\bf Composition of interpretations:} Let $\MA, \MB$ and $\MC$ be algebraic structures. To find out how the composition of interpretations $\mu_\Gamma\colon \MB \rightsquigarrow \MA$ and $\mu_\Delta\colon \MC \rightsquigarrow \MB$ works, we define previously the composition of codes. Suppose that $\Gamma$ and $\Delta$ are codes, such that $\Gamma$ is a set of $L(\MB)$-formulas as in~\eqref{eq:code} and $\Delta$ is a set of $L(\MC)$-formulas:
$$
\Delta =  \{U_\Delta(\bar x,\bar z), E_\Delta(\bar x_1, \bar x_2,\bar z), Q_\Delta(\bar x_1, \ldots,\bar x_{t_Q},\bar z) \mid Q \in L(\MB)\},
$$
where $|\bar x|=|\bar x_i|=\dim \Delta$, $|\bar z|=\dim_{par}\Delta$. Then the {\em composition} of codes $\Gamma$ and $\Delta$ is the code
$$
 \Gamma \circ \Delta = \{U_{\Gamma\circ\Delta}, E_{\Gamma\circ\Delta}, Q_{\Gamma\circ\Delta} \mid Q \in L(\MA)\}=\{(U_\Gamma)_\Delta, (E_\Gamma)_\Delta, (Q_\Gamma)_\Delta \mid Q \in L(\MA)\}.
$$
Note that 
\begin{gather*}
\dim \Gamma\circ \Delta =   \dim \Gamma\cdot\dim \Delta,\\
\dim_{par}\Gamma\circ\Delta = \dim_{par}\Gamma\cdot \dim\Delta+\dim_{par}\Delta.
\end{gather*}

The following are two important technical results about the transitivity of interpretations, that we use throughout the paper.

\begin{lemma}  \label{le:int-transitivity}
Let $\MA, \MB$ and $\MC$ be algebraic structures. Then the following hold:
\begin{itemize}
\item [1)] If $\MA$ is interpretable with parameters (absolutely interpretable) in $\MB$ and $\MB$ is interpretable with parameters (absolutely interpretable) in $\MC$, then $\MA$ is interpretable with parameters  (absolutely interpretable) in $\MC$. In more detail, if $\mu_{\Gamma,\bar p}\colon \MB \rightsquigarrow \MA$ and $\mu_{\Delta, \bar q}\colon \MC \rightsquigarrow \MB$ then $\MA\simeq \Gamma\circ\Delta (\MC, \bar p^\ast)$, where $\bar p^\ast=(\bar p_0,\bar q)$ with $\bar p_0 \in \mu_{\Delta, \bar q}^{-1}(\bar p)$ (and the $L(\MA)$-structure $\Gamma\circ\Delta (\MC, \bar p^\ast)$ does not depend on the choice of $\bar p_0 \in \mu_{\Delta,\bar q}^{-1}(\bar p)$).
\item [2)] If $\MA$ is regularly interpretable in $\MB$ and $\MB$ is regularly interpretable in $\MC$, then $\MA$ is regularly interpretable in $\MC$. In more detail, if $\MA\simeq\Gamma(\MB,\phi)$ and $\MB\simeq\Delta(\MC,\psi)$ then $\MA\simeq \Gamma\circ\Delta (\MC, \phi^\ast)$, where $\phi^\ast=\phi_\Delta \wedge \psi$. Moreover, $\bar p^\ast \in \phi^\ast(\MC)$ if and only if there exist $\bar p\in \phi (\MB)$, $\bar q\in \psi (\MC)$ (and consequently
there exists an interpretation $\MB\simeq_{\mu_{\Delta,\bar q}} \Delta(\MC,\bar q)$) and $\bar p_0\in \mu^{-1}_{\Delta,\bar q}(\bar p)$ such that $\bar p^\ast=(\bar p_0, \bar q)$.
\end{itemize}
\end{lemma}

\begin{proof}
1) By Lemma~\ref{le:interpr_corol}  for any tuples $\bar b, \bar b_1, \ldots, \bar b_i, \ldots, $ of length $n=\dim\Gamma$ of elements of $\MB$, any their preimages $\bar c, \bar c_1, \ldots, \bar c_i, \ldots,$ under $\mu_{\Delta,\bar q}$ (that are tuples of length $m=\dim\Gamma\cdot \dim\Delta$ from $\MC$) and any $\bar p_0\in\mu_{\Delta,\bar q}^{-1}(\bar p)$ of length $|p|\cdot \dim\Delta$ of elements of $\MC$ one has 
\begin{equation}\label{eq:U}
  \MB\models U_\Gamma (\bar b,\bar p) \iff \MC\models U_{\Gamma\circ\Delta}(\bar c,\bar p_0,\bar q);
  \end{equation}
\begin{equation}\label{eq:E}
  \MB\models E_\Gamma (\bar b_1,\bar b_2,\bar p) \iff \MC\models E_{\Gamma\circ\Delta}(\bar c_1,\bar c_2, \bar p_0,\bar q);
  \end{equation}
\begin{equation}\label{eq:Q}
  \MB\models Q_\Gamma (\bar b_1,\ldots, \bar b_{t_Q},\bar p) \iff \MC\models Q_{\Gamma\circ\Delta}(\bar c_1, \ldots,\bar c_{t_Q}, \bar p_0,\bar q).
  \end{equation}
Here if $\bar c_i\in C^m$ and the right parts of~\eqref{eq:U}, \eqref{eq:E}, \eqref{eq:Q} are true, then $\bar c_i \in \mu_{\Delta,\bar q}^{-1}(B^n)$ and the left parts are true as well.

Furthermore, all formulas from the admissibility conditions $\mathcal{AC}_\Gamma$ are equivalent to $\forall\exists$-formulas of the type that in Lemma~\ref{le:interpr_corol}. Since $\MB\models\mathcal{AC}_\Gamma(\bar p)$, so Lemma~\ref{le:interpr_corol} proves that the algebraic structure $\MC$ satisfies the admissibility conditions $\mathcal{AC}_{\Gamma\circ\Delta}$ on parameters $\bar p^\ast=(\bar p_0,\bar q)$. Thus $L(\MA)$-structure $\Gamma\circ \Delta(\MC,\bar p^\ast)$ is well-defined and it does not depend on the choice of $\bar p_0\in \mu_{\Delta,\bar q}^{-1}(\bar p)$. 

Consider the map $\mu_{\Gamma\circ\Delta}\colon U_{\Gamma\circ\Delta}(\MC,\bar p^\ast)\to A$, $\mu_{\Gamma\circ\Delta}=\mu_{\Gamma,\bar p}\circ\mu_{\Delta, \bar q}$. As $\mu_{\Gamma, \bar p}$ is surjective and~\eqref{eq:U} holds, $\mu_{\Gamma\circ \Delta}$ is well-defined and also surjective. Because $\bar\mu_{\Gamma,\bar p}\colon U_\Gamma(\MB,\bar p)/\sim_\Gamma\, \to A$ is bijective and~\eqref{eq:E} holds, $\bar\mu_{\Gamma\circ\Delta}\colon U_{\Gamma\circ\Delta}(\MC,\bar p^\ast)/\sim_{\Gamma\circ\Delta} \,\to A$ is well-defined and bijective too. Finally, since $\bar\mu_{\Gamma,\bar p}$ is an $L(\MA)$-isomorphism and due to~\eqref{eq:Q} $\bar\mu_{\Gamma\circ\Delta}\colon \Gamma\circ \Delta(\MC,\bar p^\ast) \to \MA$ is an $L(\MA)$-isomorphism as well. Thereby, $\MA$ is interpretable in $\MC$ with the code $\Gamma\circ \Delta$ and parameters $\bar p^\ast$. It is clear that if $\bar p=\emptyset$ and $\bar q=\emptyset$ then $\bar p^\ast=\emptyset$.

2) Suppose now that $\MA\simeq\Gamma(\MB,\phi)$ and $\MB\simeq\Delta(\MC,\psi)$. Take any tuples $\bar p\in \phi (\MB)$ and $\bar q\in \psi (\MC)$. Then 
there exist interpretations $\MA\simeq_{\mu_{\Gamma,\bar p}} \Gamma(\MB,\bar p)$ and $\MB\simeq_{\mu_{\Delta,\bar q}} \Delta(\MC,\bar q)$. Take any $\bar p_0\in \mu^{-1}_{\Delta,\bar q}(\bar p)$ and put $\bar p^\ast=(\bar p_0, \bar q)$. We obtain that there exists an interpretation $\MA\simeq\Gamma\circ\Delta(\MC,\bar p^\ast)$, moreover, by Lemma~\ref{le:interpr_corol} parameters $\bar p^\ast$ belong to the definable set $\phi^\ast(\MC)$, $\phi^\ast=\phi_\Delta\,\wedge\,\psi$. Conversely, if we take any $\bar p^\ast\in \phi^\ast(\MC)$ and write $\bar p^\ast=(\bar p_0, \bar q)$, then $\bar q\in \psi (\MC)$ and there exists an interpretation $\MB\simeq_{\mu_{\Delta,\bar q}} \Delta(\MC,\bar q)$. And since $\bar p_0\in \phi_\Delta(\MC,\bar q)$, then by Lemma~\ref{le:interpr_corol} there exists $\bar p\in \phi(\MB)$ with $\bar p_0\in \mu^{-1}_{\Delta,\bar q}(\bar p)$. Thus for any $\bar p^\ast\in\phi^\ast(\MC)$ there exists an interpretation $\MA\simeq\Gamma\circ\Delta(\MC,\bar p^\ast)$, i.e., $\MA\simeq\Gamma\circ\Delta(\MC,\phi^\ast)$.
\end{proof}

\subsection{Properties preserved under interpretations}

Interpretations provide a useful tool to show that algebraic structures satisfy some  particular properties, that are inherited under interpretations. We say that a property $P$ is \emph{preserved  (or inherited) under interpretations} if for any structures $\MA$ and $\MB$ such that $\MA$ is interpreted in $\MB$ if $P$ holds in $\MB$ then $P$ holds in $\MA$.  Some of such properties $P$ are preserved only under absolute or regular interpretations, while others are preserved under interpretations with parameters. In this section we mention some of these properties that we  use in the sequel.

\begin{lemma} \label{co:interp} 
\begin{itemize}  The following holds:
\item  [1)] Let $\MA \simeq \Gamma(\MB,\phi)$ for some code $\Gamma$ and a formula $\phi$, i.e.,  $\MA$ is regularly interpretable in $\MB$. Then if   the first-order theory $Th(\MB)$ is decidable then  $Th(\MA)$ is also decidable. 
\item  [2)] Let $\MA_1,\MA_2$ be $L(\MA)$-structures and $\MB_1, \MB_2$ be $L(\MB)$-structures. If $\MA_i = \Gamma(\MB_i, \phi)$,  $i= 1,2$,  for some code $\Gamma$ and a formula $\phi$, then  $\MB_1 \equiv \MB_2$ implies $\MA_1 \equiv \MA_2$.
\item [3)] If $\MA = \Gamma(\MB, \bar p)$ and $\MB \preceq \MB_1$ is an elementary extension of $\MB$ then $\MB_1$ satisfies the admissibility conditions $\mathcal{AC}_\Gamma$ on the tuple $\bar p$,  $\Gamma(\MB, \bar p)$ is a substructure of $\MA_1 = \Gamma(\MB_1, \bar p)$   and $\MA  \preceq \MA_1$ is an elementary extension of $\MA$.
\end{itemize}
\end{lemma}

\begin{proof}
If $\MA \simeq \Gamma(\MB,\phi)$ for some code $\Gamma$ and a formula $\phi$, then by Lemma \ref{le:interpr_corol_2} there is an algorithm that for any sentence $\psi$ in the language $L(\MA)$ constructs a sentence $\psi_{\Gamma,\phi}$ in the language $L(\MB)$ such that 
$$
\MA \models \psi \Longleftrightarrow \MB \models \psi_{\Gamma,\phi}.
$$
Therefore, if $Th(\MB)$ is decidable then $Th(\MA)$ is decidable. This proves 1).
 
 Almost the same  argument works for 2), here one only need to notice that the sentence $\psi_{\Gamma,\phi}$ depends only on $\psi, \Gamma$, and $\phi$. 
 
 3) Since $\MB_1$ is an elementary extension of $\MB$  the admissibility condition $\mathcal{AC}_{\Gamma}$ holds in $\MB_1$ on the tuple $\bar p$. Hence $(\Gamma,\bar p)$ interprets in $\MB_1$ an $L(\MA)$-structure $\Gamma(\MB_1, \bar p)$, which has $\Gamma(\MB,\bar p)$ as a substructure. Let $\psi(x_1, \ldots,x_m)$ be a formula in $L(\MA)$ and $\psi_\Gamma(\bar x_1,\ldots,\bar x_m,\bar y)$ its $\Gamma$-translation.  Now take a tuple of elements  $\bar a= (a_1, \ldots,a_m)$ from the structure  $\Gamma(\MB,\bar p)$. We intend to show that $\psi(\bar a)$ holds in $\Gamma(\MB,\bar p)$ if and only if it holds in $\Gamma(\MB_1,\bar p)$. Consider a tuple $\bar b = (b_1, \ldots,b_m)$, where $b_i$ is a tuple from  $A_\Gamma = U_\Gamma(B^n)$ such that $a_i$ is the $\sim_\Gamma$-equivalence class of $b_i$, $i = 1, \ldots, m$. Then by Lemma~\ref{le:interpr_corol}  $\psi(\bar a)$ holds in $\MA \simeq \Gamma(\MB,\bar p)$  if and  only if $\psi_\Gamma(\bar b, \bar p)$ holds in $\MB$. Since $\MB \preceq \MB_1$ the formula $\psi_\Gamma(\bar b, \bar p)$ holds in $\MB$  if and only if it holds in $\MB_1$. Then, again by Lemma~\ref{le:interpr_corol} $\psi_\Gamma(\bar b, \bar p)$ holds in $\MB_1$ if and only if $\psi(\bar a)$ holds in $\Gamma(\MB_1,\bar p)$.  Hence 
 $\Gamma(\MB_1,\bar p)$ is an elementary extension of $\Gamma(\MB,\bar p)$, as claimed. This proves the lemma.
 \end{proof}

We also mention the following basic result. Let $\MN$ be the set of natural numbers.  Then   $\MN = \langle \MN; +,\cdot,0,1\rangle$ is called the \emph{arithmetic} or the \emph{standard arithmetic}. The G\"odel first incompleteness theorem (more precisely, the G\"odel's argument there) implies that the first-order theory $Th(\MN)$ of $\MN$ is undecidable. We state this in the form of the following  result.

\begin{lemma} \label{le:undec-Z} \cite{TRM,Ershov-Lavrov}
If the arithmetic $\MN$  is interpretable (with parameters) in a structure $\MB$ then the first-order theory $Th(\MB)$ is undecidable.
\end{lemma}

Note that $\MN$ is $0$-interpretable in the ring of integers $\MZ = \langle \MZ; +,\cdot,0,1\rangle$. Indeed, by Lagrange’s Four Squares Theorem every integer is a sum of four squares, so the formula 
$$
\exists x_1\exists x_2\exists x_3 \exists x_4(y = x_1^2 +x_2^2 +x_3^2 +x_4^2 )
$$
defines $\MN$ in $\MZ$. Hence, if $\MZ$ is interpretable (with parameters) in $\MA$ then $Th(\MA)$ is undecidable.   The following theorem, due to J. Robinson, shows this result can be extended to arbitrary rings of algebraic integers or algebraic number fields.

\begin{theorem} \label{th:J-Rob}  \cite{J-Robinson}
The following holds:
\begin{itemize}
    \item [1)] The ring $\MZ$ is $0$-definable in any ring of algebraic integers $\mathcal{O}$.
    \item [2)] The ring of algebraic integers $\mathcal{O}$ is $0$-definable in its field of fractions.
    \item [3)] It follows from 1) and 2) that the ring $\MZ$ is $0$-definable in any algebraic number field $F$.
    \end{itemize}
\end{theorem}

Scanlon extended the result above as follows.
\begin{theorem} \cite{Scanlon} \label{th:Scanlon}
The ring $\MZ$ is interpretable in any finitely generated infinite field $F$.
\end{theorem}

\begin{cor} \label{co-undec-essential}
If any of the rings $\mathcal O$ or $F$  from Theorems \ref{th:J-Rob} and \ref{th:Scanlon} are interpretable in a structure $\MA$ then the first-order theory $Th(\MA)$ is undecidable.
\end{cor}
At present,  there are known many rings, groups, and monoids $R$ where the arithmetic $\MN$ is interpretable (see, for instance, the Sections 6-14 below), so  for all of them Corollary \ref{co-undec-essential} holds when $\mathcal{O}$ or $F$  is replaced by $R$.  However, the  rings $\mathcal{O}$ and  $F$ above  most often  appear as interpretable in other structures, especially, the rings of algebraic integers $\mathcal{O}$ typically  occur via interpretations in finitely-generated groups (see, for example, \cite{GMO,GMO2,GMO3}, where the rings $\mathcal{O}$ are interpreted in groups and rings by means of equations).  
In this direction we would like to mention the following general result.

\begin{theorem} [Noskov \cite{Noskov}]
Let $G$ be a finitely generated soluble group. Then if $G$ is  non-virtually abelian then $\MZ$ is interpretable in $G$.
\end{theorem}

There are many properties  that are preserved under interpretations: $\lambda$-stability, being $\lambda$-saturation, $\omega$-categoricity, independence property, and finite Morley rank, to mention a few (for the proof about Morley rank we refer to~\cite{C-R}, the others are rather easy to prove using the translation $\phi \to \phi_\Gamma$ above).

\subsection{Bi-interpretability}
\label{se:bi-int}

In this section we discuss a very strong version of  mutual interpretability of two structures, so-called {\em bi-interpretability}. 

\begin{definition} \label{de:bi-inter}
Two algebraic structures  $\MA$ and $\MB$ are called {\em bi-interpretable}\index{bi-interpretable} (with parameters) in each other  if  the following conditions hold:
\begin{itemize}
\item [1)] $\MA$ and $\MB$ are interpretable (with parameters) in each other, so $\MA \simeq \Gamma(\MB,\bar p)$ and $\MB \simeq \Delta(\MA,\bar q)$ for some codes $\Gamma$ and $\Delta$ and tuples of parameters $\bar p, \bar q$. By transitivity  $\MA$, as well as  $\MB$,  is    interpretable (with parameters) in itself, so $\MA \simeq \Gamma \circ \Delta(\MA,\bar p^\ast)$ and $\MB \simeq \Delta \circ \Gamma(\MB,\bar q^\ast)$, where $\circ$ denotes composition of interpretations and $\bar p^\ast$,  $\bar q^\ast$  the corresponding parameters.
\item [2)]  There is a formula $\theta_\MA(\bar u, x, \bar s)$ in the language $L(\MA)$ such that $\theta_\MA(\bar u, x, \bar p^\ast)$ defines in $\MA$ an isomorphism $\bar \mu_\MA\colon \Gamma \circ \Delta(\MA,\bar p^\ast)  \to  \MA$ (more precisely, it defines the coordinate map $\mu_\MA\colon A_{\Gamma \circ \Delta} \to A$). Similarly, there is a formula $\theta_\MB(\bar v, x, \bar t)$ in the language $L(\MB)$ such that $\theta_\MB(\bar v, x,  \bar q^\ast)$ defines in $\MB$ an isomorphism $\bar\mu_\MB\colon \Delta \circ \Gamma(\MB,\bar q^\ast)\to \MB$ (more precisely, it defines the coordinate map $\mu_\MB\colon B_{\Delta \circ \Gamma} \to B$).
\end{itemize}
Here $|\bar p^\ast|=\dim_{par}\Gamma\circ\Delta=|\bar p|\cdot\dim \Delta+|\bar q|$, $|\bar q^\ast|=\dim_{par}\Delta\circ\Gamma=|\bar q|\cdot\dim \Gamma+|\bar p|$ and $|\bar u|=|\bar v|=\dim\Delta\circ\Gamma=\dim\Gamma\circ\Delta=\dim\Gamma\cdot\dim\Delta$.
\end{definition}

\begin{remark} \label{re:1}
On the definition of bi-interpretability:
\begin{itemize}
    \item [a)] 
In the Definition \ref{de:bi-inter} the parameters $\bar p^\ast$ and $\bar q^\ast $ can be explicitly  expressed from $\bar p$ and $\bar q$ using the coordinate maps $\mu_\Gamma \colon A_\Gamma \to A$  and  $\mu_\Delta\colon B_\Delta \to B$ of the interpretations $\MA \simeq \Gamma(\MB, \bar p)$ and $\MB \simeq \Delta(\MA, \bar q)$. Indeed, following the construction of the compositions $\Gamma \circ \Delta$ and $\Delta \circ \Gamma$ one can see that $\bar p^\ast = (\bar p_0, \bar q)$ and $\bar q^\ast = (\bar q_0,\bar p)$, where $\bar p_0
\in \mu^{-1}_\Delta (\bar p)$ and $\bar q_0\in \mu^{-1}_\Gamma(\bar q)$ (see Lemma~\ref{le:int-transitivity}).

\item [b)] Note that the tuples $\bar p_0, \bar q_0$ (and consequently $\bar p^\ast, \bar q^\ast$) from above are not unique, they are defined up to the equivalence relations $\sim_\Delta$ and $\sim_\Gamma$. Nevertheless, all these tuples define the same interpretations $\Gamma \circ \Delta(\MA,\bar p^\ast)$ and $\Delta \circ \Gamma(\MB,\bar q^\ast)$ (see Lemma~\ref{le:int-transitivity}). 
\item [c)] The requirement in 2) that the formula  $\theta_\MA(\bar u, x, \bar p^\ast )$ defines in $\MA$ a coordinate map $\mu_\MA\colon A_{\Gamma \circ \Delta} \to A$ can be strengthened to the requirement that it defines the composition $\mu_{\Gamma\circ\Delta}=\mu_\Gamma\circ\mu_\Delta$. Indeed,  we can redefine the isomorphism $\bar \mu_\Gamma\colon\Gamma(\MB,\bar p)\to\MA$ by taking $\alpha\circ \bar \mu_\Gamma$, where $\alpha=\bar\mu_\MA\circ \bar \mu^{-1}_{\Gamma\circ \Delta}$ is the automorphism of $\MA$. So we can always assume, adjusting $\bar \mu_{\Gamma}$ if necessary, that $\theta_\MA(\bar u, x, \bar p^\ast )$ defines in $\MA$ the isomorphism  $\bar \mu_{\Gamma \circ \Delta}$.
\end{itemize}
\end{remark}

In the case when the formula $\theta_\MA$ above defines the coordinate map \linebreak $\mu_{\Gamma\circ \Delta}\colon A_{\Gamma\circ\Delta}\to A$, we will refer to the isomorphism $\bar \mu_\Gamma\colon\Gamma(\MB,\bar p)\to\MA$ and its coordinate map $\mu_\Gamma\colon A_\Gamma\to A$ as {\em compatible with the formula} $\theta_\MA(\bar u,x,\bar p^\ast)$. Item~c) in Remark~\ref{re:1} tells that such isomorphism always exists. 

Usually in examples of bi-interpretable algebraic structures one has the coordinate map $\mu_\Gamma$  compatible with the formula $\theta_\MA$ and simultaneously the coordinate map $\mu_\Delta$  compatible with the formula $\theta_\MB$. In this instance the formula $\theta_\MA$ defines the composition $\mu_\Gamma\circ\mu_\Delta$ and $\theta_\MB$ defines $\mu_\Delta\circ\mu_\Gamma$.

\begin{lemma} \label{le:bi-interpret-6}
Let $\MA$ and $\MB$ be bi-interpretable (with parameters) in each other with respect to interpretations  $\MA \simeq \Gamma(\MB, \bar p)$ and $\MB \simeq \Delta(\MA,\bar q)$ for some codes $\Gamma$ and $\Delta$ and tuples of parameters $\bar p, \bar q$. Then the following holds:
\begin{itemize}
    \item [1)] There is a formula $\sigma_\MA(\bar y, \bar z)$ in the language $L(\MA)$ with $|\bar y|=|\bar p|\cdot\dim \Delta$, $|\bar z|=|\bar q|$, which defines in $\MA$ a non-empty set of tuples $D_{\sigma_\MA}$ such that in the notation in Definition~\ref{de:bi-inter} and Remark~\ref{re:1} $\bar p^\ast  =  (\bar p_0, \bar q) \in D_{\sigma_\MA}$ and for any tuple $\bar a  = (\bar a_0, \bar c)$ over $\MA$ with $|\bar a_0|=|\bar y|, |\bar c| = |\bar z|$ if $\bar a \in D_{\sigma_\MA}$ then:
           \begin{itemize}
               \item the $L(\MA)$-structure $\Delta(\MA,\bar c)$ is well-defined and $\bar a_0\in \Delta(\MA,\bar c)$ (more precisely, $\bar a_0\in U_\Delta(\MA,\bar c)$);
               \item the $L(\MA)$-structure $\Gamma \circ \Delta(\MA,\bar a)$ is well-defined; 
               \item  the formula $\theta_\MA(\bar u,x, \bar a)$ defines in $\MA$ an isomorphism  $\bar\mu\colon\Gamma \circ \Delta(\MA,\bar a)  \to \MA$ (more precisely, a coordinate map $\mu\colon U_{\Gamma\circ\Delta}(\MA,\bar a)\to A$).
           \end{itemize}
    
     \item [2)] Similarly, there is a formula $\sigma_\MB(\bar s, \bar t)$ in the language $L(\MB)$ 
     with $|\bar s|=|\bar q|\cdot\dim \Gamma$, $|\bar t|=|\bar p|$, which defines in $\MB$ a non-empty set of tuples $D_{\sigma_\MB}$ such that $\bar q^\ast  =  (\bar q_0, \bar p)$ is in $D_{\sigma_\MB}$ and for any tuple $\bar b  = (\bar b_0, \bar d) \in D_{\sigma_\MB}$, $|\bar b_0|=|\bar s|, |\bar d| = |\bar t|$, the $L(\MB)$-structures $\Gamma(\MB,\bar d)$ and $\Delta \circ \Gamma(\MB,\bar b)$ are well-defined, $\bar b_0\in \Gamma(\MB,\bar d)$, and the formula $\theta_\MB(\bar v,x, \bar b)$ defines an isomorphism $\Delta \circ \Gamma(\MB,\bar b)  \to \MB$.
    \end{itemize}
\end{lemma}
\begin{proof}
We use  notation from  Definition~\ref{de:bi-inter} and Remark~\ref{re:1}.  

It suffices to prove~1). To construct the required formula $\sigma_\MA(\bar y, \bar z)$ we will use the formulas $\Theta_\Delta(\bar z)$ (that says that $\bar c$ satisfies in $\MA$ the admissibility conditions $\mathcal{AC}_\Delta(\bar z)$) and $\Theta_{\Gamma\circ\Delta}(\bar y,\bar z)$ (that says that $\bar a$ satisfies in $\MA$ the admissibility conditions $\mathcal{AC}_{\Gamma\circ\Delta}(\bar y,\bar z)$). Thus $\MA\models\Theta_\Delta(\bar c)\wedge \Theta_{\Gamma\circ\Delta}(\bar a)$ if and only if the $L(\MA)$-structures $\Delta(\MA,\bar c)$ and $\Gamma \circ \Delta(\MA,\bar a)$ are well-defined, in particularly, when $\bar a=\bar p^\ast$. Note that if $\bar a\in \Theta_{\Gamma\circ\Delta}(\MA)$ then there exists $\bar x$ such that $\A\models U_{\Gamma\circ \Delta}(\bar x,\bar a)$, and since $U_{\Gamma\circ\Delta}$ is the $\Delta$-translation of the formula $U_\Gamma$, so $\bar a_0\in U_\Delta(\MA,\bar c)$ (see the algorithm of $\Delta$-translation for details).  

Assume now that the $L(\MA)$-structure $\Gamma \circ \Delta(\MA,\bar a)$ is well-defined. The formula $\theta_\MA(\bar u, x, \bar a)$ defines in $\MA$ an isomorphism  $\Gamma \circ \Delta(\MA,\bar a)  \to \MA$ if $\bar a=\bar p^\ast$. We are going to impose restrictions on $\bar a$ such that $\theta_\MA(\bar u, x, \bar a)$ defines an isomorphism for all $\bar a\in D_{\sigma_\MA}$. Let us write the following formulas to say that
\begin{enumerate}
    \item[(i)] a morphism $\mu\colon U_{\Gamma\circ\Delta}(\MA,\bar a)\to A$ exists:
    $$
    \theta_1(\bar y, \bar z)\:=\:\forall \,\bar u \; (U_{\Gamma\circ\Delta}(\bar u,\bar y, \bar z)\:\longleftrightarrow\: \exists \,x \;\theta_\MA(\bar u, x, \bar y, \bar z));
    $$
    \item[(ii)] the morphism $\mu$ is well-defined and the morphism $\bar \mu$ is injective:
    \begin{multline*}
      \theta_2(\bar y, \bar z)\:=\:\forall \,\bar u_1 \:\forall \,\bar u_2 \:\forall \, x_1 \: \forall \, x_2 \; (U_{\Gamma\circ\Delta}(\bar u_1,\bar y, \bar z)\,\wedge\, U_{\Gamma\circ\Delta}(\bar u_2,\bar y, \bar z)\,\wedge\\
      \wedge \,\theta_\MA(\bar u_1, x_1, \bar y, \bar z)\,\wedge\, \theta_\MA(\bar u_2, x_2, \bar y,\bar z) \:\longrightarrow\\
      \longrightarrow\: (E_{\Gamma\circ\Delta}(\bar u_1,\bar u_2, \bar y,\bar z)\,\longleftrightarrow\, (x_1=x_2)\,)\,);  
    \end{multline*}
    \item[(iii)] the morphism $\mu$ is surjective:
    $$
    \theta_3(\bar y, \bar z)\:=\:\forall \,x\: \exists \,\bar u \;  \theta_\MA(\bar u, x, \bar y,\bar z);
    $$
    \item[(iv)] the morphism $\mu$ preserves constants form $L(\MA)$: for each $c\in L(\MA)$
    $$
    \theta_c(\bar y, \bar z)\:=\:\forall \, \bar u \; (c_{\Gamma\circ\Delta}(\bar u, \bar y,\bar z)\, \longrightarrow \, \theta_\MA(\bar u, c, \bar y, \bar z));
    $$
    \item[(v)] the morphism $\mu$ preserves functions form $L(\MA)$: for each $f\in L(\MA)$
    \begin{multline*}
      \theta_f(\bar y, \bar z)\:=\:\forall \, \bar u_0\:\forall \, \bar u_1\ldots\forall\, \bar u_{n_f}\:\forall \, x_0\:\forall \, x_1\ldots\forall\, x_{n_f} \;
      (\bigwedge\limits_{i=0}^{n_f} \theta_\MA(\bar u_i, x_i,\bar y, \bar z)\,\wedge\\ \wedge\, f_{\Gamma\circ\Delta}(\bar u_1,\ldots,\bar u_{n_f},\bar u_0, \bar y, \bar z) \, \longrightarrow\, (f(x_1,\ldots,x_{n_f})=x_0)\,);    
    \end{multline*}
    \item[(vi)] the morphism $\mu$ preserves predicates form $L(\MA)$: for each $P\in L(\MA)$
    \begin{multline*}
      \theta_P(\bar y, \bar z)\:=\:\forall \, \bar u_1\ldots\forall\, \bar u_{n_P}\:\forall \, x_1\ldots\forall\, x_{n_P} \;
      (\bigwedge\limits_{i=1}^{n_P} \theta_\MA(\bar u_i, x_i,\bar y, \bar z)\, \longrightarrow\\ 
      \longrightarrow\, (P_{\Gamma\circ\Delta}(\bar u_1,\ldots,\bar u_{n_P},\bar y, \bar z) \, \longleftrightarrow\, P(x_1,\ldots,x_{n_P})\,)\,).   
    \end{multline*}
\end{enumerate}
Now, to write down the formula $\sigma_\MA(\bar y, \bar z)$  from 1) it suffices to take conjunction of the formulas $\Theta_{\Delta}(\bar z)$, $\Theta_{\Gamma\circ\Delta}(\bar y,\bar z)$, $\theta_1(\bar y, \bar z)$, $\theta_2(\bar y, \bar z)$, $\theta_3(\bar y, \bar z)$, $\theta_c(\bar y, \bar z)$, $\theta_f(\bar y, \bar z)$, $\theta_P(\bar y, \bar z)$ for all $c,f,P\in L(\MA)$. This proves 1) and the lemma.
\end{proof}

\begin{cor}
Let $\MA$ and $\MB$ be bi-interpretable (with parameters) in each other with respect to interpretations  $\MA \simeq \Gamma(\MB,\bar p)$ and $\MB \simeq \Delta(\MA,\bar q)$ for some codes $\Gamma$ and $\Delta$ and tuples of parameters $\bar p, \bar q$. Then in the notation of Lemma~\ref{le:bi-interpret-6} the interpretation $\Gamma \circ \Delta$ together with the formula $\sigma_\MA$ regularly interprets $\MA$ in $\MA$, i.e., $\MA \simeq \Gamma \circ \Delta(\MA,\sigma_\MA)$. Furthermore, for any tuples $\bar a_1$ and $\bar a_2$ which satisfy $\sigma_\MA$ in $\MA$ the formula 
$\Theta_{\bar a_1,\bar a_2}(\bar u_1, \bar u_2)= \exists \,x\: (\theta_\MA(\bar u_1,x,\bar a_1) \wedge \theta_\MA(\bar u_2,x,\bar a_2))$ defines an isomorphism $\lambda_{\bar a_1,\bar a_2}\colon \Gamma \circ \Delta(\MA,\bar a_1) \to \Gamma \circ \Delta(\MA, \bar a_2)$ such that ${\lambda_{\bar a_1,\bar a_2}(\bar u_1/\sim_{(\Gamma\circ\Delta,\bar a_1)})} = \bar u_2/\sim_{(\Gamma\circ\Delta,\bar a_2)}$. A similar statement holds for the interpretation $\Delta\circ \Gamma$.
\end{cor}

Algebraic structures  $\MA$ and $\MB$ are called {\em $0$-bi-interpretable} or {\em absolulutely  bi-interpretable}\index{absolulutely  bi-interpretable} in each other if in the definition above the tuples of parameters  $\bar p$ and $\bar q$ are empty.  

Unfortunately, $0$-bi-interpretability is rather rare. Indeed, the following result puts quite strong restrictions on such structures.

\begin{lemma}\label{autiso:lem} \cite[Section~5.4, Ex.~8 (b)]{Hodges}
If $\MA$ and $\MB$ are $0$-bi-interpretable in each other then their groups of automorphisms are isomorphic.
\end{lemma}

Fortunately, there is a notion of regular bi-interpretability, which is less restrictive, occurs more often, and which enjoys many properties of $0$-bi-in\-ter\-pre\-ta\-bi\-li\-ty.

\begin{definition} \label{de:regular-bi-int}
Two algebraic structures  $\MA$ and $\MB$ are called {\em regularly  bi-in\-ter\-pre\-table} in each other  if  the following conditions hold:
\begin{itemize}
\item [1)] $\MA$ and $\MB$ are regularly interpretable  in each other, so $\MA \simeq \Gamma(\MB,\phi)$ and $\MB \simeq \Delta(\MA,\psi)$ for some codes $\Gamma$ and $\Delta$ and the corresponding formulas $\phi, \psi$ (without  parameters). By transitivity  $\MA$, as well as  $\MB$,  is    regularly interpretable  in itself, so $\MA \simeq \Gamma \circ \Delta(\MA,\phi^\ast)$ and $\MB \simeq \Delta \circ \Gamma(\MB,\psi^\ast)$, where $\circ$ denotes composition of interpretations and  $\phi^\ast=\phi_\Delta\,\wedge\,\psi$, $\psi^\ast=\psi_\Gamma\,\wedge\,\phi$ (see Lemma~\ref{le:int-transitivity}).
\item [2)]  There is a formula $\theta_\MA (\bar u, x, \bar s)$ in the language of $\MA$ such that for every tuple $\bar p^\ast$ satisfying $\phi^\ast(\bar s)$ in $\MA$ the formula $\theta_\MA (\bar u, x,\bar p^\ast)$ defines in $\MA$ an isomorphism $\bar \mu_{\MA,\bar p^\ast} \colon \Gamma \circ \Delta(\MA,\bar p^\ast) \to \MA$  and there is  a formula $\theta_\MB (\bar v, x, \bar t)$ in the language of $\MB$ such that for every tuple $\bar q^\ast$ satisfying $\psi^\ast(\bar t)$ in $\MB$ the formula $\theta_\MB (\bar v, x, \bar q^\ast)$ defines in $\MB$ an isomorphism $\bar \mu_{\MB,\bar q^\ast} \colon  \Delta\circ \Gamma(\MB,\bar q^\ast)  \to \MB$.
\end{itemize}
Here $|\bar s|=\dim_{par}\Gamma\circ\Delta=\dim_{par}\Gamma\cdot\dim \Delta+\dim_{par}\Delta$, $|\bar t|=\dim_{par}\Delta\circ\Gamma=\dim_{par}\Delta\cdot\dim \Gamma+\dim_{par}\Gamma$ and $|\bar u|=|\bar v|=\dim\Delta\circ\Gamma=\dim\Gamma\circ\Delta=\dim\Gamma\cdot\dim\Delta$.
\end{definition}

Suppose that in the notation above for any tuple $\bar q\in \psi(\MA)$ the coordinate map $\mu_{\Delta,\bar q}\colon B_{\Delta,\bar q} \to B$ is fixed, and then for any tuples $\bar p\in \phi (\MB)$ and $\bar p_0\in \mu^{-1}_{\Delta,\bar q}(\bar p)$ the coordinate map $\mu_{\Gamma,\bar p^\ast}\colon A_{\Gamma,\bar p}\to A$ that is compatible with the formula $\theta_\MA(\bar u, x, \bar p^\ast)$ is also fixed, where $\bar p^\ast=(\bar p_0,\bar q)$. In this case we will say that the coordinate maps $\mu_{\Gamma,\bar p^\ast}$ (and isomorphisms $\bar\mu_{\Gamma, \bar p^\ast}$) are {\em uniformly compatible} with the formula $\theta_\MA$. In our applications in most cases one has precisely that the isomorphisms $\bar\mu_{\Gamma, \bar p^\ast}$ are uniformly compatible with the formula $\theta_\MA$.


\begin{lemma} \label{le:bi-in} 
Let $\MA$ and $\MB$ be structures bi-interpretable with parameters with respect to interpretations $\MA \simeq_{\mu_\Gamma} \Gamma(\MB,\bar p)$ and $\MB \simeq_{\mu_\Delta} \Delta(\MA,\bar q)$ for some codes $\Gamma$ and $\Delta$ and tuples of parameters $\bar p, \bar q$. Suppose that the coordinate map $\mu_\Gamma$ is compatible with the formula $\theta_\MA(\bar u, x,\bar p^\ast)$ defining an isomorphism $\Gamma\circ\Delta(\MA,\bar p^\ast)\to\MA$, $\bar p^\ast=(\bar p_0,\bar q)$, $\bar p_0\in \mu^{-1}_\Delta(\bar p)$. Then for any subset  $S \subseteq A^m$ one has that  $S$ is definable with parameters in $\MA$ if and only if $\mu_\Gamma^{-1}(S)$ is definable with parameters in $\MB$. 

Moreover, if $\MA$ and $\MB$ are absolutely bi-interpretable in each other then $S \subseteq A^m$ is definable in $\MA$ without parameters if and only if $\mu_\Gamma^{-1}(S)$ is definable in $\MB$ without parameters.     
\end{lemma}

\begin{proof}
In the proof we use notation from Definition \ref{de:bi-inter} and Remark \ref{re:1}. 

Suppose that $S$ is definable with parameters in $\MA$ then by Lemma~\ref{le:interpr_corol} it is definable with parameters in $\MB$. In particularly, if $\bar p=\emptyset$ and $S$ is definable without parameters then $\mu^{-1}_\Gamma(S)$ is definable without parameters.

To show the converse suppose that $\mu_\Gamma^{-1}(S)$ is definable in $\MB$ by a formula $\Sigma(\bar z,\bar r)$ with parameters $\bar r$ from $\MB$, $|\bar z|=m\cdot\dim\Gamma$. By Lemma~\ref{le:interpr_corol}  
$$
\MB \models \Sigma(\bar b, \bar r) \Longleftrightarrow \MA \models \Sigma_\Delta(\mu_\Delta^{-1}(\bar b),\bar r_0,\bar q), \quad \bar r_0\in\mu_\Delta^{-1}(\bar r).
$$
Here by $\mu_\Delta^{-1}(\bar b)$ in the formula above we mean any element from the set \linebreak $\mu_\Delta^{-1}(\bar b)$. 
This shows that the set $S_{\bar q} = \mu_\Delta^{-1}(\mu_\Gamma^{-1}(S))$, $S_{\bar q}\subseteq A^{m\cdot\dim \Gamma\cdot\dim \Delta}$, is definable in $\MA$ by the formula $\psi_\Delta$ and parameters $(\bar r_0,\bar q)$. Hence a tuple $\bar a$ of elements from $\MA$ belongs to $S$ if and only if there is a tuple $\bar c$ in $S_{\bar q}$ such that $\bar a =\mu_\Gamma \circ \mu_\Delta(\bar c)$. Since the set $S_{\bar q}$ and the map $\mu_\Gamma \circ \mu_\Delta$ are definable in $\MA$ it follows that the set $S$ is also definable with parameters in $\MA$ by the formula
$$
\sigma(x_1,\ldots,x_m, \bar p_0, \bar r_0, \bar q)\;=\; \exists \,\bar c_1\ldots \exists \,\bar c_m\; (\bigwedge\limits_{i=1}^m \theta_\MA(\bar c_i, x_i, \bar p^\ast) \:
\wedge \:\Sigma_\Delta(\bar c_1,\ldots,\bar c_m, \bar r_0,\bar q))
$$
with parameters $(\bar p_0, \bar r_0, \bar q)$. Here $|\bar c_i|=\dim\Gamma\circ\Delta$.

At the same time, if $\bar p=\bar q=\bar r=\emptyset$ then $S$ is definable without parameters. This proves the lemma. 
\end{proof}

\begin{lemma} \label{le:bi-in-0-definable} 
Let two structures $\MA$ and $\MB$ are regularly bi-interpretable in each other, so $\MA \simeq \Gamma(\MB,\phi)$ and $\MB \simeq \Delta(\MA,\psi)$, while $\theta_\MA(\bar u, x, \bar s)$ is an $L(\MA)$-formula, such that $\theta_\MA(\bar u, x, \bar p^\ast)$ defines an isomorphisms $\Gamma\circ\Delta(\MA,\bar p^\ast)\to\MA$ for all tuples $\bar p^\ast\in \phi^\ast(\MA)$, $\phi^\ast=\phi_\Delta\,\wedge\,\psi$. Suppose also that the coordinate maps $\mu_{\Gamma,\bar p^\ast}$ are uniformly compatible with the formula $\theta_\MA$.
Take an arbitrary subset $S \subseteq A^m$, such that for all tuples $\bar p^\ast\in \phi^\ast(\MA)$ the preimages $\mu^{-1}_{\Gamma,\bar p^\ast}(S)$ coincide, denote these preimages by $\mu^{-1}_\Gamma(S)$. Then the set $S$ is $0$-definable in $\MA$ if and only if $\mu^{-1}_\Gamma(S)$ is $0$-definable in $\MB$.
\end{lemma}
 
\begin{proof}
If $S \subseteq A^m$ is $0$-definable, i.e., $S=\varphi(A^m)$ for an $L(\MA)$-formula $\varphi$, then by Lemma~\ref{le:interpr_corol} $\mu^{-1}_\Gamma(S)=\Sigma(\MB)$, where
$$
\Sigma (\bar z) = \exists\, \bar y\;(\phi(\bar y) \,\wedge \,\varphi_\Gamma(\bar z, \bar y)),
$$
and $|\bar z|=m\cdot\dim \Gamma$. Thus $\mu^{-1}_\Gamma(S)$ is $0$-definable.

To prove converse statement suppose that $\mu^{-1}_\Gamma(S)$ is definable by an $L(\MB)$-formula $\Sigma (\bar z)$ without parameters with $|\bar z|=m\cdot\dim \Gamma$. And consider the $L(\MA)$-formula
\begin{multline*}
\varphi(x_1,\ldots,x_m)\;=\; \forall\,\bar q\:\forall\,\bar p_0 \;(\psi(\bar q)\,\wedge\,\phi_\Delta(\bar p_0,\bar q)\,\longrightarrow\\ \longrightarrow\,\exists \,\bar c_1\ldots \exists \,\bar c_m\; (\bigwedge\limits_{i=1}^m \theta_\MA(\bar c_i, x_i, \bar p_0,\bar q) \:
\wedge \:\Sigma_\Delta(\bar c_1,\ldots,\bar c_m, \bar q)\,)\,),
\end{multline*}
where $|\bar q|=\dim_{par}\Delta$, $|\bar p_0|=\dim_{par}\Gamma\cdot\dim\Delta$, $|\bar c_i|=\dim\Gamma\circ\Delta$. We state that $S=\varphi(\MA)$. Indeed, let $\bar a$ be a tuple form $S$. For any tuple $\bar p^\ast\in \phi^\ast(\MA)$ one gets that the set $\mu^{-1}_{\Gamma, \bar p^\ast}(S)$ is defined by $\Sigma (\bar z)$, therefore, by Lemma~\ref{le:bi-in} $\MA\models\sigma(\bar a, \bar p^\ast)$, where $\sigma$ is the formula from the proof of Lemma~\ref{le:bi-in} with $\bar r_0=\emptyset$. Thus $\bar a$ is in the set $\varphi(\MA)$. Conversely, let $\bar a\in\varphi(\MA)$. Chose a tuple $\bar p^\ast=(\bar p_0,\bar q)\in \phi^\ast(\MA)$. We obtain that $\MA\models\sigma(\bar a, \bar p^\ast)$, and by Lemma~\ref{le:bi-in}, $\bar a$ is in $S$, as claimed.
 \end{proof}

The following is a particular case of Lemma~\ref{le:bi-in-0-definable}.

\begin{cor} \label{co:bi-in-0-definable} 
Let two structures $\MA$ and $\MB$ are regularly bi-interpretable  in each other  in such a way that  $\MB \simeq \Delta(\MA,\psi)$ and $\MA  \simeq \Gamma(\MB)$ is an absolute interpretation. 
 Assume that $\mu_\Gamma$ is the coordinate map of the interpretation 
 $\MA  \simeq \Gamma(\MB)$ such that it is compatible with the formula $\theta_\MA(\bar u, x, \bar q)$ for all $\bar q\in \psi(\MA)$. For any set $S \subseteq A^m$ if the set $\mu_{\Gamma}^{-1}(S)$ is  $0$-definable in $\MB$, then the set $S$ is $0$-definable in $\MA$.
 \end{cor}

\begin{remark}
The type of regular bi-interpretability of $\MA$ and $\MB$ that is described in Corollary \ref{co:bi-in-0-definable} will be useful in the sequel. We sometimes refer to it as \emph{half-absolute bi-interpretability} of $\MA$ in  $\MB$. 
\end{remark}

\section{Weak second order logic}\label{se:weak-second-order} 
\label{se:weak-second-order-intro}

In this section we describe the weak second order logic (WSOL) over an arbitrary structure $\MA$. Informally, this logic allows one to use three different types of variables, one is for elements in $\MA$, another~--- for   finite subsets of $\MA$,  and  the third one~--- for natural numbers. Also,  one is allowed to use all the functions and predicates from the language of $\MA$, the ``membership'' predicate $\in$ for finite subsets of $\MA$, and addition and multiplication for natural numbers. Otherwise, the formulas are built like in the usual first-order languages. This logic is much more powerful than the standard first-order logic in $\MA$, indeed, everything which is describable in finitary set theory over $\MA$ can be described in the weak second order logic over $\MA$. 
 
Below, we describe WSOL  over $\MA$ in three different but equivalent ways, each one will be useful for different purposes. Firstly, we introduce two larger structures, containing $\MA$,  the superstructure $HF(\MA)$ of hereditary finite sets over $\MA$, and the list superstructures  $S(\N,\MA)$.  We show that they are bi-interpretable in each other, so they have essentially the same expressive power. The crucial result here is that the expressive power of the first-order theory of $HF(\MA)$ (or $S(\N,\MA)$) is the same as the power of WSOL over $\MA$. Secondly, we describe a  particular fragment $L_{WSOL}$ of the logic $L_{\omega_1,\omega}$   that is equivalent over $\MA$ to the weak second order logic. More precisely, in this fragment formulas with free variables $x_1, \ldots,x_n$  express  the same things in $\MA$ as  formulas in the language of $HF(\MA)$ or   $S(\N,\MA)$  which have only free variables $x_1, \ldots,x_n$  that run over $\MA$. This description will be used to show that some properties  of $\MA$ are definable in the WSOL.

\subsection{Hereditary finite and list superstructures} 

For a set $A$ let  $Pf(A)$ be the set of all finite subsets of $A$.
Now we define by induction  the set $HF(A)$  of hereditary finite sets over $A$;
\begin{itemize}
\item  $HF_0(A)= A $,  
\item $HF_{n+1}(A) = HF_n(A)\cup Pf(HF_n(A))$, 
\item  $HF(A)=\bigcup _{n\in\omega}HF_n(A).$  
 \end{itemize}

For a structure $\MA = \langle A;L \rangle$ define a new 
 first-order structure  as follows. Firstly, one replaces all operations in $L$ by the corresponding predicates (the graphs of the operations) on $A$, so one may assume from the beginning that $L$ consists only of predicate symbols. Secondly, consider the structure 
 $$
 HF(\MA)=\langle HF(A); L, P_A, \in\rangle, 
 $$
 where $L$ is defined on the subset $A$, $P_A$ defines $A$ in $HF(A)$, and $\in$ is the membership predicate on $HF(A)$.  Then everything that can be expressed in the weak second order logic in $\MA$ can be expressed in the first-order logic in $HF(\MA)$, and vice versa. The structure  $HF(\MA)$ appears naturally in the  weak second order logic, the theory of admissible sets, and $\Sigma$-definability,~--- we refer to  \cite{B1,B2,E,Ershov2} for details.

 There is another structure, termed the {\em list superstructure}\index{list superstructure} $S(\MA,\MN)$ over $\MA$ whose  first-order theory has the same expressive power as the weak second order logic over $\MA$ and which  is more convenient for us to use in this paper. To introduce $S(\MA,\MN)$ we need a few definitions.  Let $S(A)$ be the set of all finite sequences (tuples) of elements from $A$.
 For a  structure $\MA = \langle A;L \rangle$ define in the notation above  a new two-sorted structure $S(\MA)$ as follows:
$$
S(\MA) = \langle \MA, S(A); \frown, \in\rangle,
$$
where $\frown$ is the binary operation of  concatenation of two sequences from $S(A)$ and $a \in s$ for $a \in A, s \in S(A)$ is interpreted as $a$ being a component of the tuple $s$.  As customary in the formal language theory we will denote the concatenation  $s\frown t$ of two sequences $s$ and $t$ by   $st$.
 
 Now, the structure $S(\MA,\MN)$ is defined as  the three-sorted structure 
 $$
 S(\MA,\MN) = \langle \MA, S(A),\MN; t(s,i,a), l(s), \frown, \in \rangle,
 $$
 where $\N = \langle N, +,\cdot, 0,1\rangle$  is the standard arithmetic, $l\colon S(A) \to N$ is the length function, i.e., $l(s)$ is the length $n$ of a sequence $s= (s_1, \ldots,s_{n})\in S(A)$, and $t(x,y,z)$ is a predicate on $S(A)\times N \times  A$ such that $t(s,i,a)$ holds in $S(\MA,\MN)$ if and only if $s = (s_1, \ldots,s_{n})\in S(A), i \in N, 1\leq i \leq n$, and $a = s_i \in A$. Observe, that in this case the predicate $\in$ is $0$-definable in $S(\MA,\N)$ (with the use of $t(s,i,a)$), so sometimes we omit it from the language.

  In the following lemma and theorem we summarize some known results \cite{bauval} about the structures $HF(\MA), S(\MA)$, and $S(\MA,\N)$.
  
  \begin{lemma}\cite{bauval} \label{le:bauval}
  Let $\MA$ be a structure which has at least two elements. Then the following structures are $0$-interpretable in each other:
    $$S(\MA) \rightsquigarrow  S(\MA,\N) \rightsquigarrow HF(\MA) \rightsquigarrow S(\MA,\N) \rightsquigarrow S(\MA)$$ 
    uniformly in $\MA$.
   \end{lemma}
  
  Using ideas from  the proof of Lemma \ref{le:bauval} above from \cite{bauval} one can  prove the following result.

  \begin{theorem}\label{th:HF-S}
  Let $\MA$ be a structure that has at least two elements.  Then $HF(\MA)$ and $S(\MA,\N)$ are absolutely  bi-interpretable in each other.
    \end{theorem}
    \begin{proof} $HF(\{1,\ldots ,N\})$ is isomorphic to $(\N ,\{1,\ldots ,N\}, \in _N)$, where
  $m\in _N n$ if and only if $n>N$ and the coefficient of $2^m$ in the binary representation of $n-N$ is 1.
  
We now define the interpretation $\Gamma \colon HF(\MA)\rightarrow S(\MA,\N)$. A pair of elements $(S,m)$ of $S(\MA,\N)$ will represent an element $a$ of $HF(\MA)$ if and only if there exists a natural number $N$
and elements $a_1,\ldots, a_N$ of $A$ such that $S=a_1,\ldots,a_N$ and the isomorphism from  $(\N ,\{1,\ldots ,N\}, \in _N)$ to $HF(\{a_1,\ldots ,a_N\})$ sending $i$ to $a_i$ for any  $i\in \{1,\ldots ,N\}$ sends $m$ to $a$. 

We say that a finite sequence of natural numbers $v_0,\ldots ,v_k$ is an $N$-construction for a natural number $n$ if and only if $v_k=n$ and for any natural number $i\leq k$, for any $m\in\N$ such that $m\in_Nv_i$, there is a natural number $j<i$ such that $m=v_i$. Let $\sim$ be a binary relation defined by $(S,m)\sim(S',m')$ if and only if there exists an $\ell (S)$ construction of $m$, $v_0,\ldots ,v_k$, and an $\ell (S')$-construction of $m'$ with the same length $w_0,\ldots ,w_k$ , such that for any $i,j\leq k$, $v_1\in_{\ell(S)}v_j$ if and only if $w_1\in_{\ell(S')}w_j$ and for any $i\leq k$ such that  $v_i\leq \ell(S), v_i=w_i=0$, or $v_i$ and $w_i$ are different from $0$ and $t(S,v_i-1,a)$, $t(S',w_i-1,a)$ for some $a\in A$. Let $R$ be the binary relation defined by $(S,m)R(S',m')$ if and only if there exists $n$ such that $(S,m)\sim(S',m')$ and $n\in_{\ell (S')}m'.$

A finite sequence of natural numbers $v_0,\ldots ,v_k$ can be coded by the number 
${p_0}^{v_0+1}\cdots {p_k}^{v_k+1}$ (where $p_0,\ldots ,p_k$ is the increasing sequence of prime numbers). The set of representatives for finite sequences of natural numbers and the relation ``$k$ is the $m$'th element of the sequence coded by $n$",  are definable in $\N$, therefore $\sim$ and $R$ are definable in $S(\MA,\N)$.  The set $B=\{(S,m)|0<m\leq\ell (S)\}$ of representatives for the set $A$ in $HF(\MA)$ is definable in $S(\MA,\N)$. 

Define $a=\bar t(S,m)$ if and only if $t(S,m,a).$
Let $L'$ be the interpretation of $L$ in $B$ defined as follows: for any $k$-ary predicate $P$ of $L$, if $P$ is interpreted in $M$ by the relation $\tau$, $P$ is interpreted  in $B$ by the relation $\tau '$ defined by $\tau'((S_1,m_1),\ldots ,(S_k,m_k))$ if and only if $\tau (\bar t(S_1,m_1-1),\ldots ,\bar t(S_k,m_k-1)).$

By construction, $\sim$ is an equivalence relation compatible with $B,L',R$, and 
$HF(\MA)$ is isomorphic to $S(\MA,\N)=(B, S(A)\times \N, L',R)/\sim$.  

Combining the constructed interpretation $\Gamma$ of  $HF(\MA)$ in $S(\MA,\N)$ with the interpretation of $S(\MA,\N)$ in $HF(\MA)$ given in \cite[Section II.2]{bauval} which we denote by $\Delta$, one can easily verify that  the isomorphisms
 $\S(\MA,\N)\to (\Gamma \circ \Delta) S(\MA,\N)$ and  $HF(\MA) \to (\Delta \circ \Gamma) HF(\MA)$ are definable.

  \end{proof}
  
Now we define one more superstructure over $\MA$,  which we call the \emph{superstructure of finite binary predicates} over $\MA$.
 
 Let $A$ be a set. By a finite binary predicate over $A$ we mean a finite subset of pairs of elements from $A$. By $FBP(A)$ we denote the set of all finite predicates over $A$. For a structure $\MA = \langle A;L\rangle$ denote by $FBP(\MA)$ the following two-sorted superstructure
 $$
 FBP(\MA) = \langle \MA, FBP(A); s(x,y,z)\rangle 
 $$
 where $x$ and $y$ run over $A$, $z$ runs over $FBP(A)$ and $s(a,b,H)$ holds in $FBP(\MA)$ on $a, b \in A, H \in FBP(A)$ if and only if $(a,b) \in H$.

 \begin{theorem} \label{th:bi-int-HF-FBP}
  Let $\MA$ be an infinite structure. Then $HF(\MA)$ and $FBP(\MA)$ are absolutely bi-interpretable in each other.
 \end{theorem}
 \begin{proof}
 The proof is long and cumbersome. The main ideas are known in the folklore. We refer for some ideas to \cite{BT}  where it is shown how to rewrite formulas $\phi(x_1, \ldots,x_n)$ in the language of $HF(\MA)$ with free variables $x_i$ that run over $A$ into equivalent formulas $\phi^\ast (x_1, \ldots,x_n)$ in the language of $FBP(\MA)$ such that for any $a_1, \ldots,a_n \in \MA$  one has
 $$
 HF(\MA) \models \phi(a_1, \ldots, a_n) \Longleftrightarrow FBP(\MA) \models \phi^\ast (a_1, \ldots, a_n)
 $$
 \end{proof}
  
  The following result is known, it is based on two facts: the first one is that there are effective enumerations (codings) of the set of all tuples of natural numbers such that the natural operations over the tuples are computable on their codes;  and the second one is that  all computably enumerable predicates over natural numbers  are $0$-definable in $\N$ (see, for example, \cite{Coper, Rogers}).  
  
\begin{lemma} \label{le:bi-int-Z-Z}
The structures $\N$,  $\Z$, $S(\N,\N)$, $S(\Z,\N)$  are pair-wise absolutely  bi-interpretable with each other.
\end{lemma}

\subsection{The weak second order logic as a fragment of $L_{\omega_1,\omega}$}

For a language $L$  the logic $L_{\omega_1,\omega}$   admits a new (in comparison with the first-order logic) rule of forming formulas: if $\Phi$ is a countable set of formulas in $L_{\omega_1,\omega}$ then $\bigwedge \Phi$ (conjunction of $\Phi$) and $\bigvee \Phi$ (disjunction of $\Phi$) are also formulas  in $L_{\omega_1,\omega}$.  Below, following \cite{BT},  we describe a subset $L_{WSOL} \subset L_{\omega_1,\omega}$ which is equivalent to WSOL over $\MA$ (were $L  = L(\MA)$). 

We fixed an arbitrary effective enumeration of the set $\F_L$ of first-order formulas in $L$ by natural numbers, so there is an injection $\nu\colon  \F_L \to \N$ such that for every formula $\phi \in \F_L$  one can compute the number $\nu(\phi)$ and  given the number $n \in \N$ one can decide if $n = \nu(\phi)$ for some $\phi$  and if so find the formula $\phi$. We assume that $\nu(\F_L)$ has infinite complement so we can extend the enumeration to the fragment $L_{WSOL}$.  Similarly, we have an effective  enumeration $\mu\colon  \F_\N \to \N$ of  the first-order formulas of the arithmetic.  We define formulas in $L_{WSOL}$ by induction and, simultaneously, do two other things: we extend the enumeration $\nu$ to the the fragment $L_{WSOL}$ and define if a given formula $\phi \in L_ {WSOL}$ has complexity at most $r$, for $r \in \N$. 

Recall that a formula $\phi = Q_1x_1 \ldots Q_mx_m\psi$, where $Q_i$ are quantifiers $\forall, \exists$ and $\psi$ is quantifier-free, is in $\Sigma_n$ ($\Pi_n$), $n>0$,  if $Q_1 = \exists$ ($Q_1 = \forall$) and the 
prefix $Q_1x_1 \ldots Q_mx_m$ has at most $n-1$  alternations of quantifiers.  

We say that a set of formulas in $L_{WSOL}$ is arithmetic if the set of its codes (Godel's numbers) is arithmetic (definable in arithmetic).   

Now we describe the fragment $L_{WSOL}$. Every first-order formula $\phi$ in $\F_L$ is in $L_{WSOL}$. Its complexity is at most $r$ if it is in $\Sigma_r$ and $\Pi_r$ over $L$. If $\phi_1, \phi_2 \in L_{WSOL}$ then $\phi_1 \wedge \phi_2$, $\phi_1 \vee \phi_2$, $\phi_1 \to \phi_2$ are also in $L_{WSOL}$. We extend the enumeration (in some arbitrary but fixed way) onto  these formulas and define that  these formulas have complexity $\leq r$ if both of the formulas $\phi_1, \phi_2$ have complexity $\leq r$. If $\phi  \in L_{WSOL}$ then $^\neg \phi$, $ \forall x_i \phi$ and $\exists x_i \phi$ are in $L_{WSOL}$.  If complexity of $\phi$ is $\leq r$ then complexity of  $^\neg \phi$ is $\leq r$. If  $\phi = \forall x_j \psi$ and complexity of $\psi$ is $\leq r$ then complexity of $ \forall x_i \phi$ is $\leq r$; if 
$\phi = \exists x_j \psi$ then complexity of $\exists x_i \phi$ is also $\leq r$.  In all other cases complexity of $ \forall x_i \phi$ and $\exists x_i \phi$ is $\leq r+1$ where complexity of $\phi$ is $\leq r$. Again we extend the enumeration to these formulas. 
Finally, if $\Phi = \{\phi_i \mid i \in I\}$ is a set of arithmetic formulas from $L_{WSOL}$, all of them have complexity $\leq r$ and all their free variables are among $x_1, \ldots,x_m$ for some $m$ then $\bigwedge \phi_i$ and $\bigvee \phi_i$ are formulas from $L_{WSOL}$. We extend the enumeration to these formulas and their codes include the code of the arithmetic set $\nu(\Phi) \subset \N$. The complexity of these formulas are $\leq r+1$, where $\nu(\Phi)$ is $\Pi_r$ set in $\N$, and each formula $\phi$ has complexity $\leq r$. The formula $\bigwedge \phi_i$ is true in $\MA$ if every formula $\phi$ is true in $\MA$, while $\bigvee \phi_i$ is true in $\MA$ if some formula $\phi$ is true in $\MA$.

For our applications the following will suffice.
\begin{lemma}
Let
$\{\varphi_{i_1i_2\ldots i_n}(x_1,\ldots,x_m)|  i_1,i_2,\ldots,i_n \in {\bf N}\}$
be  a recursively enumerable set of quantifier-free formulas of language $L$. Then 

\[
  \bigvee\limits_{i_1} \bigwedge\limits_{i_2} \bigvee\limits_{i_3}  \ldots
\bigvee\limits_{i_n} \varphi_{i_1i_2\ldots i_n}(x_1,\ldots,x_m),
\]
and 
\[
\bigwedge\limits_{i_1}\bigvee\limits_{i_2} \bigwedge\limits_{i_3}\ldots
\bigwedge\limits_{i_n} \varphi_{i_1i_2\ldots i_n}(x_1,\ldots,x_m),
\]

are formulas in $L_{WSOL}$.

\end{lemma}

\begin{theorem} \cite{BT} \label{th:BT} The following holds for every algebraic structure $\MA$ of language $L$:
\begin{itemize}
\item [1)] for every formula $\phi (x_1, \ldots,x_n) \in L_{WSOL}$ one can construct a formula $\phi^\ast (x_1, \ldots,x_n)$ in the language of $HF(\MA)$, which does not contain any free variables other then  $x_1, \ldots,x_n$, and  such that for any assignment of variables $x_1 \to a_1, \ldots, x_n \to a_n$, where  $a_i \in \MA$ one has 
$$
\MA \models \phi(a_1, \ldots,a_n) \Longleftrightarrow HF(\MA) \models \phi^\ast (a_1, \ldots,a_n).
$$
\item [2)] for every formula $\psi (x_1, \ldots,x_n)$ in the language of $HF(\MA)$, which does not contain any free variables other then  $x_1, \ldots,x_n$ that run over $\MA$, one can construct a formula $\psi^\ast (x_1, \ldots,x_n) \in L_{WSOL}$ such that for any assignment of variables $x_1 \to a_1, \ldots, x_n \to a_n$, , where  $a_i \in \MA$ one has 
$$
HF(MA) \models \psi(a_1, \ldots,a_n) \Longleftrightarrow \MA \models \psi^\ast (a_1, \ldots,a_n).
$$

\end{itemize}

\end{theorem}

\section{Definitions and basic properties of rich structures}

Let $\MA$ be a structure. By a WSOL  formula $\phi(x_1, \ldots,x_n)$  in the language $L = L(\MA)$ we understand loosely  a formula either in the language of $HF(\MA)$, or in  $S(\N,\MA)$, or in $L_{WSOL}$ with only free variables that occur in the list $x_1, \ldots,x_n$ and these variables run over $\MA$. So $\phi$ does not have free variables that run over hereditary finite subsets of $A$ (in the case of $HF(\MA)$) or over lists over $\MA$, or over $\N$ (in case of $S(\N,\MA)$).

\begin{definition}\label{de:rich}
Let $\MA = \langle A,L\rangle$ be a structure in a language $L$. The structure $\MA$ is termed:

\begin{itemize}
    \item [1)] {\em rich}\index{rich structure}  if for every WSOL formula $\phi(x_1, \ldots,x_n)$  in the language $L = L(\MA)$  there is  a first-order formula $\phi^{\diamond}(x_1, \ldots,x_n, \bar p)$  in the language $L$ with parameters $\bar p$ in $\MA$ such that for any $a_1, \ldots,a_n \in \MA$
$$
\MA \models \phi(a_1, \ldots,a_n) \Longleftrightarrow \MA \models \phi^{\diamond}(a_1, \ldots,a_n, \bar p).
$$
\item [2)]  {\em absolutely rich}\index{absolutely rich structure} if for any formula $\phi(x_1, \ldots,x_n)$ as above  the formula $\phi^{\diamond}(x_1, \ldots,x_n)$ from 1) has no parameters. 
\item [3)] {\em effectively rich}\index{effectively  rich structure} ({\em absolutely and effectively rich}) \index{ absolutely and effectively rich structure} if the map $\phi \to \phi^\diamond$ is computable.
\end{itemize}
\end{definition}

Let $\phi(x_1, \ldots,x_n, y_1, \ldots,y_m)$ be a formula in $L_{WSOL}$ and $p_1, \ldots,p_m \in A$. Then, as usual,  by $\phi(x_1, \ldots,x_n, p_1, \ldots,p_m)$  we denote the formula $\phi$ with parameters $p_1, \ldots,p_m$. Then the structure $\MA$ is rich if and only if  any subset $S \subseteq A^n$  which is definable in $\MA$ by a $L_{WSOL}$-formula  with parameters in $\MA$ is definable  in $\MA$ by a first order formula  in $L$  with parameters in $\MA$.

In this case everything which is described by weak second order logic formulas  in $\MA$ can be also described by  first order ones in $\MA$.

\medskip
{\bf Examples}
\begin{itemize}
    \item [1)] 
 Any finite structure is  absolutely rich.
 \item [2)] For any structure $\MA$ the structure $HF(\MA)$ is   absolutely and effectively rich. 
 \end{itemize}
\begin{proof}
To see 1) observe that in a finite structure $\MA$ for every $n$ there are only finitely many non-equivalent over $\MA$ formulas in free variables $x_1, \ldots,x_n$. So every infinite conjunction (or disjunction) of formulas $\phi(x_1, \ldots,x_n), i \in I$, is equivalent over $\MA$ to a finite conjunction (or disjunction) of some of these formulas. 

To show 2) observe that for any structure $\MA$ the structures $HF(HF(\MA))$ is obtained  from $HF(\MA)$ by adding a new everywhere true unary predicate. Indeed, $HF(HF(A)) = HF(A)$, so $
HF( HF(\MA)) =\langle HF(A); L, P_A, P_{HF(A)}, \in \rangle, $ where $P_{HF(A)}$ defines the whole set $ HF(A)$. For every formula $\phi(x_1, \ldots,x_n)$ in the language of $HF(HF(\MA))$ where each $x_i$ runs over $HF(\MA)$ one can obtain an equivalent over $HF(\MA)$ formula by replacing each occurrence of the predicate $P_{HF(A)}(x)$ in $\phi$  by the formula $x= x$. The new formula is equivalent to $\phi$ over $HF(\MA)$.
 
\end{proof}

The argument above also shows that  $HF(\MA)$ is absolutely bi-interpretable with $HF(HF(\MA))$.

The following result  and its corollaries  give an easy tool to prove that a structure is rich.

\begin{theorem}\label{co:bi-int-rich}
The following holds:
\begin{itemize}
\item [1)] Any structure bi-interpretable (with parameters) with a rich structure is rich.
\item [2)] Any structure absolutely bi-interpretable with an absolutely rich structure is absolutely rich.
\item [3)] In the case above if one of the structures is effectively rich (effectively absolutely rich) so is the other.
\end{itemize}
\end{theorem}
\begin{proof}
Let $\MA = \langle A;L(\MA)\rangle$ and $\MB = \langle B;L(\MB)\rangle$ be bi-interpretable (with parameters) in each other. Suppose $\MB$ is rich. Denote by  $\mu =\mu_\Gamma\colon  \MA \rightsquigarrow \MB$ the coordinate map, so 
$\MA \simeq \Gamma(\MB, p)$ and $\mu\colon U_\Gamma(B) \to A$. To show that $\MA$ is rich consider a formula  $\phi(\bar x)$ in the WSOL in $L(\MA)$. We need to show that there is a first-order formula $\phi^{FO}(\bar x)$  in the language $L(\MA)$ such that for any tuple $\bar a$ in $\MA$ one has 
$$
\MA\models \phi(\bar a) \Longleftrightarrow \MA \models \phi^{FO}(\bar a).
$$
We prove this by induction on complexity of $\phi$. 
First, we define by induction on complexity of $\phi$ a formula $\phi_\Gamma(\bar z, \bar y)$ of WSOL in the language $L(\MB)$ (here the tuple of variables $\bar z$ corresponds to the tuple $\bar x$, and $\bar y$ corresponds to the tuple of parameters $p$, see Lemma \ref{le:interpr_corol}),  which we call the \emph{$\Gamma$ translation of $\phi$}. If $\phi(\bar x)$ is a first-order formula in the language $L(\MA)$ then  $\phi_\Gamma(\bar z, \bar y)$ is the $\Gamma$ translation of $\phi$ defined in Lemma \ref{le:interpr_corol}. Let $\Phi = \{\phi_i(\bar x) \mid i \in I\}$ be an arithmetic set of formulas of   WSOL in the language of $L(\MA)$ of uniformly bounded complexity and assume that  $\phi(\bar x) = \bigvee_{\phi \in \Phi} \phi_i(\bar x)$. Assume that by induction the $\Gamma$-translations $(\phi_i)_\Gamma(\bar z,\bar y)$ are defined for all $\phi_i \in \Phi$ and satisfy the following conditions:   each  $(\phi_i)_\Gamma(\bar z, \bar y)$ has the same tuple of free variables $\bar z, \bar y$, all formulas $(\phi_i)_\Gamma(\bar z, \bar y)$ have uniformly bounded complexity, the set  $\Phi_\Gamma = \{(\phi_i)_\Gamma(\bar z, \bar y)\mid i \in I\}$ is arithmetic, and 
 for all tuples $a$ in $\MA$
$$
\MA\models \phi_i (\bar a)\iff \MB\models (\phi_i)_\Gamma(\mu_\Gamma^{-1} (\bar a),\bar p).
  $$
  Put 
  $$
  \phi_\Gamma(\bar y) = \bigvee_{\phi_i \in \Phi} (\phi_i)_\Gamma(\bar z,\bar y)
  $$
 Then  $\phi_\Gamma(\bar z, \bar y)$ is in WSOL of $L(\MB)$  and 
 
 \begin{equation}\label{eq:Gamma-transl}
 \MA\models \phi (\bar a)\iff \MB\models (\phi)_\Gamma(\mu_\Gamma^{-1} (\bar a),\bar p).
 \end{equation}
 
We define similarly $\phi_\Gamma(\bar z,\bar y)$ for $\phi = \bigwedge_{\phi \in \Phi}\phi_i(\bar x)$. Assume now that $\phi = \forall \bar u \phi_0(\bar x, \bar u)$, where the complexity of $\phi_0$ is less then the complexity of $\phi$. Then we set $\phi_\Gamma(\bar z, \bar y) = \forall \bar v (\phi_0)_\Gamma(\bar z, \bar v,  \bar y )$, where the tuples of variables $\bar v$ correspond to the tuples $\bar u$  in the  $\Gamma$-translation of $\phi_0$.  Similarly we define $\phi_\Gamma(\bar z, \bar y)$ for $\phi = \exists \bar u \phi_0(\bar x, \bar u)$. In all these case $\phi_\Gamma$ satisfies the condition (\ref{eq:Gamma-transl}). This defines $\phi_\Gamma$ for every $\phi$ in WSOL of $L(\MA)$.

Since $\MB$ is rich there is a first-order formula   $\phi_\Gamma^{FO}(\bar z, \bar y)$ in $L(\MB)$ such that  for any tuples $\bar b$, $\bar c$ over $\MB$ one has 
$$
\MB \models \phi_\Gamma(\bar b, \bar c)  \Longleftrightarrow \MB \models \phi_\Gamma^{FO}(\bar b, \bar c)
$$
 It follows from  the above and (\ref{eq:Gamma-transl}) that 
 $$
 \MA\models \phi (\bar a)\iff \MB\models (\phi)_\Gamma(\mu_\Gamma^{-1} (\bar a),\bar p) \iff  \MB \models \phi_\Gamma^{FO}(\mu_\Gamma^{-1} (\bar a),\bar p).
 $$
 
 By Lemma \ref{le:bi-in} there is a formula $\psi(\bar z, \bar u)$ in the first-order language of $L(\MA)$ such that 
\begin{equation}\label{eq:a-a}
\MB \models \phi_\Gamma^{FO}(\mu_\Gamma^{-1} (\bar a),\bar p) \iff \MA \models \psi(\bar a,\bar p^\ast ),
\end{equation}
 where the tuple $\bar p^\ast $ described in 
   Lemma \ref{le:bi-in}. 
   
 The second and the third statement of the theorem follows by inspection from the first one.

\end{proof}

 

\begin{lemma}\label{le:rich-S(A,N)}
 For any structure $\MA$ the structures $S(\MA,\MN)$ and $FBP(\MA)$ are  absolutely and effectively rich. 
 \end{lemma}
\begin{proof}
By Theorem \ref{th:HF-S} for any structure $\MA$ the structures $S(\MA,\MN)$ and  $HF(\MA)$  are absolutely bi-interpretable in each other. If $\MA$ is finite then $FBP(\MA)$  is also finite hence absolutely rich. Otherwise by Theorem \ref {th:bi-int-HF-FBP} $FBP(\MA)$ and $HF(\MA)$ are absolutely bi-interpretable in each other. 
Since $HF(\MA)$ is absolutely and effectively rich (see the examples above) so, by Theorem \ref{co:bi-int-rich}, are the structures $S(\MA,\MN)$ and $FBP(\MA)$.  
\end{proof}

\begin{lemma}
The structures $\N$ and $\Z$ are absolutely and effectively rich.
\end{lemma}
\begin{proof}
  Indeed, it follows from Lemma \ref{le:bi-int-Z-Z} that $\N$, $\MZ$ and $S(\N,\N)$ are absolutely  bi-interpretable in each other.  By Lemma  \ref{le:rich-S(A,N)} the structure $S(\N,\N)$ is absolutely and effectively rich. Hence by Theorem \ref{co:bi-int-rich} \ref{th:HF-S} the structures $\N$, $\MZ$ are also absolutely and effectively rich.  
    \end{proof}

\begin{cor} \label{co:bi-int-Z}
Let $\MA$ be a structure. Then the following holds:
\begin{itemize}
    \item If $\MA$ is bi-interpretable with $\Z$ then $\MA$ is rich.
    \item If $\MA$ is absolutely bi-interpretable with $\Z$ then $\MA$ is absolutely and effectively rich.
\end{itemize} 
\end{cor}

\begin{cor}
The field of rational numbers $\Q$ is absolutely and effectively rich.
\end{cor}
\begin{proof}
Indeed, $\Z$ and $\Q$ are absolutely bi-interpretable in each other.
\end{proof}

The results above showed how one can use absolute bi-interpretability to prove that a  structure is absolutely rich. Now we describe a method that allows one  to prove that a structure is absolutely rich via regular bi-interpretability.

\begin{lemma} \label{le:half-bi-int}
Let a structure $\MA$ is half-absolute bi-interpretable in a structure $\MB$ (see Corollary \ref{co:bi-in-0-definable}). Then if $\MB$ is absolutely rich then $\MA$ is also absolutely rich.
\end{lemma}
\begin{proof}
By definition $\MA$ is half-absolute bi-interpretable in $\MB$ if they are regularly bi-interpretable   in such a way that  $\MB \simeq \Delta(\MA,\psi)$ and $\MA  \simeq \Gamma(\MB)$ is an absolute interpretation.  In our poof we follow the argument in Theorem \ref{co:bi-int-rich}. In the notation of the theorem, let   $\phi(\bar x)$ be a WSOL formula in the language of $\MA$. Consider the WSOL formula $(\phi)_\Gamma$  in thje language of $\MB$ constructed Theorem \ref{co:bi-int-rich}. It was shown in (\ref{eq:Gamma-transl})  that 

$$
 \MA\models \phi (\bar a)\iff \MB\models (\phi)_\Gamma(\mu_\Gamma^{-1} (\bar a)).
 $$
Since $\MB$ is absolutely rich there is a first-order formula $\phi^{FO}$ in the language of $\MB$ which is equivalent to $(\phi)_\Gamma$ on $\MB$. So if $S$ is the true-set of the formula $\phi$ in $\MA$ then $\mu_\Gamma^{-1}(S)$ is defined in $\MB$ by the formula $\phi^{FO}$ (without parameters). By Corollary \ref{co:bi-in-0-definable} the set $S$ is defined by some first-order formula (without parameters) in $MA$. This shows that $\MA$ is absolutely rich.  
\end{proof}

\begin{remark}
The situation of the Lemma \ref{le:half-bi-int} is typical when there is a regular bi-interpretability of a structure $\MA$ with $\MN$ or $\MZ$. Indeed, in this case, since every element of $\MN$ (or $\MZ$) is $0$-definable in $\MN$ (or $\MZ$) every interpretation (with parameters) of $\MA$ in $\MN$ (or in $\MZ$) is $0$-interpretation, so the regular bi-interpretation of $\MA$ and $\MZ$ becomes half-absolute. Moreover, as we mentioned above, the structures $\MN$ and  $\MZ$ are absolutely rich.
\end{remark}

The results above show that it is helpful  to have a large collection 
 of rich structures to show  that some other structures are rich.

\section{First-order rigidity, quasi-finite axiomatizability, primality, and homogeneity }\label{sec:qfa}

In this section we discuss the properties mentioned in the title above and their relationship  with bi-interpretability. 

\subsection{Primality and homogeneity}\label{sec:primal}

We begin by recalling some model-theoretic definitions (we refer to books \cite{Marker,Hodges} for details). For the remainder of this section $L$ is a countable language and $T$ is a complete $L$-theory with infinite models. 

A model $M$ of $T$  is a \emph{prime}\index{prime model} model of $T$ if it embeds elementarily in any model of $T$. A model $M$ of $T$ is \emph{atomic}\index{atomic model} if every type realized in $M$ is principal.
A model is $M$ is \emph{homogeneous}\index{homogeneity} if for every two tuples $\bar a, \bar b \in M^n$ ($n \in \MN$) that realise the same types in $M$ there is an automorphism of $M$ that maps $\bar a$ onto $\bar b$.
It is known that a model $M$ of $T$ is prime if and only if it is countable and  atomic. Furthermore, if $M$ is atomic then it is homogeneous. 

It is easy to see that $\MN = \langle N; +,\cdot,<,0,1\rangle$ is a prime model of $Th(\MN)$ and $\MZ = \langle Z; +,\cdot,<,0,1\rangle$ is a prime model of $Th(\MZ)$.

The following result is useful.

\begin{lemma} \label{le:prime}
Let $\MA$ and $\MB$ be infinite $L$-structures bi-interpretable (with parameters) in each other. If $\MA$ is prime in $Th(\MA)$ then $\MB$ is also prime in $Th(\MB)$.  
\end{lemma}
\begin{proof}
Follows from Lemma \ref{co:interp}.
\end{proof}

\begin{cor}
If an infinite structure $\MA$ is bi-interpretable with $\MZ$ then $\MA$ is prime in $Th(\MA)$, hence it is atomic and homogeneous.
\end{cor}

The following result is often useful when proving that a finitely generated structure is homogeneous. To state it we need to recall a few notions and definitions. 

An $L$-structure $\MA$ is called \emph{Hopfian} if every epimorphism $\MA \to \MA$ is an automorphism. $\MA$ is \emph{equationally Noetherian} if for every $n \in \MN$ every system of equations in variables $\{x_1, \ldots,x_n\}$ and coefficients from $\MA$ is equivalent over $\MA$ to a finite subsystem.
A finitely generated structure $\MA$ has \emph{definable $n$-generation} for a given $n \in \MN$ if there is a first-order $L$-formula $\phi(x_1, \ldots,x_n)$ such that there is a generating $n$-tuple $\bar{a} = (a_1, \ldots,a_n)$ of $\MA$ such that $\MA \models \phi(a_1, \ldots,a_n)$ and for any $n$-tuple $\bar{b} = (b_1, \ldots,b_n)$  of elements of $\MA$ if $\MA \models \phi(b_1, \ldots,b_n)$ then the tuple $b$ generates $\MA$.

\begin{theorem}
Let $\MA$ be a finitely generated $L$-structure  which satisfies the following conditions:
\begin{itemize}
    \item $\MA$ is either finitely presented or equationally Noetherian,
    \item $\MA$ is Hopfian,
    \item $\MA$ has definable generation for some $n \in \MN$.
\end{itemize}
Then $\MA$ is prime. 
\end{theorem}
\begin{proof}
We will prove the theorem for groups, leaving the general case to the reader. 

The structure $\MA$ has definable $n$-generation for some $n \in \MN$, so there is a first-order $L$-formula $\phi(x_1, \ldots,x_n)$ and a generating $n$-tuple $\bar{a} = (a_1, \ldots,a_n)$ in $\MA$ such that $\MA \models \phi(\bar{a})$ and for any $n$-tuple $\bar{b} = (b_1, \ldots,b_n)$  of elements of $\MA$ if $\MA \models \phi(\bar{b})$, then the tuple $\bar{b}$ generates $\MA$.
 Now we describe a first-order $L$-formula $R(x_1, \ldots,x_n)$ in variables $\bar{x}= (x_1, \ldots,x_n)$.
 
 Case 1. Suppose  $\MA$ is finitely presented and 
 $$
 \MA = \langle a_1, \ldots,a_n \mid r_1 = s_1, \ldots, r_k = s_k \rangle ,
 $$
where $r_i, s_i$ are $L$-terms in $a_1, \ldots,a_n$, is a finite presentation.  Then we put
\begin{equation} \label{eq:R(X)}
   R(\bar{x}) = \bigwedge_{i = 1}^m (r_i(\bar{x}) = s_i(\bar{x})). 
\end{equation}

Case 2. Suppose $\MA$ is equationally Noetherian. Consider an $L$-system $S$ of all equations $r(x_1, \ldots,x_n) = s(x_1, \ldots,x_n)$  without coefficients which are true for $x_1 = a_1, \ldots,x_n = a_n$.  Since $\MA$ is equationally Noetherian the system $S$ is equivalent over $\MA$ to a finite subsystem $r_1 = s_1, \ldots, r_k = s_k$. Again we put $R(\bar{x})$ as in (\ref{eq:R(X)}). Clearly, in both cases the tuple $\bar{a}$ satisfies $R(\bar{x})$ in $\MA$. 
 
 Let $c = (c_1, \ldots,c_m)$ be an arbitrary tuple of elements of $\MA$. Then there are $L$-terms $t_1, \ldots, t_m$ in variables $\bar{x} =(x_1, \ldots,x_n)$ such that $c_i = t_i(\bar{a}), i = 1, \ldots,m.$ Then the tuple $c$ satisfies the following first-order formula
$$
\psi_c(y_1, \ldots y_m) = \exists \bar{x} \left(  \bigwedge_{i = 1}^m (y_i = t_i(\bar{x})) \wedge \phi(\bar{x}) \wedge R(\bar{x})\right).
$$
Suppose now that a tuple $d = (d_1, \ldots,d_m)$ satisfies the formula $\psi_c$ in $\MA$. Then there is a tuple $\bar{b}$  in $\MA$ such that $\MA \models \phi(\bar{b})$,  $\MA \models R(\bar{b})$, and $d_i = t_i(\bar{b}), i = 1, \ldots,m$. It follows that $\bar{b}$ is a generating tuple of $\MA$, which satisfies all the relations in the formula $R(\bar{x})$. It also follows that the map $a_1  \to b_1, \ldots,a_n \to b_n$ extends to a homomorphism $\alpha\colon  \MA \to \MA$ which is onto. Since the structure $\MA$ is Hopfian the homomorphism $\alpha$ is an automorphism of $\MA$ that maps $c$ to $d$. This shows that $\MA$ is homogeneous. Moreover, we claim that the type $tp(c)$ of $c$ in $\MA$ is principal (isolated) and is generated by the formula $\psi_c$. Indeed, if not, then there is a model $\MB$ of the first-order theory of $\MA$ which omits $tp(c)$, so there is a formula $\theta(x_1, \ldots,x_m) \in tp(c)$ and a tuple $e$ in $\MB$ such that $\MB \models \psi_c(e) \wedge \neg \theta(e)$. Then the formula $\exists x_1, \ldots,x_m (\psi_c(x_1,\ldots, x_m) \wedge \neg \theta(x_1,\ldots, x_m))$ holds in $\MB$, hence in $\MA$. Therefore, there is a tuple $e'$ in $\MA$ such that $\MA \models  \psi_c(e')\wedge \neg \theta(e')$. This implies that there is an automorphism $\alpha \in Aut(\MA)$ such that $\alpha(c) = e'$~--- contradiction, because $c$ satisfies the formula $\theta$ in $\MA$, but $e'$ does not. This proves that the type $tp(c)$ is principal. Since $c$ is an arbitrary tuple in $\MA$ the structure $\MA$ is atomic, hence prime.

This proves the theorem

\end{proof}

\subsection{First-order rigidity and quasi-finite axiomatizability}

\begin{definition}
A finitely generated $L$-structure $\MA$ is called \emph{first-order rigid} if for any finitely generated $L$-structure $\MB$ first-order equivalence $\MA \equiv \MB$ implies isomorphism $\MA \simeq \MB$.
\end{definition}

\begin{definition}	Fix a finite signature. An infinite finitely generated  structure is Quasi Finitely Axiomatizable (QFA)\index{QFA model} if there exists a first-order sentence $\phi$ of the signature such that 
	\begin{itemize}
		\item $\MA\models \phi$
		\item If $\MB$ is a finitely generated structure in the same signature and $\MB\models \phi$ then $\MA\cong \MB$.
\end{itemize} \end{definition}

Recall that if $X$ is a subset of an $L$-structure $\MA$ then the \emph{diagram} of $X$ in $\MA$ (denoted $D_X(\MA)$) is the set of all atomic sentences in the language $L$ with  constants from $X$ and their negations  that hold in $\MA$. If $G$ is a group and $X \subseteq G$ then $D_X(G)$ describes the multiplication table of the subgroup of $G$ generated by $X$. 

A countable $L$-structure $\MA = \langle A,L\rangle$ is called \emph{arithmetic}  if the set of the Godel's numbers of the  diagram $D_A(\MA)$  of $\MA$   is an arithmetic subset  of $\MN$. If the structure $\MA$ is generated by a finite  set $S$ then it suffices to consider effective enumerations of $\MA$ via the terms of the language $L$ with constants from $S$, in this case $\MA$ is arithmetic if the diagram $D_S(\MA)$ is an arithmetic subset of $\MN$. In particular, a  group $G$ generated by a finite set $S$  is arithmetic if and only if the word problem $W(G,S)$  of $G$ relative to the set $S$ is arithmetic. Here $W(G,S)$ is the  set of  all group words in the generators $S$ which are equal to 1 in $G$. $W(G,S)$ is arithmetic if there is an effective enumeration $\nu\colon  M_X \to \MN$ of all words in the alphabet  $S \cup S^{-1}$  such that the set $\nu(W(G,S))$ is arithmetic.  

\begin{lemma}
Let $\MA$ be a structure interpretable in $\MZ$ or $\MN$. Then $\MA$ is arithmetic.
\end{lemma}
\begin{proof}
Let $\MA = \Gamma(\MN,\bar p)$. Then there is $k\in \MN$,  a definable in $\MN$ subset $A^\ast  \subseteq  N^k$ and a definable in $\MN$ equivalence relation $=_\Gamma$ such that  $A^\ast /=_\Gamma$ is the universe of $\Gamma(\MN,\bar p)$. One can effectively  enumerate tuples in $N^k$, then effectively enumerate tuples in $A^\ast $,  and then effectively enumerate some set of representatives of the equivalence classes of $A^\ast /=_\Gamma$. Since the basic operations and predicates in $\Gamma(\MN,\bar p)$ are definable in $\MN$ the diagram of $\Gamma(\MN,\bar p)$ with respect to the enumeration of chosen set of representatives is arithmetic in $\MN$. This proves the lemma.
\end{proof}

\begin{theorem} \label{th:QFA}  Let $\MA$ be a  rich  structure in a finite signature. If $\MA$ is generated by a finite set $S$ with an arithmetic diagram $D_S(\MA)$  then $\MA$ is QFA. \end{theorem}

We prove the theorem in three steps, starting with the following lemmas.

 Let $\MA = \langle A;L \rangle$ be an  $L$-structure in a finite language $L$. Recall that by $S(\MA)$ we denote the two-sorted structure 
$$
S(\MA) = \langle \MA, S(A); \frown, \in\rangle,
$$
where $S(A)$ is the set of all finite sequences (tuples) of elements from $A$, 
  $\frown$ is the binary operation of  concatenation of two sequences from $S(A)$ and $a \in s$ for $a \in A, s \in S(A)$ means that $a$ is  a component of the tuple $s$.  
  The superstructure $S(\MA,\MN)$ is defined as  the three-sorted structure 
 $$
 S(\MA,\MN) = \langle \MA, S(A),\MN; t(s,i,a), l(s), \frown, \in \rangle,
 $$
 where $\N = \langle N, +,\cdot, 0,1\rangle$  is the standard arithmetic, $l\colon S(A) \to N$ is the length function, i.e., $l(s)$ is the length of the tuple $s$ and $t(x,y,z)$ is a predicate on $S(A)\times N \times  A$ such that $t(s,i,a)$ holds in $S(\MA,\MN)$ if and only if $s = (s_1, \ldots,s_{n})\in S(A), i \in N, 1\leq i \leq n$, and $a = s_i \in A$. By $L_{list}$ we denote the language of the three-sorted structure $S(\MA,\MN)$ (one can represent $S(\MA,\MN)$ as a standard one-sorted structure, taking union of $A, S(A)$ and $N$ as a new  universe and introducing unary predicates defining $A, S(A)$ and $N$  in the union). Observe that the operation $\frown$ and the predicate $\in$ are  $0$-definable in $S(\MA,\N)$ (with the use of $t(s,i,a)$), so sometimes we omit them from the language. Now, we are ready to state a lemma.
 
 \begin{lemma} \label{le:S(A,N)-1}
 let $L$ be a finite language, $\MA$ an $L$-structure, and $S(\MA,\MN)$ the list superstructure over $\MA$.
 Then there is a sentence $\Sigma$ in the language $L_{list}$ such that $S(\MA,\MN) \models \Sigma$, and for any three-sorted $L_{list}$-structure $T = \langle A_T,S_T,N_T\rangle$ if $T \models \Sigma$, then the following holds:
 \begin{itemize}
     \item [1)] $A_T$ is an $L$-structure.
     \item [2)] $N_T$ is a model of arithmetic where $\MN$ is the initial segment of the ordered set $\langle N_T;\leq\rangle$, with respect to the standard ordering $\leq $ of $N_T$. In this case $\MN$ is an elementary substructure of $N_T$. 
     
     \item [3)] Let $S_\omega = \{s \in S_T \mid l(s) \in \MN \}$. Then $\langle S_\omega; \frown, \in\rangle$ is a substructure of $S_T = \langle S_T; \frown, \in\rangle$, which is isomorphic to $ \langle S(A_T); \frown, \in\rangle$. 
     \item [4)] $\langle A_T,S_\omega,\MN\rangle$ is a substructure of $T = \langle A_T,S_T,N_T\rangle$.   
 \end{itemize}
 \end{lemma}
 \begin{proof}

One can take $\Sigma$ as a conjunction of the following conditions:
\begin{itemize}
    \item [1)]  The symbols in $L_{list}$ that correspond to operations in $L$, $S(\MA)$, $\MN$, and $(\MA,\MN)$ define in $T$ correspondingly, operations on $A_T$, $S_T$, $N_T$, and $T = \langle A_T,S_T,N_T\rangle$. Since there only finitely many such operations this condition can be written by a first-order sentence in $L_{list}$.

\item [2)] Since Robinson arithmetic $\mathcal Q$ is finitely axiomatizable, one can write a sentence in arithmetic (hence in $L_{list}$) that makes sure that $\langle N_T;+,\cdot, 1\rangle $ is a model of $Th(\MN)$. It is known that in this case $\MN$ is an in initial segment of $\langle N_T;\leq\rangle$ and it is an elementary substructure of $N_T$. 

\item [3)] If  $S_\omega = \{s \in S_T \mid l(s) \in \MN \}$,  then $\langle S_\omega; \frown, \in\rangle$ is a substructure of $S_T = \langle S_T; \frown, \in\rangle$, since concatenation of two tuples with length in $\MN$ is again a tuple of length in $\MN$. Now we need to write down by an $L_{list}$-sentence some conditions that ensure that the structure $S_\omega$ with operations $\frown$ and $\in$ induced from $S_T$ is isomorphic to the structure $S(A_T); \frown,\in\rangle$. 
   \begin{itemize}
   \item [3.1)] The first condition makes sure that the every element $s \in S_T$ is uniquely determined by it's length  $l(s)$ and all its components:
   $$
   \forall s, r \in S_T \left ( l(s) = l(r) \wedge \forall i \in N_T, \forall x \in A_T (t(s,i,x) \leftrightarrow t(r,i,x))    \to s= r  \right)
   $$ 
    
    \item [3.2)] The second condition states that for any $n \in \MN$, for any $1\leq i \leq n$, and for any $a_i \in A_T$ there is $s \in S_T$ with $l(s) = n$ and for which $t(s,i,a_i)$ holds for any such $i$. 
    To do this it suffices to write down that for any $a \in A_T$ there is $s_a$ of length 1 ($l(s) = 1$) such that $a$ is the only component of $s_a$, i.e., $t(s_a,1,a)$ holds. This ensures that the tuple $s$ with the properties described above exists in $S_T$. Indeed, the tuple $s$ has finite length hence it is a finite concatenation of the   1-tuples $s_{a_1}, \ldots, s_{a_n} $. It is left to write down a sentence which expresses that concatenation of any two tuples $s, r \in S_T$ exists and gives precisely what it is supposed to. Namely, for any $r,s \in S_T$ there is $u \in S_T$ such that  for any $x \in A_T$ and any $i\leq l(r)$ one has $t(r,i,x) \leftrightarrow t(u,i,x)$ (this condition makes sure that $r$ is a prefix of $u$) and for any $j \leq l(s)$  one has $t(s,j,x) \leftrightarrow t(u,l(r)+j,x)$. 
\end{itemize}

  We claim that if the conditions 3.1) and 3.2) are satisfied then there is an isomorphism    $\lambda\colon  \langle S_\omega; \frown, \in\rangle \to \langle S(A_T); \frown, \in\rangle$. For $s \in S_\omega$ put $\lambda(s)$ to be the unique tuple $(a_1, \ldots,a_{l(s))}$ such that for any $i\leq l(s)$ the condition $t(s,i,a_i)$ holds.  Then 3.1) tells one that $\lambda$ is injective, while 3.2) makes sure that $\lambda$ is onto.  Obviously, $\lambda$ preserves the operation $\frown$ and the predicate $\in$.  
 
 \end{itemize}
 
 Now 4) follows from 1)-3). Indeed, by construction $\langle A_T,S_\omega,\MN\rangle$ is a substructure of $T = \langle A_T,S_T,N_T\rangle$.  
 \end{proof}

 Let $X = \{x_i \mid i \in \MN\}$ be a countable  set of variables, and $T_L(X)$ be the set of all terms in the language $L$ in variables $X$. If the language $L$ is finite or countable, then there is a function $\nu\colon T_L(X) \to \MN$ such that $\nu$ is injective (we may also assume that $\nu$ is surjective, if needed), the subset $\nu(T_L(X))$ of $\MN$ is computable, and the functions $\nu$ and its inverse $\nu^{-1}\colon  \nu(T_L(X)) \to T_L(X)$ are computable.

 \begin{lemma}\label{le:S(A,N)-2}
 Let $L$ be a language with finite signature,  $\MA$ a finitely generated $L$-structure, and $S(\MA,\MN)$ its list superstructure. Then the following holds:
 \begin{itemize}
      \item [1)] 
 for any $n \in \MN$ there is a formula $\Phi_n(x_1,\ldots,x_n,y,z)$ in the language $L_{list}$ (where $x_1,\ldots,x_n,y$ are variables of the sort $\MA$ and $z$ is a variable of the sort $\MN$) such that for any  $g_1, \ldots,g_n,h \in \MA$ and $m \in \MN$  
 $$
 S(\MA,\MN) \models \Phi_n(g_1, \ldots,g_n,h,m)
 $$ 
 if and only if there exists a term $t = t(x_1, \ldots,x_n)$ in the language $L$ such that $h = t(g_1,\ldots,g_n)$ and  $\nu(t) = m$ with respect to some fixed effective enumeration $\nu\colon  T_L(X) \to \MN$ of all terms $t$ in the language $L$.
 \item [2)] There is a formula $\Psi(u,y,z)$ in the language of $L_{list}$ (where $u$ is a variable of the sort $S(\MA)$, $y$ is a  variable of the sort $\MA$ and $z$ is a variable of the sort $\MN$) such that for any $s = (a_1, \ldots,a_n) \in S(\MA)$, $h \in \MA$, $m \in \MN$ 
 $$
 S(\MA,\MN) \models \Psi(s,h,m)
 $$ 
 if and only if there exists a term $t = t(x_1, \ldots,x_n)$ in the language $L$ such that $h = t(a_1,\ldots,a_n)$ and  $\nu(t) = m$ with respect to some fixed effective enumeration $\nu\colon  T_L(X) \to \MN$ of all terms $t$ in the language $L$.
 \end{itemize}
 \end{lemma}
 \begin{proof}
 We prove this lemma for the language $L$ of group theory (the same argument works for the language of rings), leaving the general case to the reader. Thus, $\MA$ is a group.
 We prove 2) first and then derive 1) from 2).
 
Let $X = \{x_i\mid i \in \MN\}$ be an infinite countable set of variables. $X^{-1} = \{x_i^{-1} \mid x \in X\}$ the set of formal inverses of elements of $X$ and $M_X$ the set of all finite words  in the alphabet $X^{\pm 1} = X\cup X^{-1}$ (called \emph{group words} in $X$) viewed as a free monoid on  $X\cup X^{-1}$, where  multiplication is concatenation. A word $w \in M_X$ can be represented as $w = x_{k_1}^{\varepsilon_1} \ldots x_{k_\ell}^{\varepsilon_\ell}$, where $x_{k_i} \in X^{\pm 1}$, $\varepsilon_i \in \{-1,1\}$, $\ell \in \MN$.  With the word $w$ we associate  a unique  tuple $m_w \in \MN^{2\ell}$:
 $$
 m_w = (k_1,\varepsilon_1+1,k_2,\varepsilon_2+1, \ldots, k_\ell,\varepsilon_\ell+1)
 $$
 (here we use $\varepsilon_i+1$ instead of $\varepsilon_i$ to have numbers in $\MN$).
 
 Let $\nu\colon \bigcup_{\ell > 0}\MN^\ell \to \MN$ be a bijection such that $\nu$ and $\nu^{-1}$ are computable. Here $\nu^{-1}$ is computable when for every $m \in \MN$ the length $\ell = |\nu^{-1}(m)|$ of the tuple $\nu^{-1}(m) = (a_1, \ldots,a_\ell)$ is computable and there is a computable function $f(m,i)$ such that for any $m \in \MN$ and any $i, 1\leq i\leq \ell$ one has $f(m,i) = a_i$. There are many such bijections, we fix one of them, defined by 
 $$
 \nu(a_1, \ldots,a_\ell) = 2^{a_1} + 2^{a_1+a_2+1} + \ldots + 2^{a_1+\ldots +a_\ell+\ell-1} -1.
 $$
 To find $\nu^{-1}(m)$ for a given $m \in \MN$ one can use the unique binary representation of $m-1$ and find unique integers $0 \leq b_1 < b_2 < \ldots <b_\ell$ such that $m-1 = 2^{b_1}+ \ldots 2^{b_\ell}$ and then find the numbers $a_1, \ldots,a_\ell$. 
 
 The subset $W = \{ \nu(m_w)\mid w \in M_X\}$ is definable in $\MN$ by a first-order formula, since the functions $m \to |\nu^{-1}(m)|$ and $f(m,i)$ above are computable, hence definable in arithmetic. Observe that the map $w\to \nu(m_w)$ gives an effective enumeration of all group words in $M_X$.
 
 Fix $n \in \MN$ and put $X_n = \{x_1, \ldots, x_n\} \subset X$. Let $M_{X_n}$ be the subset (submonoid) of $M_X$, consisting of all group words in variables $X_n^{\pm 1}$. Then the subset $W_n =\{\nu(m_w) \mid w \in M_{X_n}\}$ is also definable in $\MN$. Indeed, it suffices to add a conjunct to the formula defining $W$ that says that all odd components of $m_w$ are less or equal to $n$. 
 
With a tuple  $r = (r_1, \ldots,r_n) \in S(\MA)$ of length $n$ and a word  $w  = x_{k_1}^{\varepsilon_1} \ldots x_{k_\ell}^{\varepsilon_\ell} \in M_{X_n}$ (here $x_{k_j} \in X_n$ for any $j$) we associate a unique tuple $s_{r,w} \in S(\MA)$ defined as follows: 
 
 $$
 s_{r,w} = (r_{k_1}^{\varepsilon_1},r_{k_1}^{\varepsilon_1}r_{k_2}^{\varepsilon_2}, \ldots, r_{k_1}^{\varepsilon_1} \ldots r_{k_\ell}^{\varepsilon_\ell})
 $$
 so the last component of $s_{r,w}$ is equal to $w(r_1, \ldots,r_n) \in \MA$.
 
 \medskip 
 {\it Claim. The subset $P = \{(r,s_{r,w},\nu(m_w)) \mid r \in S(\MA),  w \in M_{X_{|r|}} \}$ of $S(\MA)\times S(\MA) \times  \MN$ is  definable in $S(\MA,\MN)$. }
 
 \medskip
 Indeed, a triple  $(r,s,m) \in S(\MA)\times S(\MA)\times \MN$ belongs to $P$ if the following conditions hold:
 
 \medskip
 1) $|\nu^{-1}(m)| = 2|s|$ and $m \in W_{|r|}$, i.e., $m = \nu(m_w)$ for some $w \in M_{X_{|r|}}$.

\medskip
2) Let  $r = (r_1, \ldots, r_n)$, $s = (s_1, \ldots, s_\ell)$ and $\nu^{-1}(m) =  (a_1,b_1, \ldots,a_\ell,b_\ell)$. Then for all $i$ such that $1\leq i <\ell$ the following two conditions hold: 
$$
\bigwedge_{k = 1}^n [(a_{i+1}= k \wedge b_{i+1} = 2) \to s_{i+1} = s_ir_k]  
$$
 and 
$$
\bigwedge_{k = 1}^n [(a_{i+1}= k \wedge b_{i+1} = 0) \to s_{i+1} = s_ir_k^{-1}] , 
$$ 
 where we assume that $s_0= 1$.
 Observe, that both conditions 1) and 2) can be described by a formula in the language $L_{list}$ in $r,s$ and $m$. Indeed, using the predicate $t(s,i,m)$ of $S(\MA,\MN)$ one can express the conditions that $r_i, a_i,b_i$ and $s_i$ are the corresponding components of $r$ and $s$.  Let $\Phi_{P}(u,v,z)$ be  conjunction of the formulas 
 that describe 1) and 2) (here the variable $u$ correspond to $r$, the variable $v$ correspond to $s$ and $z$ corresponds to $m$). Then $\Phi_{P}(u,v,z)$ defines $P$ in $S(\MA,\MN)$ and the claim follows.
 
 Consider a formula 
 $$
 \Psi_P(u,y,z) = \exists v  (\Phi_P(u,v,z)  \wedge t(v,l(v),y))
 $$
 language $L_{list}$, here $t$ and $l(v)$ are the predicate  and the function from $L_{list}$, $u$  is a variable of the sort $S(\MA)$, $y$ is a variable of  the sort $\MA$, and $z$ is a variable of the sort $\MN$.   Note that if $r \in S(\MA)$, $h \in \MA$, and $m \in \MN$ are such that $S(\MA) \models \Psi_P(r,h,m)$  then there is a tuple $v \in S(\MA)$ such that $(u,v,m) \in P$ and  $h$ is the last component of the tuple $v$.  Therefore, $r = (r_1, \ldots,r_n)$ for some $n \in \MN$, $m = \nu(m_w)$ for some $w \in M_{X_n}$,  $s = s_{r,w}$, and  $h = w(r_1, \ldots,r_n)$. As we mentioned above the map $w \to \nu(m_w)$ gives an effective enumeration of all group words in variables from $X$.

 This proves the statement 2) of the lemma.
 
 To show 1) fix $n \in \MN$ and consider arbitrary elements $g_1, \ldots,g_n \in \MA$. The condition
 $$
 (r \in S(\MA))  \wedge (|r| = n) \wedge (r= (g_1, \ldots,g_n))
 $$
 can be described by a formula, say  $\Psi_n(x_1, \ldots,x_n,u)$, where 
 the variables $x_1, \ldots,x_n$ correspond to the elements $g_1, \ldots,g_n$ and $u$ corresponds to $r$. It follows that the formula
 $$
 \Phi_n(x_1, \ldots,x_n,y,z) = \exists u (\Psi_n(x_1, \ldots,x_n,u) \wedge \Psi(u,y,z))
 $$
 satisfies all the requirements in the statement 1).
 \end{proof}

 \medskip
 {\it Proof of Theorem \ref{th:QFA}}. We prove the theorem in the case of groups. Let $G$ be a  rich finitely generated group, so $G$ is regularly bi-interpretable with its list superstructure $S(G,\MN)$, say $S(G,\MN) \simeq \Gamma(G,\delta)$ and $G \simeq \Delta(S(G,\MN),\sigma)$. Assume $G$ is generated by elements $g_1, \ldots,g_n$.  
 
 Let  $\Psi(u,y,z)$ be the formula from Lemma \ref{le:S(A,N)-2}. 
 For arbitrary tuple $r = (a_1, \ldots,a_k) \in S(G)$, an element $ h \in G$, and a number $m \in \MN$ consider the  following formula 
 $$
 \Psi^{min}(r, h,m) = \Psi(r, h,m) \wedge \forall p \in \MN (|\nu^{-1}(p)| < |\nu^{-1}(m)| \to \neg \Psi(r,h,p)),
 $$
where $\nu$ is the enumeration constructed in Lemma \ref{le:S(A,N)-2} and $|\nu^{-1}(p)|, |\nu^{-1}(m)|$ are the lengths of the tuples $\nu^{-1}(p)$ and $\nu^{-1}(m)$.  By Lemma \ref{le:S(A,N)-2} $\Psi^{min}(r, h,m)$  holds in $S(G,\MN)$  if and only if $h = w(a_1,\ldots,a_k)$ for some word $w \in M_X$  and $w$ has minimal possible length among all such $w$.  We denote the length of such $w$ by $|h|_r$ (the word length of $h$ with respect to the set $ A_r = \{a_1, \ldots,a_k\}$).
 
 Note, that if the subgroup $\langle A_r \rangle$ generated by $A_r$ is infinite then the set of natural numbers $\{|h|_r \mid h \in \langle A_r \rangle\}$ is unbounded. Hence the following formula holds in $S(G,\MN)$:
 
 $$
 \Theta(r) = \forall p \in \MN,  \exists m \in \MN, \exists h \in G (\Psi^{min}(r,h,m) \wedge p \leq  |\nu^{-1}(m)|).
 $$
 It may be that for a given tuple $r \in S(G)$ the subgroup generated by the set  $A_r$ is finite. However, since the group $G$ is infinite and elements $g_1, \ldots,g_n$ generate $G$ it is true that for any tuple $r \in S(G)$ there is a tuple $s \in S(G)$ of length $n$  (for example, the tuple $\bar g = (g_1, \ldots,g_n)$) such that the tuple $r\frown s$  generates an infinite subgroup. One can write down this by the following sentence in the language $L_{list}$:
 $$
 {\cal F} = \forall r \in S(G) \exists s \in S(G) (|s| = n \wedge  \Theta(r\frown s)), 
 $$
 where $n$ is the constant term in $\MN$ (we view $n$ as a sum $1+\ldots +1$ of $n$ units 1's).
 Since $G$ and $S(G,\MN)$ are bi-interpretable there is a first-order  sentence  ${\cal F}^\Gamma$ in group language such that 
 $$
 G \models {\cal F}^\Gamma \Longleftrightarrow S(G,\MN) \models {\cal F}. 
 $$
 
 Suppose now that $H$ is an arbitrary group finitely generated by some elements $h_1, \ldots,h_m$.
 
 Consider the sentence $\Sigma$ in the language $L_{list}$ from Lemma \ref{le:S(A,N)-1} and a sentence $\Phi_{group}$  in the language $L_{list}$ that states that the first sort $G$ in $S(G,\MN)$ is a group. Let  $\Sigma^\Gamma$ and $\Phi_{group}^\Gamma$ be the  $\Gamma$-translations of $\Sigma$ and $\Phi_{group}$ into  first-order group language sentences, so $G \models \Sigma^\Gamma$  and $G \models \Phi_{group}^\Gamma$. Then the following sentence holds in $G$:
 $$
 G \models {\cal F}^\Gamma\wedge \Sigma^\Gamma  \wedge \Phi_{group}^\Gamma.
 $$
Assume that the group $H$ satisfies the sentence ${\cal F}^\Gamma\wedge \Sigma^\Gamma  \wedge \Phi_{group}^\Gamma$.  
 By the translation properties the code $\Gamma$ interprets in $H$ a three-sorted model $T = \langle A_H,S_H,N_H\rangle$  which satisfies the sentences  $\Sigma, {\cal F},$ and $\Phi_{group}$. By Lemma \ref{le:S(A,N)-1} the model $T$ satisfies the conditions 1)-4) of the lemma. Hence $A_T$ is a group isomorphic to $H$; $N_H$ is a model of arithmetic with the standard arithmetic $\MN$ as its initial segment (with respect to the standard ordering $\leq $ on $N_H$); $S_\omega = \{s \in S_H \mid l(s) \in \MN \}$ with induced $\frown$ and $\in$ from $S_H$ is a substructure of $S_H$  isomorphic to $S(H)$. Since ${\cal F}$ holds in $T$ it follows that for the tuple $\bar h = (h_1, \ldots,h_m)$ there exists a tuple $(u_1, \ldots,u_n)$ such that $\Theta(\bar h \frown \bar u)$ holds in $T$. Since $\bar h$ generates $H$ it follows that for any $f \in H$ there is a  tuple $s \in S_H$  of length precisely $n$  (so, $s \in S_\omega$) such that the formula $\Psi(\bar h \frown s,h,\ell)$ holds in $T$ for some $\ell$ from the initial segment $\MN$ of $N_H$. Moreover we can assume (since $\Psi^{min}(\bar h \frown s,h,\ell)$ holds in $T$) that such $\ell$ is minimal for $h$  among all elements in $N_H$. Now the formula $\Theta(\bar h \frown s)$ holds in $T$, so for every $m \in N_H$ there exists $\ell \in \MN$ such that $m \leq \ell$. It follows that $N_H = \MN$. Hence $S_\omega = S(H)$. Therefore, $T\simeq S(H,\MN)$.  Thus, the formulas $\Gamma$ and $\Delta$ bi-interpret $H$ and $S(H,\MN)$. 
 
 To finish the proof we need to show that there is a sentence $Iso_G$ of group theory such that if a finitely generated group $H$ satisfies the finitely many sentences  mentioned above and the sentence $Iso_G$ then $G \simeq H$. 
 Let $W(G,{\bar g})$   be the word problem of $G$  with respect to the generating tuple $\bar g$, i.e, the set of all group words $w(x_1, \ldots,x_n)$ such that $w(g_1, \ldots,g_n) = 1$ in $G$. Since $G$ is arithmetically definable the set   $W(G,{\bar g})$ is an arithmetic set, i.e., the set $\nu(W(G,\bar g))$  is defined in $\MN$ by some formula $D(x)$. 
 
 Consider  the following formula 
 $$
 Gen(r) = \forall y \in G \exists z \in \MN  \Psi(r, y,z).
 $$
 By Lemma \ref{le:S(A,N)-2}  the formula $Gen(r)$ holds in $S(G,\MN)$ on a tuple $r = (r_1, \ldots,r_k)$
  if and only if elements  $r_1, \ldots,r_k$ generate $G$ (we will say in this case that the tuple $r$ generates $G$). The formula 
    $$
    WP_D(r) = \forall h \in G \forall m \in \MN (\Psi(r,h,m) \to (h=1 \leftrightarrow D(m))) 
    $$
 holds in $S(G,\MN)$ on a tuple $r$ if and only if the tuple $r$ generates a subgroup in $G$ whose word problem is definable in arithmetic by the formula $D(x)$. 
 Now the sentence 
 $$
 Iso_G = \exists r (Gen(r)\wedge WP_D(r) \wedge (|r| = n))
 $$
 in the language $L_{list}$  states that there is a tuple $r$ in $S(G)$ that has length $n$, generates $G$, and that the word problem of  $G$ with respect to the generating tuple $r$ is defined in arithmetic by the formula $D(x)$. Note that $Iso_G$ holds in $S(G,\MN)$ since $(Gen(r)\wedge WP_D(r) \wedge (|r| = n))$ holds in $S(G,\MN)$ on the tuple of generators $\bar g$. 
 
 Let $Iso_G^\Gamma$ be the translation of the sentence $Iso_G$ into the group theory sentence with respect to the interpretation $S(G,\MN)  \simeq \Gamma(G,\delta)$.
 Then $Iso_G^\Gamma$ is true in $G$. 
 Suppose a  finitely generated group $H$ satisfies the sentences ${\cal F}^\Gamma$, $\Sigma^\Gamma$, $ \Phi_{group}^\Gamma$, and $Iso_G^\Gamma$. Then, as we showed above, $H$ and $S(H,\MN)$ are bi-interpretable in each other by the same codes $(\Gamma,\delta)$
 and $\Delta,\sigma)$  as $G$ and $S(G,\MN)$. It follows that the sentence $Iso_G$ holds in $S(H,\MN)$, which implies that there is a finite generating set $(f_1, \ldots,f_n)$ with the same word problem as the generating set $(g_1, \ldots,g_n)$ in $G$. The map $g_1 \to f_1, \ldots,g_n \to f_n$ gives rises to an isomorphism $G \to H$. This proves the theorem. 
 
\hfill  $\Box$

 The following well known theorem now follows from Theorem \ref{th:QFA}.
 \begin{theorem}[\cite{Khelif}, \cite{Nies2007}, Theorem 7.14]\label{Nies2007} If $\MA$ is a finitely generated structure with finite signature which is bi-interpretable (possibly with parameters) with the ring  $\Z$ then $\MA$ is QFA. \end{theorem}

\section{Properties definable in rich groups}
\label{wsol}

In this section we describe various properties that are definable in in first-order logic in rich groups. 

The general strategy is to describe a property by a single formula in WSOL and then transform this formula into a first-order one  using richness of the group.  We focus mostly on groups though many  results hold for general structures.

\subsection{Malcev's problem}

 The generalized \emph{Mal'cev's problem} for a given group $G$ asks to describe  subgroups which are first-order  definable in $G$.  

In some groups subgroups are mostly undefinable, for example, in a free non-abelian group $F$ a proper subgroup is definable (with parameters) if and only if it is cyclic \cite{KMdef}, \cite{PPST}.  

On the other hand, in rich groups not only all finitely generated subgroups are definable, but they are definable in a very uniform way. Below for elements $a_1, \ldots,a_n$ in a group $G$ by $\langle a_1, \ldots,a_n \rangle $ we denote the subgroup generated by $a_1, \ldots,a_n$ in  $G$.

\begin{definition}
We say that finitely generated subgroups of $G$ are {\it uniformly definable} if for any natural number $n $  there exists a  first-order formula 
$\phi_n(x_1, \ldots,x_n,y)$ such that for any $a_1, \ldots,a_n, g  \in G$ the formula 
$\phi_n(a_1, \ldots,a_n,y)$ defines the subgroup $\langle a_1, \ldots,a_n\rangle $  in $G$, i.e.,  for any $a_1, \ldots,a_n, g \in G$ 
$$
G \models \phi_n(a_1, \ldots,a_n,g) \Longleftrightarrow g \in \langle a_1, \ldots,a_n \rangle.
$$

\end{definition}
\begin{theorem}
In a rich group $G$ finitely generated subgroups are uniformly definable.
\end{theorem}
\begin{proof}
Fix $n \in \mathbb{N}$.  Let $X_n = \{x_1, \ldots,x_n\}$ and  $X_n^{-1} = \{x_1^{-1}, \ldots,x_n^{-1}\}$.  Fix an arbitrary computable enumeration 
$$
w_0, w_1, w_2, \ldots, w_m, \ldots
$$
of  all words in the alphabet $X^{\pm 1}. = X\cup X^{-1}$, so the function $n \to w_m$ is computable (see the  proof of Lemma \ref{le:S(A,N)-2} for an example of such an enumeration). 
Observe, that 
$$
b \in \langle a_1, \ldots,a_n\rangle \Longleftrightarrow \bigvee_{i \in \mathbb{N}} (b  = w_i(a_1, \ldots,a_n)).
$$
The set of formulas $\{ (y  = w_i(x_1, \ldots,x_n)) \mid i \in \mathbb{N}\}$ is computably enumerable, so  the formula 
$$
\psi_n(x_1,\ldots,x_n,y)  = \bigvee_{i \in \mathbb{N}} (y  = w_i(x_1, \ldots,x_n))
$$
is in WSOL.   Since the group $G$ is rich there is a first-order fiormula $\phi_n(x_1,\ldots,x_n,y)$ which is equivalent to  $\psi_n(x_1,\ldots,x_n,y)$ on $G$.  This proves  the theorem.
\end{proof}

\begin{cor}
Let $G$ be a finitely generated rich group. Then for every $n \in \MN$ the set of $n$-generating tuples of $G$ is $0$-definable in $G$.
\end{cor}

\subsection{Properties of subgroups}

We say that a subgroup $H\leq G$ is  {\em malnormal} if $gHg^{-1}\cap H$ is trivial for any $g\in G-H.$

\begin{theorem}\label{1-10}
If $G$ is a rich group, then for any $n $  there exists a  formula 
$\phi_n(x_1, \ldots,x_n)$ such that for any $a_1, \ldots,a_n \in G$ the formula 
$$G \models \phi_n(a_1, \ldots,a_n)$$
if the subgroup generated by $a_1, \ldots,a_n$  is
1) normal,
2) malnormal,
3)free, 
4) finitely presented,
5) residually finite, 
6) amenable, 
 7) has property T,
8) isomorphic to a given finitely generated recursively presented group,
 9)  simple,
10)  solvable.

\end{theorem}
\begin{proof} In all the cases we will find a WSOL formula that express the corresponding property in $G$. Since $G$  is rich, this implies that there exists a first order formula that express the same property.

1) Fix an arbitrary computable enumeration 
$$
w_0, w_1, w_2, \ldots, w_N, \ldots
$$
of  all words in the alphabet $X\cup X^{-1}$ (so the function $n \to w_n$ is computable). The formula $$\forall g \bigwedge _{i=1}^n ( \bigvee _{j\in\N}a_i^g=w_j(a_1,\ldots ,a_n))$$
is in WSOL and states that the subfgroup $\langle a_1,\ldots ,a_n\rangle$ is normal. Since the group is rich the same property can be defined by a first order formula.

2) The following WSOL formula defines malnormality 
$$\forall g\forall x (\bigvee _{i\in\N} x=w_i(a_1,\ldots ,a_n)\implies \neg \bigvee _{i\in\N} x=w_i(a_1^g,\ldots ,a_n^g)).$$

 3) The formula   $\bigvee _{j\in\N}w_j(a_1,\ldots ,a_n)\neq 1$ express that the subgroup $\langle a_1,\ldots ,a_n\rangle$ is free.
 
 4)  Let $\bar x = (x_1, \ldots,x_n)$. Consider the following predicate:
$$
P_n(y,y_1,\bar x) = \exists s \in \langle \bar x\rangle (y = y_1^s \vee y = (y_1^{-1})^s)
$$
We claim that this predicate  is definable in $G$ in WSOL. Indeed, in the notation above, $P_n(y,y_1,\bar x)$ is equivalent in $G$ to the WSOL formula 
$$
\bigvee_{i \in \N}(y = y_1^{w_i(\bar x)} \vee y = (y_1^{-1})^{w_i(\bar x)}).
$$

 For $k \in  \N$  take   an arbitrary tuple $\tau = (t_1,\ldots,t_k) \in \{1,\ldots,m\}^k$.  Let $\bar z  = (z_1, \ldots,z_m)$.
Consider the following formula 
$$
P_{n,k,\tau}(y,\bar x,\bar z) = \exists y_1 \ldots \exists y_k (y = y_1 \ldots y_k \bigwedge_{j= 1}^k P_n(y_j,z_{t_j},\bar x))
$$

Now the subgroup $\langle a_1,\ldots ,a_n\rangle$ has relations $r_1, \ldots,r_m$ if and only if for every trivial element $y$ in $\langle a_1,\ldots ,a_n\rangle$ we have
$P_{n,k,\tau}(y,\bar x, r_1,\ldots ,r_m)$ for some $k,\tau .$ The subgroup is finitely presented if it has relations $r_1, \ldots,r_m$ for some finite set of elements $r_1, \ldots,r_m$.  This can be expressed by the infinite disjunction.

5) A finite index normal subgroup of $\langle a_1,\ldots ,a_n\rangle$  is finitely generated as a normal subgroup (by the multiplication table of the finite quotient). Let $\{r_{i1}, \ldots,r_{im_i}|i\in\N\}$  be the set of all generating sets of finite index normal subgroups
(generators of them as normal subgroups). Now we can write a WSOL formula that 
says that for every non-trivial element $y$ there exist $r_1,\ldots ,r_m$ such that  we have the negation  $\neg P_{n,k,\tau}(y,\bar x, r_1,\ldots ,r_m)$ for any $k,\tau .$  

6)  One has to use the F\o{}lner criterion about the ratio of the cardinality of the boundary to a finite  set.  F\o{}lner criterium can be written as a formula in $L_{WSOL}.$

7) One can  use the  result of Shalom \cite{Sha} that every property (T) group is
a quotient of a finitely presented property (T) group.
The set of quotients of a f.p. group is definable in WSOL. Then we have to take a countable disjunction over all f.p. groups with property T. To show that this disjunction is over an arithmetic set of formulas we show that the class of finite presentations  with property (T) is recursively enumerable. 
By \cite[Theorem3]{KNO}, property (T) is equivalent to the  existence of a solution  for the equation (1) in \cite[Theorem3]{KNO}.  Since equations in $\R$ are decidable and we can enumerate the word problem in the finite presentation, if the solution exists, we eventually find it.  So we can enumerate finite presentations  and for each of them start the procedure that looks for a solution. If the procedure stops in some group we add it to the list of groups with property (T).  A recursively enumerable set of numbers is arithmetic.

8) Recall the following result \cite{Sco},\cite[Theorem 3.5]{Osin}: For every finitely generated group $G$, there exists an $\mathcal L _{\omega _1,\omega}$-sentence $\sigma$  with the property that for any group $H$ we have $H\models\sigma$ if and only if $H\cong G$. 

The proof goes as follows. Let $a_1,\ldots ,a_n$ be a generating set for $G$.  Let $u_1(x_1,\ldots ,x_n),\ldots , u_2(x_1,\ldots ,x_n)\ldots $ (respectively $v_1(x_1,\ldots ,x_n),\ldots , v_2(x_1,\ldots ,x_n)\ldots $) be the enumeration of all group words in the alphabet $\{x_1,\ldots ,x_n\}$ such that $u_i(x_1,\ldots ,x_n)=1$ (respectively 
$v_i(x_1,\ldots ,x_n)\not =1$) in $G$ for all $i$. Then the formula
$$\exists x_1\ldots \exists x_n\left(\forall g \bigvee (g=u_i\vee g=v_j)\wedge \left(\bigwedge _{i\in\N} u_i=1\wedge v_i\neq 1\right)\right )$$ has the required property.
If $G$ is recursively presented then the set of words that are equal to 1 is Diophantine,  therefore its complement is arithmetic and the formula belongs to $L_{WSOL}.$

The proof of 9), 10)  is straightforward.

\end{proof}

\section{Rich existentially closed structures}

 In this section, following Belyaev and Taitslin \cite{BT}, we describe some natural classes of groups, rings and semigroups, where the existentially closed structures are rich.  
 
 Recall that a substructure $\MA$ of a structure $\MB$ is said to be \emph{existentially closed} in $\MB$ if for every quantifier-free formula $\phi(x_1, \ldots,x_n, y_1, \ldots y_m)$ and all elements $b_1, \ldots,b_m \in \MA$  such that $\phi(x_1, \ldots,x_n, b_1, \ldots b_m)$ has a solution in $\MB$ then it also has a solution in $\MA$. 
 A model $\MA$ of a theory $T$ is called \emph{existentially closed} in $T$ if it is existentially closed in every superstructure $\MB$ that is itself a model of $T$.

 \begin{definition} \label{de:BT}
 Let $T$ be a consistent first-order theory in a finite signature $L$.  Denote by $E(T)$ the class of all existentially closed models in $T$.
 We say that $T$ satisfies a condition (BT) if the following holds:
 \begin{itemize}
 \item [1)]   Every model in $E(T)$ is infinite. 
 \item [2)] There is a formula $\theta(x,y,\bar z)$ in the language $L$ such that:
   \begin{itemize}
       \item [a)] for any $\MA = \langle A;L\rangle  \in E(T)$ and any finite subset $H \subset A\times A$ there is a tuple $\bar c $ over $A$ such that 
       $$
       H_{\bar c}  = \{(a,b) \mid \MA \models \theta(a,b,\bar c)\}
       $$
       \item [b)]for any tuple $\bar c$ over $A$ with $|\bar c| = |\bar z|$ the formula $\theta(x,y,\bar c)$ defines in $\MA$ a finite subset of $A \times A$.
   \end{itemize}
   \end{itemize}
 \end{definition}
 
 The following is a crucial result.
 \begin{theorem} \label{th:BT}
 \cite{BT} The following theories satisfy the condition (BT): 
 \begin{itemize}
     \item Theory of groups.
     \item Theory of torsion-free groups.
     \item Theory of semigroups.
     \item Theory of torsion-free semigroups.
    \item Theory of cancellation semigroups. \item Theory of inverse semigroups.
    \item Theory of associative rings.
    \item Theory of division rings (skew-fields). 
    
 \end{itemize}
 \end{theorem}

 \begin{lemma}
 Let a theory $T$ satisfy the condition (BT). Then for every existentially closed model $\MA$ of the theory $T$ the structures $\MA$ and $FBP(\MA)$ are absolutely bi-interpretable with each other. Hence $\MA$ is absolutely and effectively rich.
 \end{lemma}
 \begin{proof}
 Let $\theta)x,y,\bar z)$ be the formula from the condition 2)  of Definition \ref{de:BT} above. Let $|\bar z| = n$. Take a structure $\MA = \langle A;L\rangle  \in E(T)$.  On the set $A^n$ introduce an equivalence relation $\sim_\theta$ such that for $\bar a, \bar b \in A^n$ one has $\bar a \sim_\theta \bar b$ if and only if $H_{\bar a} = H_{\bar b}$, where  $H_{\bar c}  = \{(a,b) \mid \MA \models \theta(a,b,\bar c)\}$.
 Then the two-sorted structure 
 $$
 P = \langle \MA,A^n/\sim_\theta;\theta(x,y,\bar z)\rangle
 $$
 is absolutely interpretable in $\MA$.  Notice that the maps $id\colon \MA \to \MA$ and $\mu\colon  \bar c \in A^n \to H_c$ give rise to an isomorphism $P \to FBP(\MA)$. So $FBP(\MA)$ is absolutely interpretable in $\MA$. Obviously, $\MA$ is a reduct of $FPB(\MA)$ so $\MA$ is absolutely interpretable in $FBP(\MA)$.  Let 
 $$
 P  \rightsquigarrow \MA \rightsquigarrow  FBP(\MA)
 $$
 be the corresponding interpretations. We need to construct an isomorphism $\chi\colon  P \to FBP(\MA)$ definable in $FPB(\MA)$. For  $\bar c \in A^n$  put $\chi(\bar c) = H  \in FPB(A)$  if $H$ and $\bar c$ satisfy the following formula:
 $$
 \forall  a \forall b ((a,b) \in H \leftrightarrow \theta(a,b,\bar c)).
 $$
 And define $\chi$ to be identical on $\MA$. This gives an  isomorphism $\chi\colon  P  \to  FBP(\MA)$ which is definable in $FBP(\MA)$. The identical map $\MA \to \MA$ gives the isomorphism for  the interpretations 
 $$
 \MA \rightsquigarrow  FBP(\MA) \rightsquigarrow \MA.
 $$
 This finishes the proof of the lemma.
 \end{proof}
 
 \begin{cor}
 All existentially closed models of the theories mentioned in Theorem \ref{th:BT} are absolutely and effectively rich.
 \end{cor}

 We will mention some corollaries for groups. Note first that every non-trivial algebraically closed group  is also existentially closed. Here a group $G$ is algebraically closed if it is closed under solutions of finite systems of equation with coefficients in $G$ in any overgroup.  
 
 Properties of algebraically closed groups: \begin{itemize}
 \item [1)] Every countable group can be embedded in a countable algebraically closed group.
\item [2)] Every algebraically closed group is simple.
\item [3)] No algebraically closed group is finitely generated.
\item [4)] No algebraically closed group is  recursively presented.
\item [5)] Any finitely generated group with decidable word problem embeds into any algebraically closed non-trivial group.
 \end{itemize}
 
 For us the following corollary is of interest.
 \begin{cor}
 Every countable group embeds into an absolutely  rich countable group which satisfies the conditions 2)-5) above. 
 \end{cor}

\section{Rich rings} 
A description of finitely generated commutative rings bi-interpretable with $\N$ was given in  \cite{AKNS}. 
Let A be a commutative ring with unit. As usual, we write $Spec(A)$ for
the spectrum of $A$, i.e., the set of prime ideals of $A$ equipped with the Zariski topology,
and $Max(A)$ for the subset of $Spec(A)$ consisting of the maximal ideals of $A.$
We put $Spec^{\circ}(A) := Spec(A)-Max(A)$, equipped with the subspace topology. 
\begin{theorem} \cite{AKNS} \label{thAKNS} Suppose the ring $A$ is finitely generated, and let $N$ be the nilradical
of $A$. Then $A$ is bi-interpretable with $\N$ if and only if $A$ is infinite, $Spec^{\circ}(A)$ is
connected, and there is some integer $d\geq 1$ with $dN = 0$.
\end{theorem} 
Note that the theorem says in particular that if $A$ is a finitely generated 
infinite integral domain, then $A$ is bi-interpretable with $\N$  \cite[Theorem 3.1]{AKNS}.

\begin{cor}
Each finitely generated unitary integral domain is prime, atomic, homogeneous, and  QFA.
\end{cor}
 
\begin{example}\label{O-biint-Z:lem} The ring of integers $\OR$ of a number field of finite degree satisfies the hypothesis of Theorem~\ref{thAKNS}, Therefore, it is bi-interpretable with $\Z$. This fact is well-known, and here we provide a short proof for it. So, assume $\OR$ is the ring of integers of a number field $F$ of degree $m$ and $\beta_1, \ldots , \beta_m$ generate it as a $\Z$-module. We show that $(\OR, \bar{\beta})$ and $\Z$ are bi-interpretable. 
By~(\cite{Nies2007}, Proposition~7.12) we need to prove that $\OR$ is interpretable in $\Z$ and there is a definable copy $M$ of $\Z$ in $\OR$ together with an isomorphism $f\colon  \OR \to M$ which is definable in $\OR$.
The ring $\OR$ is interpreted in $\Z$ by the $m$-dimensional interpretation $\Delta$: 
$$x=\sum_{i=1}^m a_i \beta_i \mapsto  (a_1,\ldots , a_m)$$ where 
$\Z^m$ is equipped with the ring structure: 
$$ e_i\cdot e_j =_\Z(c_{ij1},c_{ij2},\ldots, c_{ijm})\Leftrightarrow \beta_i \cdot \beta_j =_\OR \sum_{k=1}^m c_{ijk}\beta_k$$
and $e_i= (0,\ldots, 0,\underbrace{1}_{\text{$i$'th}},0,\ldots,0)$, for $i=1,\ldots,m$. On the other hand $\Z$ is defined in $\OR$ without parameters as $\Z \cdot 1_\OR$ by the well-known result of Julia Robinson~\cite{J-Robinson}.  So we can take $M= \prod_{i=1}^m \Z\cdot 1_\OR$ with $f(x)$ defined as
$$f(x)=(a_1\cdot 1_\OR, \ldots, a_m\cdot 1_\OR) \Leftrightarrow x=\sum_{i=1}^m a_i \beta_i$$  
which is obviously definable in $\OR$. 

By Lemma~\ref{autiso:lem} we can not get ride of the parameters, since $\OR$ is not automorphically rigid, while $\Z$ is such.
\end{example}
\begin{theorem} \cite{AKNS} 
Each finitely generated unitary commutative ring is  QFA.
\end{theorem}

Let $A$ be an associative commutative ring, which does not necessarily have a unit. Let $Ann(A)=\{a\in A| ax=0, \forall x\in A \}$ and let $A^2$ be the subring generated by $\{xy|x,y\in A\}$. Indeed, $A^2$ is an ideal in $A$. Define the \textit{isolator of $A^2$}, $\Delta(A)$, by
$$\Delta(A)=\{a\in A| ma\in A^2 \text{ for some } m\in \N^\ast \}.$$ We say that $A$ is \textit{regular} if $Ann(A)\leq \Delta(A)$. We note that $\Delta(A)$ is an ideal of $A$. The ring multiplication is clearly trivial in $A/\Delta(A)$. Therefore, if $A$ is a f.g. ring, $A/\Delta(A)$ is just an additive free abelian group of finite rank. Notice that every unitary commutative ring is regular.

For a group $G$, $\Delta(G)$ is defined similarly (See Definition~\ref{isolator:df} below).  

\begin{theorem}\label{ring-qfa:thm} The following are equivalent for a finitely generated commutative ring $A$:
\begin{enumerate}
    \item $A$ is $QFA$ in the language of non-unitary rings.
    \item $A$ is regular.
\end{enumerate}
\end{theorem} 
We shall provide a proof in the following. We first need a few auxiliary definitions and lemmas.

The following results are for groups, but they are true in the case of rings.  
\begin{prop}[\cite{Oger2006}~Theorem 2.]\label{Oger-qfa1:thm}
Let $G$ be a group such that there exists a sentence which is true in $G$ and
false in $G \times \dfrac{\Z}{p\Z}$ for infinitely many primes $p$. Then $Z(G)\leq \Delta(G)$.
\end{prop}
\begin{cor}[\cite{Oger2006}~Corollary 3.] \label{Oger-qfa1:cor} 
If the finitely generated group $G$ is QFA, then $Z(G)\leq \Delta(G)$.
\end{cor}

Consider a full non-degenerate bilinear map $f\colon M\times M \to N$ for some $R$-modules, $M$ and $N$. The mapping $f$ is said to have \textit{finite width} if there is a natural number $S$ such that for every $u\in
N$ there are $x_i$ and $y_i$ in $M$ we have
$$u=\sum_{i=1}^nf(x_i,y_i).$$
The least such number, $w(f)$, is the \textit{width} of $f$.

A set $E=\{e_1,\ldots e_n\}$ is a \textit{complete system}\index{complete system} for a non-degenerate mapping $f$ if $f(x,E)=f(E,x)=0$ implies $x=0$. The
cardinality of a minimal complete system for $f$ is denoted by $c(f)$.

We say a mapping $f$ is a \emph{finite type} if both $w(f)$ and $c(f)$ are finite.

\begin{prop}[\cite{M}]\label{myasnik:bilinT} Let $f:M\times M \to N$ be a full non-degenerate bilinear mapping of finite-type. Then the largest scalar (commutative associative unitary) ring $P(f)$ with respect to which $f$ remains bilinear exists. Furthermore, the ring $P(f)$ is absolutely interpretable in $f$. The formulas of the interpretation depend only on $w(f)$ and $c(f)$. \end{prop}

The following two statements are basic facts  about Noetherian modules.
\begin{lemma}\label{f.g.module:lem} Assume $R$ is Noetherian and $M$ and $N$ are f.g. $R$-modules, then $Hom_R(M,N)$ is finitely generated as an $R$-module.\end{lemma}

\begin{cor}Assume $R$ is a Noetherian scalar ring and $M$ is a f.g. $R$-module. Then $End_R(M)$ (the ring of $R$-module endomorphisms of $M$) is f.g. as an $R$-module.\end{cor}

 We shall denote the quotient ring $A/Ann(A)$ by $\barA$, and for any $a\in A$ we denote $a+Ann(A)$ be $\bara$. Consider the full non-degenerate bilinear map:
$$f_A\colon  \barA \times \barA\to A^2, \quad f(\barx,\bary)=xy, \forall x,y\in A$$

\begin{lemma}\label{ft:lem}
 If $A$ is a f.g. ring, then the bilinear map $f_A$ is of finite type.
\end{lemma}
\begin{proof} Let $\{a_1, \ldots a_n\}$ be a set of generators of $A$. Then $E=\{\bara_1, \ldots, \bara_n\}$ is a finite complete system for $f$. To prove $f$ has finite width consider $X=\sum_{i=1}^m x_iy_i\in A^2$, where each $x_i$ and $y_i$ is a product of the generators $a_j$ of $A$. Rewriting each $x_iy_i$ in as a product of the $a_j$ in increasing order imposed by the subscripts, we have $X=\sum_{i=1}^n X_i$ where either $X_i=0$ or $X_i=a_i P_i$, where $P_i$ is a sum of products of the $a_j$, and so an element of $A$. This proves that $A^2$ has width at most $n$, hence does $f$. \end{proof} 
\begin{cor}\label{ft:cor} If $A$ is a f.g. ring, the bilinear map $f_A$ is $0$-interpetable in $A$. Moreover, there exists a sentence $\phi$ of the language of rings such that $A\models \phi$, and if any f.g. ring $B\models \phi$, then $w(f_B)\leq w(f_A)$ and $w(f_B)\leq w(f_A)$. \end{cor}
\begin{proof} Note that $\barA$ is $0$-interpretable in A for any commutative ring $A$. Since $A^2$ has finite width it is $0$-definable in $A$. The rule of the bilinear map is defined using the ring product. The moreover statement also follows easily from the proof of Lemma~\ref{ft:lem} and definitions of complete system and width of a bilinear map $f$. \end{proof}
The following corollary is an immediate consequence of the above statements.
\begin{cor}\label{P(f)-interpret:cor} Let $A$ be a f.g. ring, and consider the bilinear map $f_A$ and the scalar ring $P(f_A)$. Then, the ring $P(f_A)$ and its action on $\barA$ and $A^2$ are $0$-interpretable in $A$. Moreover, if any f.g. ring $B\models \phi$, for the sentence $\phi$ from Corollary~\ref{ft:cor}, then by Theorem~\ref{myasnik:bilinT} the same formulas that interpret $P(f_A)$ in $A$ interpret $P(f_B)$ in $B$.\end{cor}
\noindent \textit{Proof of Theorem~\ref{ring-qfa:thm}.}
(1.)$\Rightarrow$ (2.): This basically follows from Corollary~\ref{Oger-qfa1:cor} adapted for rings which follows from an analog of Theorem~\ref{Oger-qfa1:thm} for rings without significant changes. We need to point out that the proof of Theorem~\ref{Oger-qfa1:thm} is based on the fact that if $Z(G) \nleq \Delta (G)$, then there is an infinite cyclic subgroup $C$ of $Z(G)$ which trivially intersects $\Delta(G)$. Now for an ultrafilter $U$ on $\MN$, in the ultrapower $G^U$ of $G$, the subgroup $C^U$ has a subgroup $D$ isomorphic to $\MQ^{\omega}$, a product of countably many copies of the additive group of rationals. The torsion-free divisible group $D$ of $Z(G^u)$ splits from $G^U$ as a direct summand. Therefore 
$$G^U \cong G^U \times (\prod_{i\in \MN}\frac{\Z}{p_i\Z})/U,$$
where $p_i$'s are the infinitely many primes guaranteed by the hypothesis, since $\displaystyle (\prod_{i\in \MN}\frac{\Z}{p_i\Z})/U$ is isomorphic to $D$. The same exact argument works for a commutative ring $A$, replacing $\Delta (G)$ with $\Delta(A)$, and $Z(G)$ with $Ann(A)$. Also if $A$ is a f.g. ring $A\ncong A\times \frac{\Z}{p\Z}=B$ for any prime $p$, since $A/A^2$ and $B/B^2$ are not isomorphic as f.g. abelian groups. So Corollary~\ref{Oger-qfa1:cor} also applies to the case f.g. rings.

(2.) $\Rightarrow$ (1.): Consider $A$ and $\barA$ as above. Since $\barA$ is a ring, it embeds into the ring of (abelian group) endomorphisms, $End(A)$, of $A$. Let $R$ be the subring of $End(A)$ generated by the unity $1\in End(A)$ and the copy of $\barA$ in $End(A)$.
Clearly $R$ is a f.g. commutative associative ring with unity. Also $\barA$ is a f.g. $R$-module. Now we define an action of $R$ on $A^2$ by extending the following action on products $xy$ of elements of $R$. For $x,y\in A$ and $r\in R$ let $s,t$ be any representatives of $r\barx$ and $r\bary$, respectively. 
Now define:
$$r\cdot (xy):=sy=xt$$
Well-definedness of the action can be verified readily. Again $A^2$ is a f.g. $R$-module. More importantly, $R$ makes the bilinear map $f$, $R$-bilinear, so it embeds as a unitary subring in $P(f_A)$. 

Now consider the ring of ($R$-module) endomorphisms, $End_R(\barA)$ of $\barA$. We note that under some natural identifications:
$$\frac{R}{I}\leq P(f_A)\leq End_R(\barA)$$
where $I$ is the annihilator ideal $I=Ann_R(\barA)$. Now $P(f_A)$ is an $R/I$-module. As $R/I$ is an $R$-module, $P(f_A)$ is an $R$-submodule of $End_R(\barA)$. Since $\barA$ is a f.g. $R$-module, and $R$ is a (Noetherian) f.g. ring, by Lemma~\ref{f.g.module:lem}, $End_R(\barA)$ is a f.g. module over a Noetherian ring $R$. So $P(f_A)$ is a Noetherian $R$-module. Since $R$ is a f.g. ring, $P(f_A)$ is a f.g. ring. By Theorem~\ref{thAKNS}, $P(f_A)$ is QFA. Let $\eta$ be the QFA-sentence for $P(f_A)$ and $\eta^\ast $ be its translation in the language of $A$. We note that $\barA$ is generated as a $P(f_A)$-module by $\{\bara_1, \ldots \bara_n\}$ and $A^2$ is generated as a $P(f_A)$-module by $\{a_ia_j| 1\leq i,j\leq n\}$.
On the other hand, in general $P(f_A)$ may not act on $Ann(A)/(Ann(A)\cap A^2)$. Now, assume $Ann(A)\leq \Delta(A)$. Then $Ann(A)/(Ann(A)\cap A^2$ is a finite additive abelian group with trivial multiplication. 
Therefore, there are formulas $\rho(\bar{x})$ and $\mu(\bar{x})$, and $\psi(\bar{x})$ which hold on the tuple $\bar{a}$, and 
\begin{itemize}
    \item $A \models \rho(a_1, \ldots , a_n)$ if and only if $\barA$ is generated by $\{\bara_i:i=1,\ldots, n\}$  and $A^2$ is generated by $\{a_ia_j| 1\leq i,j\leq n\}$ as $P(f_A)$-modules.
    \item $A\models \mu(a_1, \ldots , a_n)$ if and only for some terms $w_1(\bar{x}), \ldots, w_k(\bar{x})$, the images of $w_i(\bar{a})$ under the canonical epimorphism  $Ann(A)\to Q(A)$, where $Q(A)=Ann(A)/(Ann(A)\cap A^2)$ generate the finite quotient $Q(A)$.
    \item $A\models \psi(\bar{a})$, where $\psi(\bar{a})$ describes the finitely many ring theoretic relations among the $a_i$'s.\end{itemize}
    
    Now assume that $B$ if a f.g. ring with elements $b_1,\ldots , b_n$, where $B\models \eta*\wedge  \rho(\bar{b})\wedge \mu(\bar{b})\wedge \psi(\bar{b})$. Consider the map:
    $g\colon  A \to B$ extending the assignment $g(a_i)=b_i$.
The map $g$ is a ring homomorphism from $A$ to the subring $B'$ of $B$ generated by the $b_i$. Consider the canonical maps $\hat{g}\colon \barA\to \barB$, and $g|_{A^2}\colon  A^2\to B^2$ induced by $g$.
By Corollary~\ref{P(f)-interpret:cor}, $P(f_A)$ and $P(f_B)$ are isomorphic as unitary rings, and both $\hat{g}$ and $g|_{A^2}$ are isomorphisms of $P(f_A)$ (or equivalently $P(f_B)$) modules. Also, $g$ induces an isomorphism between quotients $Q(A)$ and $Q(B)$. Therefore the subring $B'$ is all of $B$, that is, $B'=B$. This means that $g$ is a homomorphism onto $B$. 
We shall prove now that $ker(g)=0$. Since $\hat{g}$ is an isomorphism, $ker(g)\leq Ann(A)$. Since $g$ induces an isomorphism of $Q(A)$ onto $Q(B)$, $ker(g)\leq Ann(A)\cap R^2$. Finally, since $g$ restricts to an isomorphism from $Ann(A)\cap A^2$ to $Ann(B)\cap B^2$, $ker(g)=0$. This finishes the proof (2.) $\Rightarrow$ (1.).
\qed

The question if a finitely generated field $K$ (always in the language of rings)
encodes the isomorphism type of $K$ in the class of all finitely generated fields 
goes back to the 1970's and seems to have first been posed explicitly in \cite{P1}. The answer was recently obtained.
\begin{theorem}\cite{DP} Let $K$ be an infinite finitely generated field. If char($K$)=2 and $dim(K) > 3$ assume that resolution of singularities above ${\mathbb F}_2$ holds. Then $K$ is bi-interpretable with $\mathbb Z$,(where both $K$ and $\mathbb Z$ are considered as structures in the language of rings). Therefore $K$ is QFA and prime.
\end{theorem}

\section{Algebras}

We now introduce some results about free associative algebras and group algebras of limit groups. We consider these algebras in the language of rings $\{+,\cdot,0,1\}$.

 Recall, that a limit group is a finitely generated fully residually free group, namely a group $G$ such that for any finite set of distinct elements in $G$ there is a homomorphism into a free group that is injective on this set. For example a finitely generated free group or a closed surface group is a limit group.
\begin{theorem} \cite{KMga} Let $F=(X)$ be a non-abelian f.g. free group and $K$ an infinite field. Then  the group algebra $K(F)$ is bi-interpretable with $S(K,\N)$, with parameters $X, P$, where $P$ is a non-invertible polynomial with at least three monomials, uniformly in $K,X,P$. Hence $K(F)$ is absolutely rich. 
\end{theorem}
\begin{theorem} \cite{KMga} Let $G=G(X)$ be a non-abelian limit group and $K$ an infinite field. Then  the group algebra $K(G)$ is bi-interpretable with $S(K,\N)$ with parameters $X,P$, where $P$ is a non-invertible polynomial with at least three monomials, hence it is rich. 
\end{theorem}
\begin{cor}
Let $G$ be a non-abelian limit group and $K$ an infinite f.g. field of char$\neq 2$. Then   $K(G)$ is bi-interpretable with $\Z$. Hence it is rich, prime, atomic, homogeneous, and QFA.
\end{cor}

 This implies, in particular, that if $K$ is an infinite field and $F_n$ a free non-abelian group of rank $n$, then
   $K(F_n)$ is not  elementarily equivalent to   $K(F_m)$ for $n \neq m$.

 Let $G$ be a group with a finite generating $X$. The word metric $d_X\colon G \times G \to \N$ on $G$ with respect to $X$ is defined for a pair $(g,h) \in G\times G$ as the length of a shortest word $w$  in the generators from $X \cup X^{-1}$ such that 
 $gw = h$. This is precisely the geodesic metric on the Cayley graph of $G$ with respect to $X$. Below we view the metric space $G$ with metric $d_x$ as a structure $Met(G,X) = \langle G, \N, d_X\rangle$, where $G$ comes equipped with multiplication and the identity $1$, $\N$ is the standard arithmetic, and $d_X$ is the metric function $d_X\colon  G \times G \to \N$.

 \begin{cor} \cite{KMga}\label{th:metric} 
Let $G$ be a non-abelian limit group with  a finite generating set $X$, $K$ an infinite field. Then 
\begin{enumerate}

\item The metric space $Met(G,X)$ is definable in $K(G)$.
In particular, the set of geodesics in $G$ with respect to $X$ is definable in $K(G).$ 

\item  Finitely generated subgroups of $G$, finitely generated subrings and ideals of $K(G)$ are uniformly definable.

\item Given a number $\delta$, there exists a formula $Con_{\delta }(y_1,y_n)$ such that for any elements $h_1,\ldots  h_n$ 
$K(G)\models Con_{\delta }(h_1,\ldots ,h_n)$  if and only if the subgroup generated by $h_1,\ldots ,h_n$ in $G$ is $\delta$-quasi-convex.

\item Given a number $\delta$, there exists a formula $Hyp_{\delta }(y_1,y_n)$ such that for any elements $h_1,\ldots  h_n$ 
$K(G)\models Hyp_{\delta }(h_1,\ldots ,h_n)$  if and only if the subgroup generated by $h_1,\ldots ,h_n$ in $G$ is $\delta$-hyperbolic.

\item  For any word $w(y_1,\ldots ,y_m)$ there exists a formula $\phi_w(y, X)$ that defines the verbal subgroup corresponding to 
$w(y_1,\ldots ,y_m)$. 
A free finite rank group in a variety defined by $w(y_1,\ldots ,y_m)$ is interpretable in $K(G)$.

\item Malnormality of a finitely generated subgroup of $G$ is definable.
\end{enumerate}
  \end{cor}

\begin{theorem} \label{th:bi2} \cite{KMassoc} Let $K$ be an infinite field and $\mathbb A _K(X)$ be a free associative algebra with finite basis $X$. Then the structures $S(K,\N)$ and $\mathbb A _K(X)$  are bi-interpretable with parameters $X,P$, where $P$ is a non-invertible polynomial, uniformly in $K,X,P$. Hence the ring  $\mathbb A _K(X)$ is absolutely rich.
\end{theorem}

\section{Rich monoids}

Let $M_X$ be a free monoid with finite generating set $X$. It follows from Quine's paper \cite{Quine}, Section 4, that for a free monoid  of  rank $n \geq 2$ 
and generating set $X=\{x_1,\ldots,x_n\}$, the arithmetic  $\langle N, +,\cdot, \uparrow , 0, 1\rangle$ is bi-interpretable in $\mathbb{M}_X$ with parameters $X$, where $x\uparrow y$ means $x^y$.  Since the predicate $z=x^y$ is computable and therefore definable in terms of addition and multiplication (see, for example, \cite{M}) it can be removed from the signature. Quine was working with the structure  $\mathcal{C_X}= \langle C, \frown \rangle$, where $C$ is the set of all finite strings in a finite alphabet $X$ and $\frown$ is the concatenation operation. 
   Obviously,  $\mathcal{C_X}$ is isomorphic to the free semigroup with basis $X$.  Quine did not state the monoid version of his results but the presence of  the identity doesn't make a difference and   his results  are also valid for $\mathbb{M}_X$.

Free partially commutative monoids (also known as trace monoids or right
angled Artin monoids) are defined as follows.  Given a finite graph $\Gamma$ with the set
of vertices $V \neq \emptyset$ and edges $E$ the partially commutative monoid $A_\Gamma$ has $V$ as a set of  generators  and
relations $vu = uv$ for any edge $(u,v) \in E$. In particular, if $E = \emptyset$, then $A_\Gamma \simeq M_V$.

  Quine's results on free mopnoids are generalized in  \cite{KL} for arbitrary partially commutative monoids with trivil center. Namely, the following result holds.

\begin{theorem}\label{th:bi} \cite{KL}
 If a  free partially commutative monoid $A_{\Gamma}$ has trivial center, then $\N$ and $A_{\Gamma}$ are bi-interpretable with the generating set $V$ (the set of vertices of $\Gamma$) as parameters. If $A_{\Gamma}=M_X$is a free monoid with basis $X$, $|X|>1,$ then
 this bi-interpretation is uniform in $X$ and $M_X$ is absolutely rich. 
\end{theorem}

\begin{cor}
 Let $A_{\Gamma}$ be a  free partially commutative monoid with  trivial center. Then $A_{\Gamma}$ is rich. In particular, the free non-commutative monoid $M_X$   is rich.
\end{cor}

\begin{cor}
Let $A_{\Gamma}$ be a  free partially commutative monoid with  trivial center. Then $A_{\Gamma}$ is prime, atomic, homogeneous and QFA.
\end{cor}

In  rich monoids $A_\Gamma$ many properties  are definable by first-order formulas (see, for example Section \ref{wsol}). Here we mention only one result below. 

\begin{theorem} \label{th:sub}
Let $A_{\Gamma}$ be a  free partially commutative monoid with  trivial center. For any $k \in \mathbb{N}$, there is a formula $\psi(y,y_1,\ldots, y_k)$
such that for any elements $g,g_1, \ldots,g_k \in A_\Gamma$ $A_\Gamma \models \psi(g,g_1,\ldots,g_k,X)$   if and only if
$g \in \langle g_1,\ldots,g_k\rangle$.  
\end{theorem}

The following result describes more monoids, where $\MN$ is interpretable, though we do not know if it is bi-interpretable with the monoid.

\begin{theorem}\cite{KL} \label{th:!} The arithmetic $\mathbb{N}$ is interpretable (with parameters)  in the following classes of monoids:
\begin{enumerate}
\item [a)] Baumslag-Solitar monoids $\langle a,b \mid ab^k = b^ma\rangle$ with $k,m>2$  (we do not need parameters for them); 
\item [b)] Non-commutative free partially commutative monoids;   \item [c)] One-relator monoids $G=\langle a,b,C|x=y\rangle$, where $C$ is a non-empty alphabet, some letter of $C$ appears in $y$ and neither $x$ nor $y$ end with $a$. 
\end{enumerate}

\end{theorem}

\begin{cor}  If $G$ is a monoid from Theorem \ref{th:!}, then the first-order theory $Th(G)$ is undecidable.
\end{cor}

\begin{problem}
Which of the monoids from Theorem \ref{th:!} are bi-interpretable with $\MZ$?
\end{problem}

\section{Classical groups and lattices} 
Recall that a f.g. structure $\MA$ is called first order rigid if any f.g. structure $\MB$ which is elementarily equivalent to $\MA$ is isomorphic to it. Recall the main result of~\cite{ALM}:

\begin{theorem}~\cite{ALM}\label{ALM:thm}
Any group which is abstractly commensurable to a non-uniform higher-rank arithmetic
group is first order rigid.
\end{theorem} 
In particular, the special linear group $\SLZ$ over the ring  of rational integers $\Z$ is first order rigid. Moreover, in August 2020 Avni and Meiri put a paper to the Arxiv with the following theorem that gives a large class of rich groups.
\begin{theorem} \cite{AM} Let $\Gamma$ be a centerless irreducible higher rank arithmetic lattice in
characteristic zero. if $\Gamma$ is either non-uniform or is
uniform of orthogonal type and dimension at least 9, then $\Gamma$ is bi-
interpretable with $\Z$.

\end{theorem}

We will describe  our results on classical linear groups. Consider the general linear group $\GLR$ over a commutative associative unitary ring $R$ and let $e_{ij}$, $1\leq i\neq j\leq n$, be the matrix with $ij$'th entry $1$ and every other entry $0$, and let $t_{ij}=\one+e_{ij}$, where $\one$ is the $n\times n$ identity matrix. Let also $t_{ij}(\alpha)=\one+\alpha e_{ij}$, for $\alpha\in R$. We refer to the $t_{ij}(\alpha)$ as {\em transvections}. 

\begin{definition}\label{abdef:df}Let $\TR$ ($\UTR$) be the group of all invertible upper triangular (resp. unitriangular) $n\times n$ matrices over $R$. Recall that $G=T_n(R)\cong \UTR_{\phi_n(R)} \rtimes D_n(R)$ where $D_n(R)\cong(\RM)^n$ is the subgroup of all diagonal matrices in $G$ and $\Phi_{n,R}\colon  D_n(R)\to \text{Aut}(\UTR)$ is the homomorhism describing the action of $D_n(R)$ on $\UTR$. Let $E_n(R,f)$ be an abelian extension of $Z(G)\cong \RM$ by $D_n(R)/Z(G)\cong (\RM)^{n-1}$ defined by a symmetric 2-cocycle $f\colon (\RM)^{n-1}\times (\RM)^{n-1}\to \RM$. We call $E_n(R,f)$ a non-split torus. Define 
$$\TRf = \UTR\rtimes_{\psi_{n,R}} E_n(R,f)$$ 
where 
  \[\psi_{n,R}((x,y))=\phi_{n,R}((x,y)),~~ (x,y)\in B \times Z(G)=D_n(R)=E_n(R,f).\]
 The definition actually makes sense since $ker(\phi_{n,R})=Z(G)$ and it is easy to verify that it is indeed a homomorphism.\end{definition}

 \begin{theorem}[\cite{MS2019}]\label{sln-biinter:thm} Let $\OR$ be the ring of integers of a number field of finite degree, and $n\geq 3$.
\begin{enumerate}
    \item The ring $\OR$ and the group $\SLO$ are regularly bi-interpretable. Hence, the ring $\Z$ of rational integers and $\SLO$ are regularly bi-interpretable. Moreover, if $H$ is a group and $H\equiv \SLO$ as groups, then $H\cong \SLR$ for some ring $R\equiv \OR$ as rings.
  \item  If $\OR$ has a finite group of units then $\GLO$ and $\Z$ are regularly  bi-interpretable. If $\OR$ has an infinite group of units then $\GLO$ is not bi-interpretable with $\Z$. However, if $H$ is any group such that $H\equiv \GLO$ then $H\cong \GLR$ for some ring $R\equiv \OR$.
  \item If $\OR$ has a finite group of units then $\TO$ and $\Z$ are regularly bi-interpretable. If $H\equiv \TO$ as groups then  $H\cong \TRf$ (See Definition~\ref{abdef:df}) for some $R\equiv \OR$.
\end{enumerate}
\end{theorem}	

\begin{cor}[\cite{MS2019}] The group $\SLO$ is rich for any ring of integers $\OR$. The groups $\GLO$ and $\TO$ are rich if $\OR$ has a finite group of units.
\end{cor}
\begin{definition}\label{isolator:df} For a finitely generated group $G$ by $\Delta{G}$ we mean the inverse image of the torsion subgroup of $G/G'$ under the canonical epimorphism $G\to G/G'$.\end{definition} 
\begin{remark} Some  remarks are in order on the statements in Theorem~\ref{sln-biinter:thm}. Firstly, the interpretations of $\Z$ in $\SLO$ and $\GLO$ used in (1) are with respect to some parameters. We note that if $\OM$ is infinite then by the main theorem of~\cite{Lasserre} (See Theorem~\ref{Lasserre-main:thm} below.) $G=\TO$ is not QFA since $Z(G)$ is not included in $\Delta(G)$. Therefore, $\TO$ is not bi-interpretable with $\Z$ by Theorem~\ref{th:QFA}. \end{remark}

\section{Nilpotent and Polycyclic Groups}

\begin{theorem}[\cite{Lasserre}]\label{Lasserre-main:thm} Assume $G$ is a polycyclic-by-finite group. Then the following are equivalent.
	\begin{enumerate}
		\item For every finite index subgroup $H$ of $G$, $Z(H)\leq \Delta(H)$.
		\item $G$ is QFA.
		\item $G$ is prime 
	\end{enumerate} \end{theorem}
Indeed in an earlier paper~\cite{Oger2006} the above theorem was proved for f.g. nilpotent groups. However, in general  polycyclic groups are not bi-interpretable with $\Z$.

In~\cite{Khelif},\cite{Nies2007}  the following theorem is proved. 
\begin{theorem}[\cite{Khelif},\cite{Nies2007}] \label{khelif-ut3:thm} The group $UT_3(\Z)$ is not bi-interpretable with $\Z$.\end{theorem}
 The proof of the theorem rests on the following observation: There is a mutual interpretation of $(UT_3(\Z),t_{12}, t_{23})$ with $\Z$, and if $R$ is any model of $\Th(\Z)$, the same formulas that mutually interpret $\Z$ and $(UT_3(Z),t_{12},t_{13})$ in one another, mutually interpret $R$ and $(UT_3(R),t_{12},t_{23})$ (See Example~\ref{malcev-inter:ex} below). This implies that if $\Z$ is bi-interpretable with $UT_3(\Z)$, then $R$ is bi-interpretable with $UT_3(R)$. Now, $Aut(UT_3(R), t_{12},t_{23})\cong Aut(R)$. Let $R$ be a countable non-standard model of $\Th(\Z)$  with only countably many automorphisms. Then using the structure theory of countable non-standard models of $\Z$ one can prove that there are uncountably many distinct automorphisms of $UT_3(R)$ fixing the standard copy $UT_3(\Z)$ and in particular the constants $t_{12}$ and $t_{23}$. This contradicts our earlier conclusion that $Aut(UT_3(R), t_{12},t_{23})\cong Aut(R)$.  

In the following we extend Theorem~\ref{khelif-ut3:thm} to all torsion-free f.g. nilpotent groups. Let us first give details of the mutual interpretation of $\Z$ and $G=(UT_3(\Z), t_{12},t_{23})$, as an example. The construction works clearly for all commutative associative unitary rings, so we provide this more general version.

\begin{example}\label{malcev-inter:ex} Let $R$ be a commutative associative ring. The ring $R$ and the group $UT_3(R)$ are mutually interpretable in each other with the help of parameters $t_{12}$ and $t_{23}$ from $UT_3(R)$. First, we present the interpretation of $R$ in $UT_3(R)$ due to Mal'cev~\cite{malcev2}. We note that one can represent any element of $UT_3(R)$ by a triple $(a,b,c)\in R^3$. Then the group product in $UT_3(R)$ is defined by
 \begin{equation} \label{ut3-mult:eq} (a_1,b_1,c_1)(a_2,b_2,c_2)=(a_1+a_2,
 b_1+b_2,c_1+c_2+a_1b_2).\end{equation}
 We note that under the above identification $t_{12}=(1,0,0)$, $t_{23}=(0,1,0)$, and $[t_{12},t_{23}]=(0,0,1)=t_{13}$.
 
 The following statements hold in $G$:
\begin{itemize}
\item[(a)] $C_G(t_{12})\cap C_G(t_{23})=Z(G)$.
\item[(b)] $C_G(t_{12})$ and $C_G(t_{23})$ are both abelian.
\item[(c)] $[t_{12}, C_G(t_{23})]=[C_G(t_{12}), t_{23}]=Z(G)$.
\item[(d)] $[G,G]=Z(G)$\end{itemize}
Even though we are not going to use it here, we note that all these conditions can be expressed by first-order sentences of the language of groups, so indeed this interpretation is a regular interpretation. In fact one can prove that it can lead to a $0$-interpretation~\cite{beleg94}. Now, the properties given above allow us to define (interpret) a ring $S$ in $Z(G)=\{(a,b,c)\in UT_3(R)| a=b=0\}$ using only the group operations in $G$ as following:
\begin{itemize}
\item Let $1\in Z(G)$ be the zero $0_S$ of $S$.
\item Let $[t_{12}, t_{23}]=t_{13}$ be the unity $1_S$ of $S$.
\item Define ring addition $\boxplus$ on $S$ be the group operation on $Z(G)$, that is:
$$x\boxplus y= xy.$$
\item If $ x,y\in S$, there are $x'\in C_G(t_{12})$, $y'\in t_{23}$ such that $[x',t_{23}]=x$ and $[t_{12},y']=y$. Define ring product $\boxdot$ on $S$ by:
$$x\boxdot y=[x',y'].$$
\end{itemize}
Now the map
$$S\rightarrow R, \quad (0,0,c)\mapsto c$$
provides an isomorphism of rings.

Interpreting $UT_3(R)$ in $R$ is straightforward, once as above, we identify $UT_3(R)$ with $R^3$, and let the group operation be defined by \eqref{ut3-mult:eq}. Hence, the identity map on $R^3$ is the coordinate map of the interpretation, where the group operation is defined in $R$ by $\eqref{ut3-mult:eq}$. This is clearly a $0$-interpretation.
    \end{example}

\begin{theorem} Let $G$ be a torsion-free finitely generated nilpotent group. Then $G$ is not bi-interpretable with $\Z$.\end{theorem}

\begin{proof} To get a contradiction assume that $G$ is bi-interpretable with $\Z$. By a theorem of~\cite{Oger2006}, $Z(G)\leq \Delta (G)$, otherwise $G$ is not QFA, and it can not be bi-interpretable with $\Z$. So consider a Mal'cev basis $\bar{g}$ of length $m$ for such a group $G$, adapted to its upper central series. We denote the product $\prod_{i=1}^m g_i^{a_i}$ by $\bar{g}^{\bar{a}}$. Then, there are polynomials  $P_i(\bar{x},\bar{y})\in \Q[x_1,\ldots,x_m,y_1,\ldots ,y_m]$, 
$Q_i(\bar{x})\in \Q[x_1,\ldots,x_m,y]$ where $P_i(\Z^m,\Z^m)\subset \Z$ and $Q_i(\Z^m,\Z)\subset \Z$ computing the $i$'th-coordinate of the product of two elements in $G$, and exponentiation, respectively. That is, $\bar{g}^{\bar{a}}\bar{g}^{\bar{b}}=\bar{g}^{\bar{c}}$, where $c_i=P_i(\bar{a},\bar{b})$, and $(\bar{g}^{\bar{a}})^{e}= \bar{g}^{\bar{d}}$, where $d_i=Q_i(\bar{a},e)$.

Let $R$ be a model of $\Th(\Z)$, since $R$ is a binomial domain we can  define the Mal'cev $R$-completion $G^R$ of $G$. Now, any element of $G^R$ can be uniquely represented as $\bar{g}^{\bar{\alpha}}$, $\bar{\alpha}\in R^m$. Multiplication in $G^R$ and ``exponentiation" by elements of $R$ can be defined similarly to the case of $G$ by the same polynomials $P_i$ and  $Q_i$ as in $G$. 
 Consequently, there is an interpretation $\Gamma$, of $G$  in $\Z$ with the coordinate map $f_\Gamma\colon  \Z^m \to G$, $\bar{a}\mapsto \bar{g}^{\bar{a}}$ with the $P_i$ and $Q_i$ defining a group structure on $\Z^m$. By the above paragraph this interpretation is uniform among all models $R$ of $\Th(\Z)$.

Since $G$ is infinite it has an infinite (torsion-free) center $Z(G)$, with rank $r>0$ as a free abelian group. So pick a non-trivial element $z\in Z(G)$. Let $\tilde{z}=f_\Gamma(z)$ be the image of $z$ in $\Z^m$ under the coordinate map of the interpretation of $G$ in $\Z$. Since exponentiation is arithmetic in $\Z$, the cyclic subgroup $\langle \tilde{z} \rangle= \{(\tilde{z})^b: b\in \Z)\}$ is definable in $\Z$. Hence, the cyclic subgroup $\langle z  \rangle=\{z^b: b\in \Z\}$ is definable in $G$ via the bi-interpretability. The above argument holds for all elements of the Malbasis $g_i\in Z(G)$. Without loss of generality we can assume that for all such $i$, $m_0\leq i \leq m$, for some $m_0\in \N^\ast $. Therefore, there is a first-order sentence of language of groups expressing that $Z(G)=\langle g_{m_0}\rangle \times \cdots \langle g_m\rangle $.  The interpretation $\Gamma$ needs to be with respect to some parameters $\bar{h}$ of $G$, otherwise $Aut(G)\cong Aut(\Z)$ which is only possible if $G$ is cyclic. By bi-interpretability, for any axiom $\Phi$ of $\Z$, there exists a first-order formula $\phi_G(\bar{x})$ of $G$ such that $\Z\models \Phi \Leftrightarrow G\models \Phi_G(\bar{h})$. Therefore, the same formulas used in the interpretation $\Gamma$ of $\Z$ in $G$, interpret a ring $S\equiv \Z$ in $G^R$. The same formulas that defined $\langle g_i\rangle$ in $G$. Now, define one-parameter subgroups $\langle g_i\rangle^S= \{g_i^b|b\in S\}$ in $G^R$. In particular we have that:
$$R^r \cong \prod_{i=m_0}^m \langle g_i\rangle^R \cong Z(G^R) \cong \prod_{i=m_0}^m \langle g_i\rangle^S\cong S^r.$$ 

This shows that $R\cong S$. Indeed, we showed that $R$ and $G^R$ are bi-interpretable with respect to the parameters $h_i\in G\leq G^R$. 

Since $G$ is infinite, $G/G'$ is infinite. Hence, there exists at least one element $g_i$ of the Mal'cev basis such that $g_i\notin \Delta(G)$. Without loss of generality we can assume that $i=1$. Now an argument using Hall-Petresco collection process shows that $P_1(\bar{x},\bar{y})= x_1+y_1$, that is 
$$\bar{g}^{\bar{a}}\bar{g}^{\bar{b}}= g_1^{a_1+b_1} \prod_{i=2}^m g_i^{P_i(\bar{a},\bar{b})}.$$ 

Since $R$ is a countable non-standard model of $\Z$ its additive group $R^+\cong A \oplus D$, where $D$ consists of countably many copies $\Q^\omega$ of the additive group of rationals, $\Q$. We note that $\Z\in A$. Consider the uncountably many distinct homomorphisms $\{\phi_k\colon \Q^\omega \to \Q^\omega\}$. One can extend each $\phi_k$ to a homomorphism $\sigma_k\colon  R\to R$ such that $(\sigma_k)|_A= id_A$. Now for each $k$ define a map $$\psi_k\colon  G^R \to G^R,~ \bar{g}^{\bar{a}}\mapsto \bar{g}^{\bar{a}}g_m^{\sigma_k(a_1)}$$

One can easily check that each $\psi_k$ is an automorphism of $G^R$. Moreover, these $\psi_k$ are pairwise distinct and each one fixes the standard copy $G^\Z$ of $G$ in $G^R$. In particular the parameters $\bar{h}$ of the bi-interpretation are fixed. Note that since $G^R$ is bi-interpretable with $R$ with parameters $\bar{h}$, $Aut( G^R, \bar{h})\cong Aut(R)$. We arrived at a contradiction, since $Aut( G^R, \bar{h})$ is uncountable, while $Aut(R)$ is countable.

\end{proof}


\section {Metabelian groups: examples and open problems}

We believe that there are  many finitely generated rich metabelian groups. 

In \cite{Khelif} Khelif showed that wreath product $\mathbb{Z} wr \mathbb{Z}$, as well as the metabelian Baumslag-Solitar groups $BS(1,n) = \langle a,b \mid b^{-1} ab = a^n\rangle$ are rich. We sketch his proof of these results in the sections below. 

 Kelif also claimed that a free metabelian group $G$ of finite rank $n \geq 2$ is bi-interpretable with $\mathbb{Z}$ (see \cite{Khelif,Nies_blog,AKNS}).  However, his argument, as seems from the provided proof, works only in the case $n=2$. Indeed, the argument is based on an assumption that $G'$ is a free module over the integral group ring $\MZ[G/G^\prime]$, but this module is free only for $n = 2$.   We  give a complete proof of the general result below in a separate Section \ref{se:free-metabelian}, where we also discuss applications of this result to some known open problems.  
 
 \subsection{Metabelian Baumslag-Solitar groups and the wreath product $\MZ wr \MZ$}
 The metabelian Baumslag-Solitar groups are defined by  one-relator presentations $BS(1,k)=\langle a,b \mid b^{-1}ab=a^k
\rangle$, where $k \in \mathbb{N}$. If $k = 1$ then $BS(1,1)$ is free abelian of rank $2$.
\begin{prop} \cite{Khelif} If $k\neq 1$, then $BS(1,k)$ is bi-interpretable with $\mathbb N$.
\end{prop}
 \begin{proof} Recall that the group $BS(1,k)$ is isomorphic to the group $\mathbb{Z}[1/k]
\rtimes \mathbb{Z}$, where $\mathbb{Z}[1/k] \cong ncl(a)$ and
$\mathbb{Z} \cong \langle b \rangle$, where $$\mathbb{Z}[1/k] =\{zk^{-i}, z\in\mathbb Z, i\in\mathbb N\}$$ and the action of $\langle b \rangle$ is given by $b^{-1}ub=uk$. Therefore $BS(1,k)$ has a structure of a $\mathbb Z$ module $\mathbb{Z}[1/k]$. Thus, we can think of elements in $BS(1,k)$
as pairs $(zk^{-i} , r)$ where $z,r,i \in \mathbb{Z}$.  The product is defined as
$$(z_1k^{-i_1} , r_1)(z_2k^{-i_2} , r_2)=(z_1k^{-i_1}+z_2k^{-(i_2+r_1)} , r_1+r_2).$$
The inverse of an element $(zk^{-y},r)$ is $(-zk^{-(y-r)}, -r)$.

This represents of $BS(1,k)$ as the set of triples of integers with operations and therefore
$BS(1,k)$ is interpretable in $\mathbb Z$.
Now we have to define $\mathbb Z$ in the $\mathbb Z$-module $\mathbb{Z}[1/k]$.   Notice that if $m,n$ are two integers, then $k^m-1$ divides $k^n-1$ if and only if $m$ divides $n$ in $\mathbb Z$. On the set of  powers of $k$ we can, therefore, define not only addition but also divisibility of exponents. Therefore, multiplication is also definable \cite{Robinson}.  Now the element $n$ of $\mathbb Z$ can be interpreted in $G$ as $k^{n}$ ($b^{-n}ab^{n}).$  

We now have to show that there is a definable isomorphism between the set of elements $(zk^{-i}, r)=(b^{i}a^zb^{-i},b^r)$ and triples $(k^z, k^{-i}, k^r)=(b^{-z}ab^{z}, b^{i}ab^{-i}, b^{-r}ab^{r})$. The subgroups $\langle b\rangle$ and $ncl (a)$ are definable in $G$ as centralizers of $b$ and of $a$. We can write a formula that connects $b^r$ ($b^i$) with $b^{-r}ab^{r}$ ($b^{i}ab^{-i}$) saying that $b^r$ ($b^i$) is a conjugator that belongs to $\langle b\rangle$. Now we have to express $a^z$ in terms of $b^{-z}ab^{z}=a^{K^z}, b^{z}ab^{-z}$. For all integers $n$, the integer $z$ is the only integer such that $(k^{nz}-1)/(k^n-1)\equiv z (mod (k^n-1)). $ This defines 
$a^z$ as the only power of $a$ such that for all $n$, $a^z=a^\frac{k^{nz}-1}{k^{n}-1}a^{(k^{n}-1)y}$ or
$a^{z(k^n-1)}=a^{(k^{nz}-1)}a^{y(k^n-1)^2}$. We can re-write this as
$[b^n,a^{-z}]=[b^{nz},a^{-1}][b^n,[b^n,u]],$ where $u\in ncl(a).$ 
 \end{proof} 
One can also show that $\MZ wr \MZ$ is bi-interpretable with $\mathbb Z$.
 \subsection{Open problems for solvable groups}

Here we formulate several open problems.

\begin{problem} [General problem] 
Describe  finitely generated metabelian groups that are bi-interpretable with $\MN$.
\end{problem}

Now we ask several particular questions in the direction of the general question above.

\begin{problem}
Are wreath products $\MZ^n wr \MZ^m$, where either $m >1$ or $n >1$,   bi-iterpretable with $\MN$? 
\end{problem}

Rigid groups were introduced By N. Romanovski in \cite{Rom6}, metabelian rigid groups are precisely the ones that are universally equivalent to free metabelian groups $M_n, n\geq 2$, so intuitively  they have properties similar to $M_n$. Hence the following question. Notice that the groups $\MZ^n wr \MZ^m$ are rigid.

\begin{problem}
Let $G$ be a finitely generated non-abelian metabelian rigid group. Is it true that $G$ is bi-interpretable with $\MN$?
\end{problem}

Let $\Gamma =(V,E)$ be a finite graph. 
Denote by $M_\Gamma$ the partially commutative metabelian group defined by the graph $\Gamma$ as follows. The set of vertices $V$ is a generating set of $M_\Gamma$ with defining relations (in the variety  of metabelian groups) $uv = vu$  for $u, v \in V$,  provided that $(u,v) \in E$. 

\begin{problem}
Let $M_\Gamma$ be a partially commutative metabelian group with the trivial center. Is it true that $M_\Gamma$ is bi-interpretable with $\MN$? 
\end{problem}

The notion of a random finitely presented group came to group theory from the works of Gromov and Olshanskii. Nowadays, it occupies a prominent place in geometric group theory. In \cite{GMO} Garreta, Legarreta, Miasnikov and Ovchinnikov  introduced and studied random metabelian groups.   It turns out, these groups share many properties with free metabelian groups. Hence the following problem.

\begin{problem}
Let $m, n \in \MN$ with $0 \geq m \leq n-2$ be fixed.  Let $G$ be a metabelian group given in the variety of metabelian groups by a presentation
$$
G = \langle a_1, \ldots,n \mid r_1, \ldots,r_m\rangle.
$$
where $r_1, \ldots,r_m$ are groups words of length $\ell$ in the generators $a_1, \ldots,n$.
Is it true that asymptotically almost surely (when $\ell \to \infty$) the group $G$ is bi-interpretable with $\MN$?
\end{problem}

\begin{problem}[General Problems]
 Describe finitely generated metabelian groups that are 
 \begin{itemize}
     \item [1)] QFA,
     \item [2)] first-order rigid.
 \end{itemize}
  \end{problem}
 In this respect it is interesting to consider the classes of metabelian groups mentioned above. 
 
 \begin{problem}
 Problems 2-5 where the property "bi-interpretable with $\MN$" is replaced with "QFA" or "first-order rigid".
 
 \end{problem}

Finally we state several   open questions for free solvable groups.

\begin{problem}
Is a free solvable group of class 3 and finite rank $n \geq 2$ bi-interpretable with $\MN$? 
\end{problem}

\begin{problem}
Is a free solvable group of class 3 and finite rank $n \geq 2$ QFA? 
\end{problem}

\begin{problem}
Is a free solvable group of class 3 and finite rank $n \geq 2$ prime? 

\end{problem}
\begin{conjecture} Every finite rank free solvable group of class $c \geq 2$ is first order rigid.
\end{conjecture}

 \section{Free metabelian groups}
 \label{se:free-metabelian}
 
 Throughout this section we denote by $G$ a free metabelain group
of rank $n\geq 2$ with basis $X = \{x_1,\ldots ,x_n \}$.

\begin{theorem}  \label{metab} Every free metabelian group of finite rank $\geq 2$ is bi-interpretable with $\mathbb Z$.\end{theorem} 

\subsection{Preliminaries for metabelian groups}

In this section we introduce notation and describe  some   results  that we need in the sequel.  

Some notation: $G' = [G,G]$ is the commutant of $G$, $G_m$ is the $m$'th term of the lower central series of $G$,  $\langle A \rangle$ is the subgroup generated by $A \subseteq G$,  $C_G(A)$ is the centralizer of a  subset $A  \subseteq G$, if  $x, y \in G$  then  $[x,y] = x^{-1}y^{-1}xy$ is the commutator of $x$ and $y$, and $x^y = y^{-1}xy$ is the conjugate of $x$ by $y$.  The maximal root of an element $x \in G$ is an element $x_0 \in G$ such that $x_0$ is not a proper power in $G$ and $x \in \langle x_0\rangle$. We term an element  $x \in G$ a { \it root} if it is not a proper  in $G$.

\medskip \noindent
\subsubsection{Centralizers.} In \cite{Mal1} Mal'cev obtained complete description of centralizers of elements in $G$. Namely, the following holds.   Let $x \in G$. Then 
\begin{itemize}
\item [1)] if $x \in G'$ then $C_G(x) = G'$
\item [2)] if $x \not \in G'$ then $C_G(x) = \langle x_0\rangle$, where $x_0$ is the unique maximal root of $x$, i.e., $x = x_0^k$ for some $k \in \N$ and $x_0$ is not a proper power. 
\end{itemize}

It follows, in particular, that the maximal roots of elements in $G$ exist and they are unique.

\medskip \noindent
\subsubsection{A homomorphism on  commutators.} Let $v \in G\smallsetminus G'$. Define a map $\lambda_v\colon G' \to G'$ such that  for $c \in G'$ 
$ \lambda_v(c) = [v,c]$.  Then the map $\lambda_v$ is a homomorphism.  Indeed, using the commutator identity 
$$
[x,yz] = [x,z][x,y]^z,
$$
which holds for every elements $x,y,z$ of an arbitrary group, one has for $c_1,c_2 \in G'$ 
\begin{equation}\label{eq:1.2}
[v,c_1c_2] = [v,c_2][v,c_1]^{c_2} = [v,c_2][v,c_1] =  [v,c_1][v,c_2],
\end{equation}
as claimed.

Similarly, in the notation above,  the map $\mu_v\colon  c \to [c,v]$ is a homomorphism $\mu_v\colon G' \to G'$.

\medskip \noindent
\subsubsection{An exponentiation formula.} Let $v \in G\smallsetminus G'$ and $d \in G'$. Then for any $k \in \Z$ there exists $c \in G'$ such that:
$$
(vd)^k = v^kd^k[c,v].
$$
We prove first, by induction on $k$,  that 
$$
d^kv = vd^k[c,v]
$$ 
for  some $c \in G'$. Indeed, for $k = 1$ one has the standard equality $dv = vd[d,v]$. Now 
$$
d^{k+1}v = d^kdv = d^kvd[d,v] = vd^k[c_1,v]d[d,v] = vd^{k+1}[c_1,v][d,v]=vd^{k+1}[c_1d,v],
$$
the last equality comes from the property~\eqref{eq:1.2}, that the map $\mu_v$ is a homomorphism on $G'$.

Now one can finish the claim by induction on $k$ as follows (here elements $c_i \in G'$ appears as the result of application of the induction step and the claim above):
$$
(vd)^{k+1} = (vd)^kvd  = v^kd^k[c_2,v]vd = v^kd^kv[c_2,v][[c_2,v],v]d = 
$$
$$
= v^kvd^k[c_3,v][c_2,v][[c_2,v],v]d = v^{k+1}d^{k+1}[c_3c_2[c_2,v],v] =  v^{k+1}d^{k+1}[c,v]
$$
where the second to last equality comes again from the property~\eqref{eq:1.2}, and $c = c_3c_2[c_2,v]$. This proves the claim.

\medskip \noindent
\subsubsection{A characterization of the identity.} 
\begin{lemma} \label{le:identity}
Let $d \in G'$. If for any $v \in G\smallsetminus G'$ there exists $c \in G'$ such that 
$d = [c,v]$, then $d = 1$.
\end{lemma} 
\begin{proof}
In \cite{Mag} (see also \cite{LR}) Magnus constructed an embedding $\mu\colon  G \to \Z^n wr \Z^n$ of $G$ into the restricted wreath product of two free abelian groups $\Z^n$ and  $\Z^n$ as follows. Recall that  $X = \{x_1, \ldots,x_n\}$ is the  basis  of $G$.  Let $A$ be a free abelian group of rank $n$ with basis $\{a_1, \ldots,a_n\}$. Denote by $ - $ the standard abelianization epimorphism: $G  \to G/G' = \bar G$.  Consider the  restricted wreath product $\bar G wr A$ (here $A$ acts on the direct product $A^{\bar G}$). The map $X \to \bar G wr A$ defined by $x_i \to {\bar x_i}a_i$ extends to a homomorphism $\mu\colon  G \to\bar G wr A$ which is injective. This is called Magnus embedding. For $g \in G$ the image $\mu(g)$ can be written as $\bar g \hat g$, where $\hat g  \in A^{\bar G}$. 

Now fix an element $d \in G'$. Suppose the equality $d = [c,v] = c^{-1} v^{-1}cv$ holds for some $v \in G\smallsetminus G'$ and $c \in G'$.  Then $\mu(d) = {\hat d}, \mu(c) = \hat c$, $\mu(v) = \bar v \hat v$, and the  equality above gives 
$$
\hat d = {\hat c}^{-1} {\hat v}^{-1} {\bar v}^{-1} {\hat c} {\bar v}{\hat v} = {\hat c}^{-1}  {\bar v}^{-1} {\hat c} {\bar v}.
$$
An element $x \in A^{\bar G}$ is  a function $x\colon  \bar G \to A$ with  finite support $supp(x) = \{\bar g \in \bar G \mid x(\bar g) \neq 1\}$.  Viewing  the group $\bar G$ as the lattice $\Z^n$ one can define the diameter $dim(\bar g)$  of an element $\bar g \in \bar G$ as the longest distance (in $\Z^n$)  between any two pairs of elements of $supp(\bar g)$.  When $\bar v$ conjugates $\hat c$, it translates  $supp(\hat c)$ by the integer vector $\bar v$.  Note that the diameter of the support of the element ${\hat c}^{-1}  {\bar v}^{-1} {\hat c} {\bar v}$ is at least the length of the vector $\bar v$, unless $\hat c = 1$. The equality $\hat d = {\hat c}^{-1}  {\bar v}^{-1} {\hat c} {\bar v}$ implies that the elements $\hat d$ and ${\hat c}^{-1}  {\bar v}^{-1} {\hat c} {\bar v}$ have the same support. But if the vector $\hat v$ is chosen such that its length  is greater then the diameter $dim(\hat d)$ the equality $\hat d = {\hat c}^{-1}  {\bar v}^{-1} {\hat c} {\bar v}$ implies that $\hat c = 1$, so $\hat d = 1$, which implies that $d = 1$, as required.

\end{proof}

\medskip \noindent
\subsubsection{Normal forms in $\Z\bar G$-module $G'$.}  \label{subsec:2.5}   The group $G$ acts by conjugation on $G'$, which gives an action of the abelianization $\bar G = G/G'$ on $G'$. This action extends by linearity to an action of the group ring $\Z\bar G$ on $G'$ and turns $G'$ into a $\Z\bar G$-module. Denote by $a_i$ the image of $x_i$ in $\bar G$, $i = 1, \ldots,n$.  The group $\bar G$ is a free abelian group with basis $a_1, \ldots,a_n$, so the group ring $\Z\bar G$ can be viewed as  the Laurent polynomial ring $A = \Z[a_1,a_1^{-1}, \ldots,a_n,a_n^{-1}]$. For the action of $A$ on $G'$ we use the exponential notation, namely, for $u\in G'$ and $a \in A$ we denote by $u^a$ the result of the action of $a$ on $u$.      Let $Y = \{ [x_i,x_j] \mid 1\leq j < i \leq n\}$.  Then the set of commutators $Y$ generates $G'$ as an $A$-module. Note, that  $G$ (as a metabelian group) satisfies the Jacobi identity, i.e, for every $u, v, w \in G$ the following equality holds 
$$
[u,v,w][v,w,u][w,u,v] = 1.
$$
In particular,  for $u = x_i, v = x_j, w = x_k$ one gets (in the module notation)
$$
[x_i,x_j]^{a_k-1}[x_j,x_k]^{a_i-1}[x_k,x_i]^{a_j-1} = 1.
$$
It shows also that commutators from $Y$ satisfy the following identity
$$
[x_i,x_j]^{a_k-1}[x_j,x_k]^{a_i-1} = [x_i,x_k]^{a_j-1},
$$
so $Y$ is not a free generating set of the module $G'$. However, there are nice normal forms of elements of the module $G'$ described in \cite{BR}.  Namely, every element $u \in G'$ can be uniquely presented as the following product
$$
u = \Pi_{1\leq j<i\leq n}[x_i,x_j]^{\beta_{ij}(a_1,\ldots,a_i)},
$$
where $\beta_{ij}(a_1,\ldots,a_i) \in \Z[a_1,a_1^{-1}, \ldots,a_i,a_i^{-1}] \leq \Z \bar G$.
The following statement  follows from normal forms. 
\begin{prop}\label{nf} $G'$ is a free module over $\mathbb Z[a_1,a_1^{-1},a_2, a_2^{-1}]$ with the basis $\{[x_i,x_j]^{a_3^{\delta _3}\ldots a_{j}^{\delta _j}}\}$ for all $1\leq i<j\leq n$, $\delta _3,\ldots,\delta _j\in\Z^{j-2}.$ 
\end{prop} 

\subsection{$\Z$ is absolutely interpretable in  $G$}

The description 2.1  of the centralizers of elements in $G$ implies that the formula 
$$
\phi(x) = \forall y \forall z ([x,[y,z]] = 1)
$$
defines the commutant $G'$ in $G$. Indeed, $\phi(g)$ holds in $G$ for an element $g \in G$ if and only if $C_G(g) \geq G'$, which happens only if $g \in G'$.

The free nilpotent group $G/G_3$ of  class 2 and rank n  is $0$-interpretable in $G$. Indeed, the verbal subgroup $G_3$ has finite width in $G$ \cite{Rom}, hence it is $0$-definable in $G$. Therefore the quotient group $G/G_3$ is $0$-interpretable in $G$.

 It was shown in \cite{MS} that the ring $\Z$ and its action by exponents on free abelian groups  $G/G'$ and $G'/G_3$ are $0$-interpretable in $G/G_3$, hence, by transitivity of $0$-interpretations,  it is $0$-interpretable in $G$. We denote this interpretation of $\Z$ in $G$ by $\Z^\ast $. 

Now, we may use in our formulas  expressions of the type $y = x^m mod G'$ for $x,y \in G \smallsetminus G'$ and $m \in \Z^\ast $  viewing them as notation for the corresponding formulas of group theory language which are coming from the interpretations of $\Z^\ast $ and its actions on $G/G'$. Similarly, for $G'/G_3$.  More precisely, the interpretation $\Z^\ast $ is given by a definable in $G$ subset $A \subseteq G^k$  together  with a definable in $G$ equivalence relation $\sim$ on $A$  and formulas $\psi_+(\bar x, \bar y,\bar z), \psi_\circ(\bar x, \bar y,\bar z)$ with $k$-tuples of variables $\bar x, \bar y, \bar z$, such that  the formulas $\psi_+$ and $\psi_\circ$ define binary operations on the factor set $A/\sim$ (denoted by $+$ and $\circ$) and the structure $\langle A/\sim; +, \circ\rangle$ is a ring isomorphic to $\Z$. Furthermore, the exponentiation by $\Z^\ast $ on $G/G'$ and on $G'/G_3$ is also $0$-interpretable, which means that there are formulas in the group language, say $expnil_1(u,v,\bar x)$ and $expnil_2((u,v,\bar x)$, such that  for $g,h \in G$ and $m \in \Z^\ast $, where $m$ is the equivalence class of some tuple $ \bar a \in A$ (we write in this case  $m = [\bar a]$), one has $g^m = h  (mod \ G')$ if and only if $expnil_1(g,h,\bar a)$ holds in $G$ and also for elements $p, q \in G'$  $ p^m  = q (mod \ G_3)$ if and only if $expnil_2(p,q,\bar a)$ holds in $G$.

\subsection{Interpretation of $\Z$-exponentiation on $G$} 

Now, in the notation above,  we construct a formula $exp(u,v,\bar x)$ of the group language, where $\bar x$ is a $k$-tuple of variables, such that for $g,h \in G$ and $m \in \Z^\ast $, where $m$ is the equivalence class $[\bar a]$ of some $\bar a \in A$, the following holds
$$
g = h^m \Longleftrightarrow G  \models exp(g,h,\bar a), \ (here \ m = [\bar a]).
$$
To construct the formula $exp(u,v,\bar x)$ we consider several two cases,  for each of them build the corresponding formula $exp_i(u,v,\bar x)$, and then use them to build $exp(u,v,\bar x)$.

\medskip \noindent
 Case 1. Let $g \in G\smallsetminus G'$.  In Section 14.2 we described  a formula  $expnil_1(u,v,\bar x)$ of group language such that   for $g,h \in G$ and $m = [\bar a]  \in \Z^\ast $  one has 
$$
g^m = h (mod \ G') \Longleftrightarrow  G \models \ expnil_1(g,h,\bar a).
$$ 
Now put
$$
exp_1(u,v,\bar x) =  ([u,v] = 1 \wedge expnil_1(u,v,\bar x)).
$$
Then the formula $exp_1(u,v,\bar x)$ holds in $G$ on elements $g,h \in G$ and $m = [\bar a]$ if and only if $h = g^m (mod \ G') $ and $h \in C_G(g)$. Since the centralizer $C_G(g)$ is cyclic there is only one such $h$  and in this case $h = g^m$.  

\medskip \noindent
 Case 2. Let $g \in G'$.  Then for any $w \in G \smallsetminus G'$  and every $m \in \Z$ there exists $c \in G'$ such that the following equality holds (see 14.1.3) 
$$
(wg)^m  = w^mg^m[c,w].
$$
hence the elements $g$ and $g^m$  and $m = [\bar a] \in Z^\ast $ satisfies the following formula
$$
\exp_2(u,v,\bar x) = \forall w (w \in G\smallsetminus G' \to \exists c (c \in G' \wedge ((wu)^m  = w^mu^m[c,w])).
$$
Here, of course,  we use the formula $exp_2(u,v,\bar x) $  to write down the condition$(wu)^m  = w^mu^m[c,w]$. 

We claim that for given $g \in G'$ and $m = [\bar a] \in \Z^\ast $ the formula $\exp_2(g,v,\bar a)$ holds in $G$ only one one element~--- precisely on $g^m$. Indeed, let $h \in G$ be  such that for a given $m \in \Z$  for any $w \in G\smallsetminus G'$ there exists $c_1 \in G'$ such that 
$$
(wg)^m  = w^mh[c_1,w].
$$
Then $w^mg^m[c,w] = w^mh[c_1,w]$, so 
$$
 h^{-1}g^m= [c,w][c_1,w]^{-1} = [c,w][c_1^{-1},w]  = [cc_1^{-1},w].
$$
Now by Lemma \ref{le:identity} from 14.1.4 one gets  $h^{-1}g^m = 1$, so $h = g^m$, as claimed. This shows that the formula $\exp_2(u,v,\bar x)$ defines the exponentiation on $G'$. 

Finally, the formula 
$$
\exp(u,v,\bar x) = (u  \notin G' \to exp_1(u,v,\bar x))\wedge (u \in G' \to exp_2(u,v,\bar x))
$$
defines $\Z$-exponentiation on the whole group $G$.

\subsection{Interpretation of $ \Z[a_1,a_1^{-1}, \ldots,a_n,a_n^{-1}]$ in $\Z$}

\label{sec:5}

By Theorem \ref{thAKNS} every infinite f.g. integral domain is bi-interpretable with ${\mathbb Z}.$ Therefore $\mathbb Z\bar G$ is by-interpretable with $\mathbb Z$.

A $0$-interpretation of the ring of Laurent polynomials $A =  \Z[a_1,a_1^{-1}, \ldots,a_n,a_n^{-1}]$ in $\Z$ is described, for example, in \cite[Theorem 3]{KMga}

\subsection{Interpretation of $\Z \bar G$-module $G'$ in $G$}

In this section we interpret in $G$ the action of the ring  $\Z \bar G$ on $G'$.   We use notation from Section 14.1.5, so $x_1, \ldots,x_n$ is a basis of $G$, $a_1 = \bar x_1, \ldots, a_n =  \bar x_n$ is the basis of the abelianization $\bar G$ of $G$, the group ring $\Z \bar G$ can be viewed as the ring of Laurent polynomials $A = \Z[a_1,a_1^{-1}, \ldots,a_n,a_n^{-1}]$.  We denote by $A_0$ the subring $\Z[a_1, \ldots,a_n]$ of the standard polynomials in $A$. 

Below we show how to interpret the action of $\Z[a_1, \ldots,a_n]$ on $G'$ and then the action of the whole ring $\Z[a_1,a_1^{-1}, \ldots,a_n,a_n^{-1}]$ on $G'$. But first we need two  preliminary results.

 For a tuple $\bar \alpha  = (\alpha_1, \ldots,\alpha_m) \in \Z^m$, $m\leq n$, denote by $\lambda_{\bar \alpha}$ the homomorphim $\lambda_{\bar \alpha}\colon  \Z[a_1, \ldots,a_n] \to \Z[a_{m+1},\ldots ,a_n]$ such that $a_i \to \alpha_i, i = 1, \ldots,m$.  The kernel $I_{\bar \alpha}$ of $\lambda_{\bar \alpha}$ is the  ideal generated in  $\Z[a_1, \ldots,a_n] $  by  $\{a_1- \alpha_1, \ldots, a_m -\alpha_m\}$.
Notice, that for every polynomial $P = P(a_1, \ldots,a_m) \in \Z[a_1, \ldots,a_n]$ one has $\lambda_{\bar \alpha}(P) = P(\alpha_1, \ldots,\alpha_n)$, so 
$$
P(a_1, \ldots,a_m) = P(\alpha_1, \ldots,\alpha_m)  + \Sigma_{i=1}^m(a_i-\alpha_i)f_i,
$$
for some $f_i \in \Z[a_1, \ldots,a_n]$. 

\medskip \noindent
\subsubsection{Discrimination of $\Z[a_1, \ldots,a_n]$ by $\Z$.}

Let $A$ and $B$ be rings and $\Lambda$ a set of homomorphisms from $A$ into $B$.  Recall that $A$ is discriminated into $B$ by a set $\Lambda$ if for any finite subset $A_0 \subseteq A$ there is a homomorphism $\lambda \in \Lambda$ which is injective on $A_0$.  

The following result is known, but we need the argument from the proof in the sequel. 

\medskip \noindent
{\bf Claim 1.} The ring $\Z[a_1, \ldots,a_n]$ is discriminated into $\Z$ by the set of homomorphisms $\lambda_{\bar \alpha}, \bar \alpha   \in \Z^n$.  

\begin{proof} Since $\Z[a_1, \ldots,a_n]$  is an integral domain it suffices to show $\Lambda$ separates  $\Z[a_1, \ldots,a_n]$ into $\Z$, i.e., for every non-zero polynomial $Q \in \Z[a_1, \ldots,a_n]$ there exists $\lambda \in \Lambda$ such that $\lambda(Q) \neq 0$. Indeed, let $A_0 = \{P_1, \ldots, P_t\}$ with $P_i \neq P_j$ for $1 \leq j < i \leq t$. Put $Q_{ij} = P_i  - P_j$ and $Q  = \Pi _{1\leq j <i \leq t}Q_{ij}$. Then $Q \neq 0$. If for some $\lambda \in \Lambda$ $\lambda(Q) \neq 0$ then $\lambda$ is injective on $A_0$. 

 Now we prove by  induction on $n$ that $\Lambda$ separates  $\Z[a_1, \ldots,a_n]$ into $Z$.  If $P \in \Z[a_1]$ then $\lambda_{\alpha_1}$ for each sufficiently large $\alpha_1$ separates $P$ into $\Z$.   If $P \in \Z[a_1, \ldots,a_n]$ then for some $m \in \N$ 
 $$
 P = Q_m a_n^m + Q_{m-1}a_n^{m-1} + \ldots +Q_1a_n + Q_0,$$
where $Q_i \in \Z[a_1, \ldots,a_{n-1}]$ and $Q_m \neq 0$. By induction there is $\bar \beta  =(\beta_1, \ldots,\beta_{n-1}) \in \Z^{n-1}$ such that the homomorphism $\lambda_{\bar \beta}$ discriminates $Q_m$ into $\Z$. Then 
$$
\lambda_{\bar \beta}(Q_m)a_n^m + \lambda_{\bar \beta}(Q_{m-1})a_n^{m-1} + \ldots +\lambda_{\bar \beta}(Q_1)a_n + \lambda_{\bar \beta}(Q_0)
$$ 
is a non-zero polynomial in $\Z[a_n]$. Now one  can separate this polynomial into $\Z$ by sending $a_n$ to a large enough integer $\alpha_n$, as above.   This proves the claim.

\end{proof}

Denote by $(G')^{I_{\bar \alpha}}$ the submodule of the module  $G'$ obtained from $G'$ by the action of the ideal $I_{\bar \alpha}$.  $(G')^{I_{\bar \alpha}}$ is  an abelian subgroup of $G$  generated by the  set  $\{g^Q \mid g \in G', Q \in I_{\bar \alpha}\}$, hence by the   set $\{g^{a_i - \alpha_i} \mid g \in G', i = 1, \ldots, n\}$. 

\medskip \noindent
\subsubsection{Definability of $(G')^{I_{\bar \alpha}}$ in $G$.}
 
\medskip \noindent
 {\bf Claim 2.} For any basis $(x_1, \ldots,x_n)$ of $G$ and any tuple $(\alpha_1, \ldots,\alpha_m) \in \Z^m$ the  subgroup $(G')^{I_{\bar \alpha}} \leq G'$ is definable in $G$ uniformly in $(x_1, \ldots,x_n)$ and $(\alpha_1, \ldots,\alpha_m)$.  More precisely, let $\Z^\ast $ be $0$-interpretation of $\Z$ in $G$ from section 3.3. Then there is a formula $\phi(y,y_1,\ldots,y_n,\bar z_1, \ldots, \bar z_m)$ of group theory such that for any basis $(x_1, \ldots,x_n)$ of $G$ and any tuple $(\bar k_1,  \ldots, \bar k_m) \in (\Z^\ast )^n$  the formula 
$\phi(y,x_1,\ldots,x_n,\bar k_1, \ldots, \bar k_m)$ defines in $G$ the subgroup $(G')^{I_{\bar \alpha}}$, where $\alpha_i = \bar k_i \in \Z^\ast , i = 1, \ldots,m$.

Indeed, let $(x_1, \ldots,x_n)$ be a basis of $G$ and $(\alpha_1, \ldots,\alpha_m) \in \Z^m$.  The abelian subgroup $(G')^{I_{\bar \alpha}}$ of $G$ is generated by the   set $\{g^{a_i - \alpha_i} \mid g \in G', i = 1, \ldots, m\}$. 
 It follows that every element $u \in (G')^{I_{\bar \alpha}}$ can be presented as a product 
$$
u = g_1^{a_1-\alpha_1} \ldots g_m^{a_m-\alpha_m},
$$
for some $g_1, \ldots,g_m \in G'$, or , equivalently, 
in the form 
\begin{equation} \label{eq:6.1.1}
u = g_1^{x_1}g_1^{-\alpha_1} \ldots g_m^{x_m}g_m^{-\alpha_m}
\end{equation}
where $g_i^{x_i}$ is a conjugation of $g_i$ by $x_i$, and $g_i^{-\alpha_i}$ is the standard exponentiation of $g_i$ by the integer $-\alpha_i$, $i = 1, \ldots,n$.  It was shown in Section 14.3 that there exists a formula $\exp_2(u,v,\bar z)$ such that for any $g, h  \in G'$ and $\alpha = \bar m \in \Z^\ast $  the formula
$\exp_2(g,h,\bar m)$ holds in $G$ if and only if $g = h^{\alpha}$.    Using formula $\exp_2(u,v,\bar z)$ and definability of the commutant $G'$ in $G$ (see 3.1) one can write down the condition (\ref{eq:6.1.1}) by a group theory formula uniformly in $(x_1, \ldots,x_n)$ and $(\alpha_1, \ldots,\alpha_m)$, as claimed.

\medskip \noindent
\subsubsection{Interpretation of the action of $\Z[a_1, \ldots,a_n]$ on $G'$.}

\begin{lemma} \label{le:6.3}
Let $g, h \in G'$ and $P \in \Z[a_1, \ldots,a_m]$, $m\leq n$. Then $g^P = h$ if and only if the following condition holds:
\begin{equation} \label{eq:g-h-P}
\forall \alpha_1, \ldots \alpha_m \in \Z(g^{P(\alpha_1, \ldots,\alpha_m)} = h \ mod \ (G')^{I_{\bar \alpha}}).
\end{equation}
\end{lemma}
\begin{proof}
If $g^p = h$ then the condition (\ref{eq:g-h-P}) holds since 
$$
P(a_1, \ldots,a_m) = P(\alpha_1, \ldots,\alpha_m)  \ mod  \  I_{\bar \alpha}.
$$

Conversely, suppose for  $h \in G'$ the condition (\ref{eq:g-h-P}) holds. We need to show that $g^P = h$. 

We will prove this by induction on $m, 2\leq m\leq n$. 
  Suppose first that $P=P(a_1,a_2)$. By Proposition \ref{nf}, $G'$ is a free module over $\mathbb Z[a_1,a_1^{-1},a_2, a_2^{-1}]$ with the basis $\{[x_i,x_j]^{a_3^{\delta _3}\ldots a_{j}^{\delta _j}}\}$ for all $1\leq i<j\leq n$, $\delta _3,\ldots,\delta _j\in\Z^{j-2}.$ 
  Let
$$
g = \Pi [x_i,x_j]^{a_3^{\delta _3}\ldots a_j^{\delta _j}\gamma_{ij\bar\delta }(a_1, a_2)}, 
$$
$$
 h=  \Pi[x_i,x_j]^{a_3^{\delta _3}\ldots a_j^{\delta _j}\beta_{ij\bar\delta}(a_1, a_2)}.
$$

Then the condition (\ref{eq:g-h-P}) becomes

$$ [x_i,x_j]^{a_3^{\delta _3}\ldots a_j^{\delta _j}\gamma_{ij\bar\delta }(a_1, a_2)P(\alpha _1,\alpha _2)}=$$ $$ [x_i,x_j]^{a_3^{\delta _3}\ldots a_j^{\delta _j}\beta_{ij\bar\delta }(a_1, a_2)}+
[x_i,x_j]^{a_3^{\delta _3}\ldots a_j^{\delta _j}(a_1-\alpha _1)f_1+(a_2-\alpha _2)f_2},$$
for some $f_1, f_2 \in \Z[a_1,a_2]$. Note that both the left-hand and the right-hand  sides of these equalities are in the normal form, so 
the following polynomials are equal:

$${\gamma_{ij\ldots k_s}(a_1, a_2)P(\alpha _1,\alpha _2)={\beta_{ij\ldots k_s}(a_1, a_2)}+
{(a_1-\alpha _1)f_1+(a_2-\alpha _2)f_2}}$$

Evaluating these polynomials at $a_1 = \alpha_1, a_2 = \alpha_2$ one gets 
$$ \gamma_{ij\ldots k_s}(\alpha _1, \alpha_2)P(\alpha _1,\alpha _2)={\beta_{ij\ldots k_s}(\alpha_1, \alpha_2)}$$

 for any $\alpha_1, \alpha_2 \in \Z$.  Hence (see Section 14.5.1) 
 $$
 \gamma_{ij\ldots k_s}(a_1, a_2)_ P(a_1, a_2)= \beta_{ij\ldots k_s}(a_1, a_2),
 $$
so $g^P = h$, as claimed.

Now assume that the statement of Lemma 2 is true for $m-1$ and we will prove it for $m$. Condition (\ref{eq:g-h-P}) for $m$  implies that 

$$
g^{P(\alpha _1,...,\alpha _{m-1},\beta)}=hh_{\beta}^{(a_m-\beta)}\ mod \ (G')^{I_{\bar \alpha}},$$
where $h_{\beta}\in (G')^{I_{\beta}}.$ 

Note that for a given $\alpha_m \in \mathbb{Z}$ one has 
$$
g^{P(a_1,...,a_m)} = g^{P(a_1,...,a_{m-1},\alpha_m)}g^{(a_m-\alpha_m)f}
$$
for some $f \in \Z[a_1,\ldots , a_n]$. By induction $g^{P(a_1,...,a_{m-1},\beta)} = hh_1$ , so 
$$
g^{P(a_1,...,a_m)}h^{-1} = (g^{f}h_{\beta})^{(a_m-\beta)}.$$
Suppose we consider the equality above for various pair-wise distinct values of $\alpha_m$, say $\beta_1, \beta_2, \ldots$. Then
$$
g^{P(a_1,...,a_m)}h^{-1} = (g^{f_1}h_{\beta _1})^{(a_m-\beta_1)} = (g^{f_2}h_{\beta _2})^{(a_m-\beta_2)}= \ldots 
$$

It follows from Proposition \ref{nf} that $G'$ can be considered as a free module over $\Z[a_m,a_m^{-1}]$ with an infinite basis. Let $b$ be an aribitrary basis element.  Suppose $g^{P(a_1,...,a_m)}h^{-1}$ contains $b^{r(a_m,a_m^{-1})}$ and each $(g^{f_i}h_{\beta _i})$ contains $b^{r_i(a_m,a_m^{-1})}.$ Then

$$r(a_m,a_m^{-1})=r_i(a_m,a_m^{-1})(a_m-\beta _i)$$
for all $i\in\N.$ Therefore $r(a_m,a_m^{-1})=0.$ 


This shows that $g^{P(a_1,...,a_m)} = h$, as claimed.

Every element $Q\in  \Z[a_1,a_1^{-1}, \ldots,a_n,a_n^{-1}]$ can be written as
$$Q=P(a_1,...,a_m)(a_1^{k_1}\ldots a_m^{k_m})^{-1}$$ for some  $k_1,\ldots ,k_m\geq 0$.
Therefore $g^Q=h$ if and only if $g^{P(a_1,...,a_m)} = h^{(a_1^{k_1}\ldots a_m^{k_m})}.$
This gives the interpretation of the action of $\Z[a_1,a_1^{-1}, \ldots,a_n,a_n^{-1}]$ on $G'$.

\end{proof}

\subsection{Proof of Theorem \ref{metab}}

  Using the Magnus embedding 
\begin{equation*}
x_i \rightarrow
\begin{pmatrix}
a_i & 0  \\
t_i& 1 
\end{pmatrix}
,\end{equation*}
where $a_1,\ldots ,a_n$ are the generators of the free abelian group  and $\{t_1,\ldots ,t_n\}$ the base of the free  $\mathbb Z \bar G$ module $T$. Then $G'$ has 
the structure of the the submodule of $T$ generated by $t_i(x_j-1)-t_j(x_i-1), i\neq j.$ Denote this submodule by $G'_{\Z \bar G}$. 

The module $G'_{\Z \bar G}$ is a two sorted structure $(\Z \bar G, G',\delta)$ where $\delta$ is the predicate describing the action of $\Z \bar G$ on $G'$. We will show  that $G'_{\Z \bar G}$ is bi-interpretable with $\mathbb  Z$.

By Theorem \ref{thAKNS} every infinite f.g. integral domain is bi-interpretable with ${\mathbb Z}.$ Therefore $\mathbb Z\bar G$ is by-interpretable with $\mathbb Z$.  Every f.g. free  $\mathbb Z\bar G$-module is bi-interpretable with $\Z \bar G$ (with some $\Z \bar G$-basis as constants), and, therefore, it is bi-interpretable with  $\mathbb Z$. Since $\Z$ is bi-interpretable with $T$ and with the submodule $B$  generated by $[u_1,u_2]$, we can define a map 
$T\rightarrow _{\phi}\Z\rightarrow _{\psi} B\leq T$. The graph of the map $\psi\circ\phi$ is definable in $T$. The restriction of the map $\phi$ to $G'_{\Z \bar G}$, $\phi|_{G'_{\Z \bar G}}$  is the interpretation of $G'_{\Z \bar G}$ in $\Z$  and $\psi$ is the interpretation of $\Z$ in $G'_{\Z \bar G}$. Te graph of  $\psi\circ\phi|_{G'_{\Z \bar G}}$ is definable and therefore $G'_{\Z \bar G}$ is bi-interpretable with $\Z$. Denote by $\alpha$ the bijection  from $G'_{\Z \bar G}$ to $\Z$.

We have the interpretation of $G$ on the set $\Z^{m+1}$ and, therefore, in $\Z$ defined by $$\tau (u_1^{m_1}\ldots u_n^{m_n}h)=(m_1,\ldots ,m_n,\alpha (h)).$$ 
We have already shown that the exponentiation and the module structure of $G'_{\Z \bar G}$ are interpretable in $G$.

Now, given an $n+1$ tuple $(m_1,...,m_n, k)$ of integers we reconstruct the element of the group as $u_1^{m_1}\ldots u_n^{m_n}\alpha ^{-1}(k).$

This proves the theorem. 



\subsection{Applications and open problems}

Let $M_n$ be a free metabelian group of finite rank $n \geq 2$. Since $M_n$ is bi-interpretable with $\MN$ it is rich and  it satisfies all the properties of rich groups mentioned in Section \ref{sec:qfa}. Here we list only those ones that answer to some open questions.  

\begin{theorem}
For any finite $n \geq 2$ the free metabelian group $M_n$ is prime, atomic, homogeneous, and QFA.
\end{theorem}
\begin{proof}
By Theorem \ref{metab} $M_n$ is bi-interpretable with $\MN$, hence by Lemma \ref{le:prime} $M_n$ is prime. Hence it is atomic and homogeneous (see Section \ref{sec:qfa}). 
The word problem is decidable in $M_n$, hence it is arithmetical. Now by Theorem \ref{th:QFA}  the group $G$ is QFA.
\end{proof}

The following results answers in the positive to an open question  in Kourovka notebook posted by E. Timoshenko.

\atreyer{There is missed something in this theorem}
\begin{theorem}
Then the set of bases of the group $M_n$, $n \geq 2$,  is $0$-definable in $M_n$. 
\end{theorem}
\begin{proof}
Let $a = (a_1, \ldots , a_n)$ be a basis of $M_n$. AS before, consider a computable enumeration 
$$
w_1, w_2, \ldots, w_m, \ldots
$$
of all words in the alphabet $X^{\pm 1}$, where $X= \{x_1, \ldots,x_n\}$. Since the word problem in $M_n$ with respect to the basis $a$ is decidable, the set 
$$
W(M_n) = \{n \in \MN \mid w_n(a_1, \ldots, a_n) = 1 \ in \ G\}
$$
is computably enumerable, hence  arithmetic.  Consider the following formula in WSOL 
$$
W_n(X) =  \forall y (y = 1 \leftrightarrow \bigvee_{i \in W(M_n)} h = w_i(x_1, \ldots,x_n)).
$$
It states that the subgroup generated by $X$ has the same word problem as the subgroup generated by $a_1, \ldots,a_n$.
Let 
$$
Gen_n(X) = \forall y \bigvee_{i \in \MN} y= w_i(x_1, \ldots,x_n).
$$ 
Then $Gen_n(X)$ states that $X$ generates $M_n$. 
Now put
$$B_n(X) = W_n(X) \wedge Gen_n(X).$$

Clearly, $M_n \models B_n(a_1, \ldots,a_n)$ and if $M_n \models B_n(b_1, \ldots,b_n)$ for a tuple $b = (b_1, \ldots,b_n) $ in $M_n$ then the map $a_1 \to b_1, \ldots,a_n \to b_n$ extends to an automorphism $\alpha\colon M_n \to M_n$ of $M_n$. Since $\alpha(a) = b$ it follows that $b$ is a base of $M_n$. The converse is also true, i.e., if $b = (b_1, \ldots,b_n)$ is a base of $M_n$ then there is an automorphism $\alpha$ of $M_n$ such that $\alpha(a) = b$. Hence $b$ has the  same word problem in $M_n$ as $a$ so it satisfies the formulas $W_n$ and $Gen_n$. This finishes the proof.

\end{proof}

\printindex
\end{document}